\documentclass{article}
\usepackage{amssymb}
\usepackage{titlesec}
\usepackage{graphicx}
\usepackage{float}
\usepackage{mathtools}
\usepackage{afterpage}
\usepackage{amsmath}
\usepackage[export]{adjustbox}
\usepackage{graphicx}
\usepackage{subcaption}
\usepackage{dsfont}
\usepackage{mathrsfs}
\usepackage{enumitem}
\DeclarePairedDelimiter{\floor}{\lfloor}{\rfloor}
\DeclarePairedDelimiter{\ceil}{\lceil}{\rceil}
\usepackage{color}
\usepackage[english]{babel}
\titleformat{\chapter}[display]
  {\normalfont\LARGE\bfseries}
  {\titleline{}\vspace{5pt}\titleline{}\vspace{1pt}%
  \MakeUppercase{\chaptertitlename} \thechapter}
  {1pc}
  {\titleline{}\vspace{0.5pc}} 
\usepackage{mathtools}
\DeclarePairedDelimiter\abs{\lvert}{\rvert}

\makeatletter

\renewcommand\section{\@startsection {section}{1}{\z@}%
                               {-3.5ex \@plus -1ex \@minus -.2ex}%
                               {2.3ex \@plus.2ex}%
                               {\normalfont\large\bfseries}}
\renewcommand\subsection{\@startsection{subsection}{2}{\z@}%
                                 {-3.25ex\@plus -1ex \@minus -.2ex}%
                                 {1.5ex \@plus .2ex}%
                                 {\normalfont\bfseries}}

\makeatother

\usepackage[english]{babel} 
\usepackage{amsmath,amsfonts,amsthm} 

\usepackage{bold-extra}

\usepackage{fancyhdr}        

\usepackage[hidelinks]{hyperref}
\numberwithin{equation}{section}       
\numberwithin{figure}{section}         
\numberwithin{table}{section}          

\newtheorem{theorem}{Theorem}[section]
\newtheorem{prop}[theorem]{Proposition}
\newtheorem{corollary}[theorem]{Corollary}
\newtheorem{lemma}[theorem]{Lemma}
\theoremstyle{remark}
\newtheorem{remark}{Remark}[section]
\theoremstyle{definition}

\newtheorem{assumption}{Assumption}
\theoremstyle{definition}

\usepackage[normalem]{ulem}
\newcommand{\po}{\left(}
\newcommand{\pf}{\right)}

\newcommand{\R}{\mathbb R}

\newcommand{\N}{\mathbb N} 

\newcommand{\na}{\nabla}

\newcommand{\bG}{\mathbf{G}} 
\newcommand{\bW}{\mathbf{W}}

\graphicspath{{img/}}

\title{
\normalfont \normalsize   
\huge Reflection coupling for unadjusted generalized Hamiltonian Monte Carlo in the nonconvex stochastic gradient case 
}
\author{Martin Chak \and Pierre Monmarch\'e}
\date{\normalsize\today}

\begin{document}
\maketitle

\begin{abstract}
Contraction in Wasserstein 1-distance with explicit rates is established for generalized Hamiltonian Monte
Carlo with stochastic gradients under possibly nonconvex conditions. The algorithms considered include splitting schemes of kinetic Langevin diffusion commonly used in molecular dynamics simulations. To accommodate the degenerate noise structure corresponding to inertia existing in the chain, a characteristically discrete-in-time coupling and contraction proof is devised. 
As consequence, quantitative Gaussian concentration bounds are provided for empirical averages.
Convergence in Wasserstein 2-distance and total variation are also given, together with numerical bias estimates.
\end{abstract}

\section{Introduction}\label{sec:intro}
In this paper, we study a class of algorithms that spans on one end from Hamiltonian Monte Carlo (HMC)~\cite{MR4133372} to splitting schemes of kinetic Langevin dynamics~\cite{MR3040887} on the other end, as previously considered in~\cite{camrud2023second,gouraud2023hmc,monmarche2023entropic} by the second author and collaborators. This family of Markov chains and variants include widely used methods in Markov Chain Monte Carlo (MCMC) methods and are the topic of an extensive literature due to their interest in machine learning and molecular dynamics. In particular, kinetic Langevin dynamics is an important physical model for particles interacting with a large population of lighter particles~\cite[Section~6.3.3]{MR3362507}. Its numerical approximations by splitting scheme, which we study below as particular cases, are widely used to calculate both expectations with respect to canonical measures and dynamical properties of the particle trajectories. Notably, the latter quantities~(see e.g.~\cite[Section~21-8]{mcquarrie76a},~\cite{water},~\cite{MR3463433},~\cite{nature}) are lost if algorithms designed solely to sample from target measures are used~\cite[Remark~2.5]{MR2681239}. The goal of this work is to provide quantitative nonasymptotic analysis given deterministic initial values on the ergodic properties for this class of algorithms, where the dynamical properties are physically relevant. Note that the dynamical properties themselves are not the focus of this article. The analysis is based on the nonconvex setting, which is ubiquitous and prominent in molecular dynamics simulations. Works on the aforementioned splitting schemes in the convex setting have also appeared recently in~\cite{lc23,MR4309974,gouraud2023hmc}. A comprehensive comparison with previous work in the nonconvex case is given below.

Although our convergence results are motivated by molecular dynamics applications, the algorithms studied are also closely related to those in computational statistics~\cite{cheng2020sharp,pmlr-v75-cheng18a,MR4278799}, where stochastic gradients~\cite{10.5555/3104482.3104568,pmlr-v32-cheni14,leimkuhler2023contraction} can alleviate the computational cost of the sampling procedure. Given these considerations, our analysis below allows for the possibility of stochastic gradients. 
To be more precise, let~$\eta\in[0,1)$,~$h>0$,~$K\in\mathbb{N}\setminus\{0\}$,~$\Theta$ be a measurable space and let~$b:\mathbb{R}^d\times\Theta\rightarrow \mathbb{R}^d$ be measurable. Given an initial position-velocity state~$(x,v)\in\mathbb{R}^{2d}$, a single iteration of the algorithm considered is given by the steps:
\begin{enumerate}
\item Draw a (partial) velocity refreshment~$G\sim \mathcal{N}(0,I_d)$ (the standard $d$-dimensional Gaussian distribution).
\item Update velocity with the refreshment by~$v\leftarrow \eta v + \sqrt{1-\eta^2}G$.
\item \label{s3} Choose~$\theta,\theta'\in\Theta$, then run a stochastic Hamiltonian step by the velocity Verlet integrator
\begin{subequations}\label{vvi}
\begin{align}
v&\leftarrow v-(h/2)b(x,\theta),\\
x&\leftarrow x + hv,\\
v&\leftarrow v-(h/2)b(x,\theta').
\end{align}
\end{subequations}
\item Repeat step~\ref{s3}.~$K-1$ more times.
\end{enumerate}
Iterating such transitions produces a trajectory~$(x_n,v_n)_{n\in\N}$, 
whose empirical distribution is then used to approximate a suitable target distribution, as usual in MCMC methods. The resulting algorithm is called the stochastic gradient (unadjusted) generalized Hamiltonian Monte Carlo (SGgHMC) chain with parameters~$(K,h,\eta)$. We call~$h$ the stepsize,~$\eta$ the damping parameter and~$T:=Kh$ the integration time. 
In the Hamiltonian step above, the way~$\theta,\theta'$ are chosen is not specified. We keep in mind the simple case where they are independent random variables drawn according to some probability distribution over~$\Theta$, but in practice this may not be exactly the case. Unless explicitly mentioned, we do not assume the way~$\theta,\theta'$ are chosen (they can be considered as deterministic, i.e. we mostly work conditionally to their values), however we assume that the Gaussian variables used in the velocity refreshment steps are independent from these variables. In any case, as indicated by the term \emph{stochastic gradient}, 
this setting is motivated by cases 
where~$b(x,\theta),b(x,\theta')$ are random estimators of~$-\nabla \log \mu(x)$ for the density~$\mu$ of a target probability measure. In addition, \emph{unadjusted} refers to the fact there is no Metropolis-Hastings accept/reject procedure (see e.g.~\cite{bourabee2023mixing}) in the algorithm to enforce the exact invariance for a suitable target distribution. When~$\eta=0$ (i.e. the velocity is fully refreshed after each Hamiltonian trajectory), SGgHMC aligns with classical HMC, so that \emph{generalized} HMC refers to the fact~$\eta$ can take any value in~$[0,1)$. In particular, when~$K=1$ with~$\eta=e^{-\bar{\gamma}h}$ for fixed~$\bar{\gamma}>0$, SGgHMC corresponds to a splitting scheme for the kinetic Langevin diffusion process (as studied in \cite{MR3040887,lc23,MR4309974}). 
SGgHMC has been defined above with the velocity Verlet integrator, which is well-suited to the motivations given in the first paragraph, but similar results for a randomized midpoint version~\cite{bourabee2022unadjusted} will be proved below as well. The latter is of special interest due to its advantage in terms of dimension dependence in asymptotic bias estimates. It is included in our analysis to demonstrate the flexibility of our approach.

Our main contribution (Theorem~\ref{intma}) is a quantitative Wasserstein contraction for the law of SGgHMC under Assumption~\ref{A1} on~$b$, which allow for nonlogconcave target densities. The proof is based on a coupling construction that turns out to interpolate precisely between those of~\cite{MR4133372} and~\cite{MR3980913} for HMC and continuous time Langevin dynamics respectively. From this, we provide non-asymptotic Gaussian concentration~\cite{MR1849347} for empirical averages of the output from SGgHMC (Corollary~\ref{gaucor}) and explicit bounds for the numerical bias (with respect to target measure) induced by the numerical integration (Corollary~\ref{invd}) and the stochastic gradient approximation (Proposition~\ref{invd2}). Together, these results yield non-asymptotic confidence intervals for the estimator of the expectation of Lipschitz functions using SGgHMC. We also state a Wasserstein-to-entropy regularization (Theorem~\ref{prop:TV/W1}). When combined with the Wasserstein contraction, this provides a convergence rate in terms of relative entropy and, as we discuss below, is a crucial point in the perspective of the analysis of the adjusted version of the SGgHMC algorithm.

Let us discuss these contributions in view of two series of work, namely~\cite{camrud2023second,gouraud2023hmc,monmarche2023entropic} on the one hand, which are concerned with the same family of generalized HMC chains, and~\cite{MR4133372,cheng2020sharp,MR3980913} on the other hand which, similar to the present work, use reflection coupling arguments for kinetic processes.  
\begin{itemize}
    \item Our study considering the whole family of gHMC algorithms, from HMC to Langevin, stems from~\cite{gouraud2023hmc}, where the setting is convex (namely the target distribution is logconcave, which would correspond here to~$R=0$ in Assumption~\ref{A1}). In that work, the observation is made among others that, in the Gaussian case for instance, the optimal choice of parameters is neither~$\eta=0$ nor~$K=1$ but rather~$\eta=e^{-\bar{\gamma}T}$ for a fixed~$\bar{\gamma}$ and a fixed~$T=Kh$ (with small~$h$). This corresponds to gHMC with inertia, which is ``in the middle" of the gHMC family. The technique of~\cite{gouraud2023hmc} is based on synchronous coupling and only works in the convex case with a sufficiently high friction parameter~$\bar{\gamma}$ (see also e.g.~\cite{pmlr-v75-cheng18a,leimkuhler2023contraction,MR4278799,MR4309974} for the convex case with synchronous coupling). The nonconvex case is considered in~\cite{camrud2023second,monmarche2023entropic} (without stochastic gradient approximation) using functional inequality entropy methods (see also~\cite{MR4278799} for similar methods for stochastic Euler schemes of the kinetic Langevin diffusion). The pros and cons of such methods by comparison with direct coupling methods are discussed in~\cite{monmarche2023entropic} and in the following. The interest of entropy methods is that in certain cases, they provide sharper convergence rates in relative entropy than direct coupling ones. However, they rely on some explicit computations related to the target density and can thus be limited. Moreover, they are shown 
    for deterministic gradients only, that is, when there exists~$U$ with~$b(\cdot,\theta)=\nabla U$ for all~$\theta$. 
    On the other hand, coupling methods are more flexible. They apply without difficulty to numerical or stochastic approximations or in non-equilibrium cases (i.e. when~$b(\cdot,\theta)$ is not the estimator of a gradient, as e.g. in \cite{Iacobucci,Ramil}, though we do assume that~$\nabla b(\cdot,\theta)$ is symmetric in order to improve our assumptions on~$T$, see the beginning of Section~\ref{convex}). In addition, as stated in Corollary~\ref{gaucor}, our approach yields non-asymptotic concentration inequalities for empirical averages over trajectories with deterministic starting points, which is less clear with the functional inequality approach. 
    Furthermore, as we discuss below, the present coupling approach can then be used in the analysis of the adjusted version of the algorithm, which is again less clear with the functional inequality approach. Note that the adjusted version is also relevant for molecular dynamics applications~\cite{MR2583309}. 
    \item The seminal work~\cite{MR3980913}, which deals with continuous-time kinetic Langevin dynamics (namely the limit of SGgHMC as~$h$ vanishes in the case~$K=1$,~$\eta= e^{-\bar{\gamma}h}$ for fixed~$\bar{\gamma}>0$) has inspired many variations. Among them are~\cite{MR4133372,10.1214/23-EJP970} for the classical HMC~($\eta=0$) and~\cite{cheng2020sharp} for a stochastic Euler scheme of the kinetic Langevin diffusion. The latter together with~\cite{camrud2023second,MR4278799} appear to be the only work in the literature beside ours to show quantitative convergence for discretizations of ergodic stochastic differential equations (SDEs) with degenerate noise and without logconcavity assumptions. 
     The works~\cite{camrud2023second,MR4278799} are, as mentioned, based on functional analytical techniques and although the work~\cite{cheng2020sharp} deals with an implementable scheme, it essentially discretizes only the gradient term. Consequently, the coupling construction in~\cite{cheng2020sharp} is characteristically continuous in time and similar in nature to that presented in~\cite{MR3980913}. 
    Our work furnishes a coupling 
    in-between~\cite{MR3980913} and~\cite{MR4133372}, which 
    by contrast to \cite{cheng2020sharp}, is essentially discrete in time. The noise degeneracy arising from the kinetic process is such that previous discrete-time coupling constructions cannot lead to contraction in general; the couplings in for example~\cite{debortoli2020convergence,MR4133372,MR3933205,MR4132634,durmus2018highdimensional} all deal with noise in a way that is reminiscent of nondegenerate noise terms in SDEs. Moreover, in the works~\cite{MR4133372,10.1214/23-EJP970} on HMC, the restriction on the integration time~$T$ forces HMC to behave similarly to ULA. 
Thus our coupling and contraction proof fills a mathematical gap in the literature by providing detailed analysis for chains where inertia persists under restrictions on~$T$. 
    In particular, our approach yields concentration inequalities (Corollary~\ref{gaucor}) that do not follow as corollary to the work of~\cite{cheng2020sharp}, since a complexity bound for Wasserstein error is given in that work rather than contraction as in~\eqref{intma1} (though our coupling method can be expected to work for stochastic Euler schemes). 
    In addition, SGgHMC is a splitting scheme for the kinetic Langevin diffusion in the case~$K=1$ with~$\eta= e^{-\bar{\gamma}h}$ for fixed~$\bar{\gamma}>0$, which aligns with its use for particle simulations as described at the beginning of this section. 
    In the non-stochastic case where~$b(\cdot ,\theta)=\nabla U$ for all~$\theta$ and under sufficient regularity conditions on~$U$, these algorithms can be shown to have a numerical bias of order $h^2$. This is to be compared to the order~$h$ for stochastic Euler schemes as in~\cite{cheng2020sharp}. 
    \item Lastly, let us emphasize that we devise a natural contraction proof, which yields favorable constants in the exponent of the contraction rates when compared to previous rates given for the particular case~$\eta=0$ (HMC and ULA). 
More specifically, let us compare our Wasserstein 1-distance contraction rate for~$\eta=0$ with those in~\cite{MR3933205,MR4132634,MR4133372,10.1214/23-EJP970} in the nonconvex case. 
In~\cite{MR3933205,MR4132634}, applications are given for the unadjusted Langevin algorithm (ULA), which is appears here as~$\eta=0$,~$K=1$,~$\Theta=\{\theta\}$. Their convergence rates scale like~$e^{-c\mathcal{R}^2}$ for a \textit{nonconvexity radius}~$\mathcal{R}$ and some constants~$c>0$. The parameter~$\mathcal{R}$ will also appear in our context below in Remark~\ref{Rcom} and our contraction rate scales similarly as~$e^{-c'\mathcal{R}^2}$ for a constant~$c'>0$. On the other hand, the authors of~\cite{MR4133372,10.1214/23-EJP970} focus on HMC where~$K\gg1$. Their results include ULA, but their contraction rates scale like~$e^{-c''\mathcal{R}/T}$ for some constants~$c''>0$. The inclusion of~$T^{-1}$ in the rate means that it becomes exponentially poor for small stepsize in the case of ULA, which is necessary to obtain the correct bias estimates to the true target distribution. Thus their results are best suited for HMC, where good bias estimates may be obtained by taking~$h$ small enough, but~$T$ relatively large. However, the latter exponents~$c''$ are (very) noticeably better than~$c$ even after taking into account the restrictions on~$T$. 
    Our result for~$\eta=0$ both removes the~$1/T$ dependence in the exponent within the contraction rate and allows~$c'$ to be similar to~$c''$, see Remark~\ref{rrcom}. This arises from a more direct approach to estimate the perturbations in the usual concave function values. 
    The ideas are encapsulated by~\eqref{uexp} and Lemma~\ref{lemexes}. In addition, inspired by~\cite{lc23}, we remove in the upper bound on~$T$ the dependence on the convexity constant at infinity that is present in all these previous works. In particular, we show that the technique in~\cite{lc23}, which relied heavily on strong convexity, can be used to analyze a nonconvex setting. 
\end{itemize}

From a methodological standpoint, there is a vast literature on strategies for sampling from distributions with nonlogconcave densities. Two examples of prominent approaches are parallel tempering and adaptive biasing methods, see~\cite{henin23,MR4152629} for reviews. The methods presented in these reviews are expected to perform significantly better than those studied presently, but many of the approaches make direct use of, or are otherwise linked to, SGgHMC or its continuous time counterparts. For example in~\cite{ge2020simulated}, the local exploration kernel in parallel tempering is naturally chosen to be that of ULA, which as mentioned above is SGgHMC with~$\eta=0$,~$K=1$ and~$\Theta=\{\theta\}$. Correspondingly, quantitative convergence results of those methods comparable to those presented here in Section~\ref{mr} are far from commonly available. 
In addition, methodologies that make direct use of reflection couplings are of relevance. For example in the context of~\cite{MR3949304}, the coupling provided in Section~\ref{secsg} below reveals a deliberate way to simulate pairs of chains based on underdamped Langevin diffusions for unbiased estimation. This use of underdamped Langevin dynamics has in fact already been the topic of investigation in~\cite{MR4673688}.


\bigskip

The rest of the paper is organized as follows. Section~\ref{mr} states the main results of the paper. Sections~\ref{ncsec},~\ref{convex} and~\ref{global} are devoted to proving the main Wasserstein contraction Theorem~\ref{intma}. More specifically, in Section~\ref{ncsec}, the reflection coupling to be used within a bounded region of state space is introduced, then it is used to prove contraction of a suitable concave function. In Section~\ref{convex}, the behaviour of a certain modified Euclidean norm is studied under synchronous and reflection couplings. The results from these two sections are gathered in Section~\ref{global} to prove contraction in a suitable semimetric after one iteration of SGgHMC. In Section~\ref{conseqs}, Theorem~\ref{intma} is proven using the results in the previous sections, then the consequences in terms of empirical averages are presented.

\paragraph*{Notation.}
For~$x\in\R^d$ and~$\bar{R}>0$,~$|x|$ stands for the Euclidean norm and~$B_{\bar{R}} = \{y\in\R^d,|y|\leq \bar{R}\}$. The transpose of a matrix $A$ is denoted by $A^{\!\top}$, and vectors of $\R^d$ are seen as column matrices. The identity matrix is denoted~$I_d$ or~$I$ in cases where the dimension is obvious from the context. The functions~$\varphi_{0,1},\Phi$ denote the probability density and cumulative distribution function respectively for the standard normal distribution. The notation~$\mathcal{U}(0,1)$ denotes the uniform distribution on~$(0,1)$ and~$\mathcal{N}(0,I_d)$ denotes the standard normal on~$\mathbb{R}^d$ as above. 
For a function~$f:\mathbb{R}^d\rightarrow\mathbb{R}$ and metric~$\tilde{d}$ on~$\mathbb{R}^{2d}$, the notation
\begin{equation*}
\|f\|_{\textrm{Lip}(\tilde{d})}:=\sup_{((x,v),(y,w))\in\mathbb{R}^{2d}\times\mathbb{R}^{2d}}\abs{f(x)-f(y)}/\tilde{d}((x,v),(y,w))
\end{equation*}
is used for the Lipschitz norm of~$f$ w.r.t.~$\tilde{d}$; similarly for a metric on~$\mathbb{R}^d$. Let~$\|\cdot\|_F$ denote the Frobenius norm,~$\|\cdot\|_{\textrm{TV}}$ denote the total variation norm and for distributions~$\nu,\mu$ on~$\mathbb{R}^{2d}$, let~$\mathrm{Ent}(\nu|\mu)=\int_{\mathbb{R}^{2d}}\log(d\nu/d\mu)d\nu$ if~$\nu\ll\mu$ and~$\mathrm{Ent}(\nu|\mu)=\infty$ otherwise denote the relative entropy of $\nu$ with respect to $\mu$.

\ 

\section{Main results}\label{mr}

The results stated in this section apply for both the velocity Verlet integrator~\eqref{vvi} and a randomized midpoint version as studied in~\cite{bourabee2022unadjusted}. The latter will also be defined at the beginning of Section~\ref{ncsec}.
\paragraph*{Main assumption and twisted metric.} To state our main result, let~$m,L,R>0$ and consider the following assumption.
\begin{assumption}\label{A1}
For every~$\theta\in\Theta$,~$b(\cdot,\theta)\in C^1(\mathbb{R}^d,\mathbb{R}^d)$ and for every~$x\in\mathbb{R}^d$,~$\nabla b(x,\theta)$ is symmetric. 
Moreover, 
it holds that
\begin{equation}\label{A10}
u^{\!\top}\nabla\!_x b(x,\theta) u \geq m\abs{u}^2 \qquad \forall u\in\mathbb{R}^d,x\in\mathbb{R}^d\setminus B_R,\theta\in\Theta
\end{equation}
and
\begin{equation}\label{A1a}
\abs{\nabla\!_x b(x,\theta) u}\leq L \abs{u} \qquad \forall u,x\in\mathbb{R}^d, \theta\in\Theta.
\end{equation}
\end{assumption}
As noted in~\cite{monmarche2023entropic,camrud2023second}, 
we may assume~$L=1$ with no loss of generality by a rescaling. However, we keep track of~$L$ throughout for ease of comparison with previous works, unless explicitly stated. For example, the restriction on the value of the friction imposed in Section~\ref{convex} aligns with that stated in~\cite[Proposition~4]{monmarche2021sure}, which is in terms of~$L$. 

\begin{remark}\label{Rcom}
For any~$x,y\in\mathbb{R}^d$, denoting~$x_t = x + t(y-x)$ for all~$t\in[0,1]$ and~$t_R = \{ t\in [0,1]:x+t(y-x)\in B_R \}$, Assumption~\ref{A1} implies that if~$x$ and~$y$ satisfy~$\abs{x-y}\geq R':= 4R(1+L/m)$, then it holds for any~$u\in\mathbb{R}^d$ that
\begin{align}
u^{\!\top}\int_0^1\nabla\!_xb(x_t,\theta)dt u &= u^{\!\top}\bigg(\int_{t_R}\nabla\!_xb(x_t,\theta)dt + \int_{[0,1]\setminus t_R}\nabla\!_xb(x_t,\theta)dt\bigg) u \nonumber\\
&\geq - \frac{2R}{R'}L\abs{u}^2 + \bigg(1-\frac{2R}{R'}\bigg)m\abs{u}^2 \nonumber\\
&= \frac{m}{2}\abs{u}^2.\label{A1b}
\end{align}
Assumption~\ref{A1} is thus slightly stronger than~\cite[(A3)]{MR4133372},~\cite[(1.5)]{MR4132634},~\cite[Assumption~2.3]{10.1214/23-EJP970} and~\cite[(C2)]{MR3933205} in the nonconvexity aspect. Note that~$R'$ is the appropriate value to compare with the nonconvexity radii that appear in these articles, rather than~$R$. It should be emphasized that the coupling in Section~\ref{ncsec} also yields contraction results in these slightly weaker settings and that this fact has already been utilized in~\cite{schuh2024c} after the first preprint of the present manuscript was made public. However, condition~\eqref{A1a} allows the dependence on~$m$ in the contraction rate to improve via the arguments in Section~\ref{convex}. On the other hand, in the case where there exists~$U$ such that~$b(\cdot,\theta)=\nabla U$ for all~$\theta$, condition~\eqref{A10} is also studied as a main analytical assumption in~\cite{MR4025861},~\cite[Assumption~3]{MR4704569} and~\cite[H2 (supplemental)]{NEURIPS2022_21c86d5b} with discussions given in all of these references. The smoothness condition~\eqref{A1a} is even more common in the literature.
\end{remark}


Next, we introduce the twisted Euclidean metric~$\hat{d}$ on which our Wasserstein contraction Theorem~\ref{intma} is based. Let~$\hat{d}:\mathbb{R}^{2d}\times\mathbb{R}^{2d}\rightarrow [0,\infty)$ be given by
\begin{equation}\label{tmet}
\hat{d}((x,v),(y,w))= \hat{\alpha}\abs{x-y}+\abs{x-y + \hat{\gamma}^{-1}(v-w)}, 
\end{equation}
for all~$x,y,v,w\in\mathbb{R}^d$,
where~$\hat{\gamma}=(1-\eta)/(\eta T)\in(0,\infty]$,~$\hat{\alpha}=1.09LT^2/(1-\eta)^2$ 
and if~$\hat{\gamma}=\infty$, then~$\hat{d}$ is interpreted as a metric on~$\mathbb{R}^d$. Moreover, let~$\abs{\cdot}_{\hat{d}}$ denote the corresponding norm. Such a twisted metric aligns with~\cite{MR3980913}, but has some consequences in terms of Gaussian concentration and bias, 
see Remark~\ref{ewr}. The constant~$1.09$ appearing in the definition of~$\hat{\alpha}$ is chosen in order to optimize the contraction rate obtained in the results below.

\paragraph*{Wasserstein convergence.} 
We present our first main result here. 
In the following Theorem~\ref{intma}, 
let the sequence of~$\theta,\theta'$ to be used for each run of the Hamiltonian step in SGgHMC be predetermined (the cases where they are random being treated easily by taking the expectation with respect to their law in the next result). For the distribution~$\nu$ of any~$\mathbb{R}^{2d}$-valued random variable, let~$\pi_n(\nu)$ denote the distribution after~$n$ iterations of SGgHMC. Moreover, 
let~$\mathcal{W}_1,\mathcal{W}_2$ denote the~$L^1$ and~$L^2$ Wasserstein distances with respect to the twisted Euclidean metric~$\hat{d}$ given by~\eqref{tmet} 
(see e.g.~\cite[equation~(1.2)]{MR4132634}). 
Finally, for~$E\in\{\mathbb{R}^d,\mathbb{R}^{2d}\}$ and any semimetric~$\rho:E\times E\rightarrow [0,\infty)$, let~$\mathcal{W}_{\rho}$ denote Kantorovich semimetric w.r.t.~$\rho$ (see e.g.~\cite[equation~(5)]{MR3568041}). 

\begin{theorem}[High friction Wasserstein bounds]\label{intma}
Under Assumption~\ref{A1}, if inequalities~$4LT^2\leq (1-\eta)^2$ and
\begin{equation}\label{carot}
\frac{L(T+h)^2(1+\eta)}{1-\eta}\leq \frac{1}{16^2}\min \bigg(\frac{1}{L\hat{R}^2},1\bigg)
\end{equation}
hold for~$\hat{R}$ given by~\eqref{rhdef}, 
then there exists a semimetric~$\rho^*$ and constants~$c,M_1,M_2>0$ such that 
for all initial distributions~$\nu_1,\nu_2$ and~$n\in\mathbb{N}$, it holds that
\begin{align}
\mathcal{W}_1(\pi_n(\nu_1),\pi_n(\nu_2)) &\leq M_1e^{-cn}\mathcal{W}_1(\nu_1,\nu_2),\label{intma1}\\
\mathcal{W}_2(\pi_n(\nu_1),\pi_n(\nu_2)) &\leq M_2e^{-cn}\sqrt{\mathcal{W}_{\rho^*}(\nu_1,\nu_2)}.\label{intma2}
\end{align}
\end{theorem}
The proof of this theorem, which is the main part of this work, is the content of Sections~\ref{ncsec}, \ref{convex} and \ref{global}, leading  eventually to the final statement of Corollary~\ref{maincore}, where explicit expressions for~$c$, $M_1$, $M_2$ and~$\rho^*$ are given, and from which Theorem~\ref{intma} follows. For the sake of clarity, let us sketch the general structure of the proof:
\begin{itemize}
    \item The proof relies on the definition of a suitable Markovian coupling of two chains starting with different initial conditions.  We distinguish two cases: if the distance (as in~$\hat{d}$ given by~\eqref{tmet}) between the two chains is above some threshold, a parallel coupling is performed, namely the same Gaussian variable is used in the velocity refreshment steps of both chains (the deterministic drift will be sufficient to bring closer the two trajectories). By contrast, below this threshold,  the Gaussian variables in the velocity refreshment steps are given by a reflection (or mirror) coupling along a particular direction. 
    This coupling will induce some variance in the evolution of the distances of the two chains along this direction, whilst forcing a contraction in another. The variance will classically be leveraged by using a suitable concave modification of the distance to obtain a contraction overall.
    \item The reflection coupling is the topic of Section~\ref{ncsec}, whose main result is Theorem~\ref{main2}, stating a contraction result on average for a suitable distance between the two chains, when their initial distance is below the aforementioned threshold.
    \item Section~\ref{convex} deals with the case where the distance between the two chains is large. In view of the contraction effect of the drift in this region, the study then shares some similarities to the convex cases considered in \cite{lc23,gouraud2023hmc}. We distinguish again two sub-cases, depending on whether a dominating part of the distance is due to the positions or the velocities, respectively studied in Sections~\ref{lpd} and~\ref{rlvd}. The two cases are then gathered in Theorem~\ref{convthm}, where a modified Euclidean distance is shown to be contracted along the parallel coupling, provided the initial distance is large. The remaining of Section~\ref{caoffi} is then devoted to the proof of Corollary~\ref{refmod}, which shows that this modified Euclidean distance is not expanded too much when 
    reflection  coupling is used instead of the parallel one.
    \item The previous results are then combined in Section~\ref{global}, leading to Theorem~\ref{coth}, which is a more detailed version of Theorem~\ref{intma}. In that section is also stated and proven Theorem~\ref{coth2}, which is a variant where the condition~$4LT^2\leq (1-\eta)^2$ is removed and replaced by a Lyapunov condition (Assumption~\ref{A2}) in the spirit of~\cite[equation~(33)]{MR4133372}.
\end{itemize}

Some remarks on Theorem~\ref{intma} are as follows.
\begin{itemize}
    \item 
The assumption~$4LT^2\leq (1-\eta)^2$ can be thought of as a lower bound on a friction-like parameter (consider~$\eta=e^{-\bar{\gamma}T}$ with fixed~$\bar{\gamma}>0$ and small~$T$). This assumption allows synchronous coupling to be useful in certain regions of state space and enables contraction as in~\eqref{intma1} (with the~$\mathcal W_1$ distance on both sides of the inequality). 
Otherwise if~$4LT^2>(1-\eta)^2$, namely in the low friction regime, a~$\mathcal{W}_1$ convergence result is still possible (as consequence to Theorem~\ref{coth2} below), with terms other than~$\mathcal{W}_1(\nu_1,\nu_2)$ on the right-hand side of the bound~\eqref{intma1}, similar to~\eqref{intma2}. In the rest of this section we mostly discuss consequences of \eqref{intma1}. Note that the high friction regime can always be enforced by the user via a suitable choice of parameters. Besides, the condition~\eqref{carot} should rather be interpreted as a bound on the integration time~$T$ to avoid periodic orbits, as discussed e.g. in \cite{gouraud2023hmc}, and a condition of this form is known to be necessary to ensure ergodicity. 

\item The constants~$M_1,M_2,c$ in Theorem~\ref{intma} depends only of~$m,L,R,\eta,K,h$ and have no additional dependency on the dimension~$d$.  
However, an implicit dependence possibly exists through the constant~$R$ satisfying Assumption~\ref{A1}. Moreover, an explicit dependence on~$d$ appears in the Gaussian concentration (Corollary~\ref{gaucor}) and estimator bias (Corollary~\ref{bicor}) results that are derived as corollaries to Theorem~\ref{intma}. The dependence of~$M_1,M_2,c$ on~$R$ is exponential, but the order of~$R$ within the exponent is no worse than in~\cite{MR4133372}. As explained in~\cite[Remark 6]{monmarche2023entropic}, this cannot be avoided since the results apply to all cases satisfying Assumption~\ref{A1}, which contains multi-modal targets with energy barriers of order $R$, for which the exact convergence rate is expected to scale exponentially in $R$. 
\end{itemize} 

\paragraph*{Wasserstein-to-TV regularization.} To complement Theorem~\ref{intma} (and Theorem~\ref{coth2}), TV/Wasserstein regularization in the spirit of~\cite[Proposition~3]{monmarche2023entropic} or \cite[Lemma~17]{MR4497240} is available.

\begin{theorem}[Wasserstein-to-TV regularization]\label{prop:TV/W1}
Under Assumption~\ref{A1}, assume moreover that there exists~$L_H>0$ such that for any~$\theta\in\Theta$ and $x,y\in\mathbb{R}^d$, we have~$\|\nabla_x b(x,\theta) - \nabla_x b(y,\theta)\|_F\leq L_H\abs{x-y}$. There exists $\beta,T_0'>0$ which depends only on~$L$ and $L_H$ such that, assuming $T\leq T_0'$ and $\eta>0$, for all initial distributions $\nu_1,\nu_2$ over $\R^{2d}$ and all $n\geq 2$ with~$s:=nT \leq 1$, denoting~$\gamma^*=(1-\eta)/T$, it holds that
\begin{equation}\label{eq:propdensite1}
\|\pi_n(\nu_1) - \pi_n(\nu_2)\|_{\textrm{TV}} \leq  \beta \sqrt{\frac{ s}{\gamma^*}} \po \frac{1}{\eta s^2} + \frac{\gamma^*}{\eta s} + 1\pf \mathcal W_1(\nu_1,\nu_2) \,.
\end{equation}
\end{theorem}

The proof, given in Appendix~\ref{sec:WtoTV}, is similar to \cite[Proposition~3]{monmarche2023entropic}, which states a similar result but for an idealized gHMC chain. 

\begin{remark}
The cases $n=1$ or $\eta=0$ are excluded in Theorem~\ref{prop:TV/W1} to get a simpler statement, but they can be treated, as seen in Appendix~\ref{sec:WtoTV}. However, when $n=1$, it is necessary to add a second velocity refreshment step after the Hamiltonian integration step. The case where $\eta=0$ and $n=1$ is covered  for the Verlet integrator by \cite[Lemma~17]{MR4497240}. The way Theorem~\ref{prop:TV/W1} is stated is mainly adapted to the case where $\gamma^*$ is a fixed positive constant. Then the prefactor in \eqref{eq:propdensite1} scales like $s^{-3/2}$ for small $s$, as expected ($s=nT$ being the physical time of the simulation after $n$ iterations).
\end{remark}

Applying Theorem~\ref{prop:TV/W1} with $\nu_i=\pi_m(\nu_i')$ for some $m$, we can then use \eqref{intma2} to bound the right hand side of~\eqref{eq:propdensite1}. This gives a convergence rate in relative entropy and total variation distance. Moreover, applying this with $\nu_i' =\delta_{z_i}$ for some $z_1,z_2\in\R^{2d}$, we can find, for any $R_0>0$, a number of steps $n\in\N$ such that 
\begin{equation}
    \label{eq:localcoupling}
\|\pi_n(\delta_{z_1})-\pi_n(\delta_{z_2})\|_{TV} \leq 1 \qquad \forall z_1,z_2\in B_{R_0}\,.
\end{equation}
This local coupling bound is a crucial ingredient in the quantitative analysis of the adjusted versions of gHMC, so that getting explicit estimates of $n$ here (in terms of $R_0$ and of the parameters of the problem) has direct consequences of the explicit estimates provided by these methods. Indeed, taking a step size sufficiently small so that a trajectory starting from any $z\in B_{R_0}$ has a probability at least 7/8 to be accepted for $n$ steps (which, in cases where $T$ is of order $h$ and $n$ is of order $1/h$, corresponds to what is called in \cite{bourabee2023mixing} the high acceptance regime), we get  a result similar to \eqref{eq:localcoupling} but now for the adjusted chain (except that $1$ is replaced by $3/4$). Such bounds are used in various methods to analyse adjusted chains:
\begin{itemize}
    \item In the classical Meyn-Tweedie (or Foster/Lyapunov) approach, it is used in conjunction with a Lyapunov function, as in \cite{Livingstone,DurmusMoulinesSaksman} for qualitative results on classical HMC (see e.g.~\cite[Theorem S7]{DurmusGuillinMonmarche} for  a quantitative statement to combine a local coupling estimate with a Lyapunov drift condition).
    \item For conductance methods, it is the main step to get local conductance lower bounds, see   e.g. in \cite[Lemma 4]{10.5555/3455716.3455808}.
    \item In the localized mixing time approach of \cite{bourabee2023mixing}, similarly, it is one of the key ingredient.   
\end{itemize}
In other words, given a compact set $\Omega$ that contains most of the mass of the target measure,  these three methods have distinct ways to deal with $\Omega^c$ but the local analysis inside $\Omega$ is always based on a local coupling condition~\eqref{eq:localcoupling}. 

The detailed analysis of adjusted gHMC in nonconvex cases based on our results is postponed to future work.

\paragraph*{Gaussian concentration for empirical averages.}
In the rest of the section, the consequences of the~$L^1$ Wasserstein contraction~\eqref{intma} are given. In Corollary~\ref{gaucor}, Gaussian concentration bounds on ergodic averages are presented.

Recall~$\hat{d}$ to be the twisted metric given as in~\eqref{tmet}, interpreted as a metric on~$\mathbb{R}^d$ if~$\hat{\gamma}=\infty$. Again, assume that~$\theta,\theta'$ used in SGgHMC are predetermined. In the following Corollary~\ref{gaucor}, 
suppose that some~$(x,v)\in\mathbb{R}^{2n}$ is given as the initial state of SGgHMC and recall for each~$n\in\mathbb{N}$ that~$(x_n,v_n)$ denotes the position-velocity state of the SGgHMC chain after~$n$ iterations.

\begin{corollary}\label{gaucor}
Let Assumption~\ref{A1} hold 
and~$f:\mathbb{R}^d\rightarrow\mathbb{R}$ be Lipschitz. 
Assume~$4LT^2\leq(1-\eta)^2$ and~\eqref{carot}.
There exists~$C>0$ depending only on~$\hat{\alpha},m,L,R$ (in particular it is independent of~$d$) 
such that for any~$N_0\in\mathbb{N}\setminus\{0\},N\in\mathbb{N}$,~$r>0$, it holds that
\begin{equation}
\mathbb{P}\bigg(\frac{1}{N}\sum_{i=N_0+1}^{N_0+N} (f(x_i) - \mathbb{E}[ f(x_i)]) > r\bigg) \leq \exp\bigg(-\frac{ (1-\eta)Nr^2}{CT^2\|f\|_{\textrm{Lip}(\hat{d})}^2}\bigg).\label{gaucoreq}
\end{equation}
\end{corollary}
Corollary~\ref{gaucor} uses the results of~\cite{MR2078555} and follows immediately from Corollary~\ref{gaucor_old}, which gives an explicit value for~$C$. 
\begin{remark}\label{ewr}
The Lipschitz norm~$\|f\|_{\textrm{Lip}(\hat{d})}$ 
suffers from a dependence on the (square inverse) friction-like value~$\hat{\alpha}=LT^2/(1-\eta)^2\in[LT^2,1/4]$. 
More concretely, observe that if~$\eta=0$ or~$\eta=1-2\sqrt{L}T$, then~$\|f\|_{\textrm{Lip}(\hat{d})} \leq 2\|f\|_{\textrm{Lip}(d
_e)}$ holds for the standard Euclidean metric~$d_e$. For all other values of~$\hat{\alpha}$, it holds that~$\|f\|_{\textrm{Lip}(\hat{d})} = (1.09\hat{\alpha})^{-1}\|f\|_{\textrm{Lip}(d_e)}$. 
This type of dependence on~$\hat{\alpha}$ appears also within the constant~$C$ in Corollary~\ref{gaucor} and it is an implication of using the twisted metric~$\hat{d}$ in contrast to the standard Euclidean metric. 
Although the use of a twisted metric is consistent with~\cite{MR3980913}, given the well-behaved large friction limits of the algorithm as detailed in~\cite[Section~7]{lc23}, it leaves open whether Gaussian concentration bounds independent of~$\hat{\alpha}$ or of~$\hat{\gamma}$ are possible in the nonconvex setting here. A similar question remains for the bound on the bias later in Corollary~\ref{bicor}, see also the dependence in~\cite[Theorem~1]{camrud2023second}. 
On the other hand, note that given one or multiple~$f$, it is possible to find algorthmically an optimal friction with respect to the asymptotic variance~\cite{refId0}.
\end{remark}

\paragraph*{Numerical and stochastic approximations biases.} The concentration inequality \eqref{gaucoreq} only controls the distance between empirical averages and their expectation, and thus it remains to bound the bias of these estimators, namely the distance between their expectation and the target quantity, which is the average of~$f$ with respect to~$\mu$ the invariant measure of the idealized gHMC chain where~\eqref{vvi} is replaced by the exact Hamiltonian flow. There are (possibly) three contributions to the bias: first, the fact that the chain is not at stationarity; second, the discretization error of the Verlet scheme; third, the stochastic gradient approximation. 
These contributions may be combined to yield a quantitative bound on the bias of empirical averages to~$\mu$.


Let~$U\in C^2(\mathbb{R}^d)$ satisfy~$\nabla U(0) = 0$ (w.l.o.g.) and also~\eqref{A10},~\eqref{A1a} both with~$b(\cdot,\theta)$ replaced by~$\nabla U$. Moreover, let~$\mu$ denote the probability measure with density proportional to~$\mathbb{R}^{2d}\ni(x,v)\mapsto e^{-U(x)-\abs{v}^2/2}$. Recall that~$\hat{d}$ is given by~\eqref{tmet} and~$\abs{\cdot}_{\hat{d}}$ by the associated norm, both interpreted in the space~$\mathbb{R}^d$ if~$\hat{\gamma}=\infty$. 
In Corollary~\ref{bicor}, suppose again that some~$(x,v)\in\mathbb{R}^{2d}$ is given as the initial state of SGgHMC and recall for each~$n$ that~$(x_n,v_n)$ denotes the state after~$n$ iterations. 
Moreover, in Corollary~\ref{bicor}, we assume that~$b$ and~$\theta$ used in SGgHMC are nice, in the sense that~$b$ is a stochastic approximation of~$\nabla U$ that satisfies some unbiasedness and finite variance conditions. We defer the precise assumptions on the stochasticity of the gradient to Section~\ref{bisec} and the statement of Corollary~\ref{bicor_old}. 
\begin{corollary}\label{bicor}
Let Assumption~\ref{A1} hold and 
let~$f:\mathbb{R}^d\rightarrow\mathbb{R}$ be Lipschitz. Assume~$4LT^2\leq(1-\eta)^2$ and~\eqref{carot}. Assume that $b$ is an unbiased estimator of $\na U$ with uniformly bounded variance as detailed in Section~\ref{bisec}. There exists~$C,\lambda>0$,~$C'\geq 0$ all depending on~$R,m,L$ but independent of~$\eta,K,h,d$ such that for any~$N_0,N\in\mathbb{N}$, it holds that, for the SGgHMC chain with initial condition $(x,v)\in\R^{2d}$, 
\begin{align*}
&\frac{1}{\|f\|_{\textrm{Lip}(\hat{d})}}\bigg|\mathbb{E}\bigg[\frac{1}{N}\sum_{i=N_0+1}^{N_0+N}\bigg(f(x_i) - \int fd\mu\bigg)\bigg]\bigg|\nonumber\\
&\quad\leq \frac{C(1-\lambda T^2/(1-\eta))^{N_0}}{N\lambda T^2/(1-\eta)}\Big(\abs{(x,v)}_{\hat{d}} + \sqrt{d^*}+ B\Big) \nonumber\\
&\qquad+ C' L^{\frac{1}{4}} \sqrt{h} \cdot\frac{\exp(\frac{3}{2}\sqrt{L}T(N_0+N))}{\frac{3}{2}N\sqrt{L}T} + B,
\end{align*}
where~$d^* = m^{-1}\max(16LR^2(1+L/m)^2,d)$ and, in the case of the Verlet integrator~\eqref{vvi},~$B = 
C^{1+2\sqrt{L}(1-\eta)/(\lambda T)}h\sqrt{d^*}$, while, if instead the randomized midpoint method as in~\cite{bourabee2022unadjusted} is used,~$B = C^{1+2\sqrt{L}(1-\eta)/(\lambda T)}T^{-1}h^{3/2}\sqrt{d^*}$. 
In case the velocity Verlet integrator~\eqref{vvi} is used, if in addition there exists~$L_2>0$ such that~$\abs{\nabla^2 U(x) - \nabla^2 U(x)}\leq L_2\abs{x-x'}$ for all~$x,x'\in\mathbb{R}^d$, then the same assertion as above 
holds but with~$C$ depending also on~$L_2$ and the definition of~$B$ replaced by~$B = C^{1+2\sqrt{L}(1-\eta)/(\lambda T)}h^2d^*$. 
Furthermore, in all cases, if~$b$ is a deterministic gradient, namely if~$b(\cdot,\theta)=\nabla U$ for all~$\theta\in\Theta$, then~$C'=0$ holds.
\end{corollary}
\begin{remark}
The value~$B$ is exactly an upper bound on the bias between the invariant law of the nonstochastic gHMC chain and~$\mu$ (given by Corollary~\ref{invd}). The scaling in the exponent on~$C$ in either definitions of~$B$ is such 
that~$B$ will not converge to zero as~$h\rightarrow 0$ if for example~$K=1$ and~$\eta$ is a constant independent of~$T$. In particular, when~$\eta=0$ and~$K=1$, which is the case of ULA~\cite{durmus2018highdimensional}, the value of~$B$ tends to infinity as~$h\rightarrow 0$. On the other hand, this particular case is well studied and in particular Theorem~5 in~\cite{durmus2023asymptotic} may be used in place of Corollary~\ref{invd} in the proof of Corollary~\ref{bicor_old}. Moreover, in the Langevin case~$\eta=1-2\sqrt{L}T$ with small~$T$, the aforementioned exponent is constant, so that taking~$h\rightarrow 0$ minimizes the bias as desired.
\end{remark}

\section{Reflection coupling in a bounded region}\label{ncsec}

In this section, a reflection coupling (see~\cite{MR0841588,MR3568041} for the seminal works on this) will be devised in order to prove a local contractivity of SGgHMC iterations based on a suitably chosen concave function. The proof is a novel and explicit estimation of the expected difference before and after a SGgHMC iteration. 
The concave function 
will then form part of a semimetric used to give a global contraction in Section~\ref{global}, which is subsequently used to prove Wasserstein contraction in Corollary~\ref{maincore} and hence the announced Theorem~\ref{intma}. Below, Section~\ref{secsg} introduces the coupling and gives the central Lemma~\ref{base}, which yields a priori bounds on the associated difference chain based on similar a priori estimates from~\cite{MR4133372}. 
Note that the results in this section are proven for all friction values, that is, all parameter values~$K,h,\eta$ satisfying conditions analogous to~\eqref{carot}.

Let~$\gamma\in(0,\infty]$ 
satisfy~$\eta(1+\gamma T)\leq1$. For now,~$\gamma$ is left arbitrary and Lemma~\ref{base} is proved for this general constant. Its value is fixed relative to~$K,h,\eta$ from the beginning of Section~\ref{contractionsec} and onwards for the rest of the paper. We build on the notation in~\cite[equation~(7)]{MR4133372} that~$(t,x,v)\mapsto q_t(x,v)$ and~$(t,x,v)\mapsto p_t(x,v)$ denote the trajectory of Hamiltonian dynamics approximated by the velocity Verlet integrator at time~$t$ with initial position and velocity~$x$ and~$v$. In particular, this notation is generalized to allow a randomized midpoint version with stochastic gradients as defined below. 
Let~$\hat{\theta} := (\theta_{ih})_{i\in\mathbb{N}\cap[0,T/h)},\hat{\theta}' :=(\theta_{ih}')_{i\in\mathbb{N}\cap[0,T/h)}\in\Theta^{T/h}$. Let~$u\in\{0,1\}$ and let~$\hat{u}:=(u_{hk})_{k\in \mathbb{N}}$ be a sequence of independent r.v.'s such that~$u_{hk}\sim\mathcal{U}(0,1)$ for all~$k\in \mathbb{N}$. Denote~$\bar{\theta} = (\hat{\theta},\hat{\theta}',u,\hat{u})$.
Moreover, for any~$s\in[0,T]$, let~$\floor{s} = \max\{t\in h\mathbb{N}:t\leq s\}$ and~$\ceil{s} = \min\{t\in h\mathbb{N}:t\geq s\}$.
For any~$x,v\in\mathbb{R}^d$, let~$[0,T]\ni t\mapsto(\bar{q}_t,\bar{p}_t) = (\bar{q}_t(x,v,\bar{\theta}),\bar{p}_t(x,v,\bar{\theta}))$ be the solution to
\begin{subequations}\label{hameqs}
\begin{align}
\frac{d}{dt}\bar{q}_t &= \bar{p}_{\floor{t}} - \frac{h}{2}b(\bar{q}_{\floor{t}} + uu_{\floor{t}} h\bar{p}_{\floor{t}},\theta_{\floor{t}}),\\
\frac{d}{dt}\bar{p}_t &= -\frac{1}{2}(1+u) b(\bar{q}_{\floor{t}} + uu_{\floor{t}}h\bar{p}_{\floor{t}},\theta_{\floor{t}}) - \frac{1}{2}(1-u) b(\bar{q}_{\ceil{t}},\theta_{\floor{t}}') 
\end{align}
\end{subequations}
with~$(\bar{q}_0,\bar{p}_0)= (x,v)$. If~$u=1$, then~$(\bar{q}_{\cdot},\bar{p}_{\cdot})$ coincides with the randomized midpoint integrator studied in~\cite[Algorithm~1]{bourabee2022unadjusted}. If instead~$u=0$, then~$(\bar{q}_{\cdot},\bar{p}_{\cdot})$ is the velocity Verlet integrator presented in~\eqref{vvi}. 

We state next a priori estimates for the integrator~\eqref{hameqs}, which will be useful for our analysis. The proof of Lemma~\ref{basic} is similar to and builds on that of Lemma~3.2 in~\cite{MR4133372}.
\begin{lemma}\label{basic}
Assume~\eqref{A1a}. Let~$c\geq1$ and assume~$LT(T+h)\leq c^{-1}$. For any~$x,y,v,w\in\mathbb{R}^d$, it holds a.s. that
\begin{align}
&\max_{s\in[0,T]}\abs{\bar{q}_s(x,v,\bar{\theta}) - \bar{q}_s(y,w,\bar{\theta}) - (x-y) - s(v-w)}\nonumber\\
&\quad\leq (3/(2c-1))\max(\abs{x-y},\abs{x-y+(v-w)T}),\label{lem64a}\\
&\max_{s\in[0,T]}\abs{\bar{p}_{\bar{s}}(x,v,\bar{\theta}) - \bar{p}_s(y,w,\bar{\theta}) - (v-w)} \nonumber\\
&\quad\leq (6c/(2c-1))LT\max(\abs{x-y},\abs{x-y+(v-w)T}).\label{lem64}
\end{align}
\end{lemma}
\begin{proof}
For any~$s\in[0,T]$, denote~$\bar{q}_s:=\bar{q}_s(x,v,\bar{\theta})$,~$\bar{q}_s':=\bar{q}_s(y,w,\bar{\theta})$ and~$\bar{p}_s:=\bar{p}_s(x,v,\bar{\theta})$,~$\bar{p}_s':=\bar{p}_s(y,w,\bar{\theta})$. In the same way as in the proof of Lemma~3.2 in~\cite{MR4133372}, it holds that
\begin{equation*}
\abs{\bar{q}_s - \bar{q}_s' - (x-y + s(v-w))} \leq \frac{L}{2}(T^2+Th)\max_{s\leq T}(\abs{\bar{q}_s-\bar{q}_s'} + h\abs{\bar{p}_s-\bar{p}_s'})
\end{equation*}
which implies
\begin{align*}
&\abs{\bar{q}_s - \bar{q}_s' - (x-y + s(v-w))} \\
&\quad\leq \frac{L}{2}(T^2+Th)\max_{s\leq T}(\abs{\bar{q}_s-\bar{q}_s'- (x-y+s(v-w))} \\
&\qquad+ \max(\abs{x-y},\abs{x-y+(v-w)T} + h\abs{\bar{p}_s-\bar{p}_s'}).
\end{align*}
Therefore, by~$LT(T+h)\leq c^{-1}$, it holds that
\begin{equation}\label{haf1}
\abs{\bar{q}_s - \bar{q}_s' - (x-y + s(v-w))} \leq \frac{1}{2c-1}( \max(\abs{x-y},\abs{x-y+(v-w)T} + h\max_{s\leq T}\abs{\bar{p}_s-\bar{p}_s'}),
\end{equation}
which implies
\begin{equation}
\max_{s\in[0,T]}\abs{\bar{q}_s - \bar{q}_s'}
\leq \frac{2c}{2c-1} \max(\abs{x-y},\abs{x-y+(v-w)T}) + \frac{h}{2c-1}\max_{s\leq T}\abs{\bar{p}_s-\bar{p}_s'}.\label{gaq}
\end{equation}
On the other hand and similar to the proof of Lemma~3.2 in~\cite{MR4133372}, it holds by~\eqref{gaq} that
\begin{equation*}
\max_{s\in[0,T]}\abs{\bar{p}_s-\bar{p}_s' - (y-w)}  \leq LT\max_{s\leq T}(\abs{\bar{q}_s-\bar{q}_s'}+h\abs{\bar{p}_s-\bar{p}_s'}),
\end{equation*}
which implies by~\eqref{gaq} that
\begin{align}
&\max_{s\in[0,T]}\abs{\bar{p}_s-\bar{p}_s' - (y-w)} \nonumber\\
&\quad\leq LT\bigg(\frac{2c}{2c-1}\max(\abs{x-y},\abs{x-y + (v-w)T}) +\frac{2ch}{2c-1}\max_{s\leq T}\abs{\bar{p}_s-\bar{p}_s'}\bigg).\label{haf2}
\end{align}
Therefore it holds by~$ \abs{\bar{p}_s-\bar{p}_s'} \leq \abs{y-w} + \abs{\bar{p}_s-\bar{p}_s' - (y-w)}$ and~$LTh\leq LT(T+h)/2\leq (2c)^{-1}$ that
\begin{equation}
\max_{s\in[0,T]}\abs{\bar{p}_s-\bar{p}_s'} \leq \frac{2c-1}{2c-2}\bigg(\abs{y-w} + \frac{2cLT}{2c-1}\max(\abs{x-y},\abs{x-y + (v-w)T})\bigg).
\end{equation}
Subsequently, by~$\abs{h(v-w)}\leq \abs{(-(x-y) + (x-y) + T(v-w)}\leq 2\max(\abs{x-y},\abs{x-y + T(v-w)})$ and again~$LTh\leq LT(T+h)/2\leq (2c)^{-1}$, it holds that
\begin{equation}\label{haf}
h\max_{s\in[0,T]}\abs{\bar{p}_s-\bar{p}_s'} \leq \frac{4c-1}{2c-2}\max(\abs{x-y},\abs{x-y + (v-w)T}).
\end{equation}
Substituting~\eqref{haf} into~\eqref{haf1} and~\eqref{haf2} concludes the proof.
\end{proof}

\subsection{Coupling}\label{secsg}
Here the relationship between~$G$ and~$\hat{G}$ at each iteration of SGgHMC is introduced. 
For any~$x,y,v,w\in\mathbb{R}^d$ and~$G,\hat{G}\sim\mathcal{N}(0,I_d)$, 
an iteration of the coupled dynamics is given by
\begin{subequations}\label{ous}
\begin{align}
v' &= \eta v + \sqrt{1-\eta^2}G,& w' &= \eta w + \sqrt{1-\eta^2} \hat{G},\label{ou1}\\
X' &= \bar{q}_T(x,v',\bar{\theta}),& Y' &= \bar{q}_T(y,w',\bar{\theta}),\\
V' &= \bar{p}_T(x,v',\bar{\theta}), & W' &= \bar{p}_T(y,w',\bar{\theta}).
\end{align}
\end{subequations}
For~$x,y,v,w\in\mathbb{R}^d$ together with~\eqref{ous}, the following notation will be used:
\begin{subequations}\label{zdefs}
\begin{align}
z &:= x-y,\\
q &:= x-y + \gamma^{-1}(v-w),\\
Z' &:= X'-Y',\\
Q' &:= X'-Y' + \gamma^{-1}(V' - W'),\\
\hat{q} &:= r^*(T+\gamma^{-1})\sqrt{1-\eta^2}\big(x-y+\eta(T+\gamma^{-1})(v-w)\big),\label{zdef5}
\end{align}
\end{subequations}
where~$r^*\in(0,1/((T+\gamma^{-1})^2(1-\eta^2))]$ and~$\gamma$ is a variable as stated at the beginning of the section, but should be interpreted to be fixed as in~\eqref{gdef} below from the beginning of Section~\ref{contractionsec} and onwards for the rest of the paper. 
The expression for~$\hat{q}$ is of interest here because the random variable~$\hat{G}$ is coupled with~$G$ by
\begin{equation}\label{cou1}
\hat{G} = \begin{cases}
G +\hat{q} &\textrm{if } \mathcal{U}\leq \frac{\varphi_{0,1}(e\cdot G + \abs{\hat{q}})}{\varphi_{0,1}(e\cdot G)} \\
(I - 2ee^T)G &\textrm{otherwise},
\end{cases}
\end{equation} 
where
\begin{equation}\label{edef}
e = \begin{cases}
\hat{q}\abs{\hat{q}}^{-1} &\textrm{if }\hat{q}\neq 0 \\
(1,0,\dots,0)&\textrm{otherwise,}
\end{cases}
\end{equation}
and~$\mathcal{U}\sim \mathcal{U}(0,1)$ is independent of~$G$ and~$(u_{hk})_{k\in\mathbb{N}}$. The value~$r^*$ is to be determined. In~\cite{MR3949304}, a similar value is considered as a tuning parameter in a coupling-based HMC algorithm. For now, we consider any possible value in the stated range, but we will fix~$r^*$ to be the maximum possible value later and justify this choice with our analysis.

In addition, it will be convenient to set the following notation in order to represent the coupling in the direction of~$\hat{q}$. For any~$\bar{v}\in\mathbb{R}^d$, let~$K_{\bar{v}}$ be given by 
\begin{equation}\label{kdef}
K_{\bar{v}}= \bigg|\frac{\bar{v}}{r^*(T+\gamma^{-1})\sqrt{1-\eta^2}} + (T+\gamma^{-1})\sqrt{1-\eta^2} \bar{G}\bigg|,
\end{equation}
where~$\bar{G} = G-\hat{G}$. In the case where~$G$ is coupled with~$\hat{G}$ as in~\eqref{cou1} and~\eqref{edef},~$\bar{G}$ is given by
\begin{equation}\label{Gdef2}
\bar{G} = G-\hat{G} = \begin{cases}
-\hat{q} & \textrm{if }\mathcal{U} \leq \frac{\varphi_{0,1}(e\cdot G + \abs{\hat{q}})}{\varphi_{0,1}(e\cdot G)}\\
2ee^{\top} G & \textrm{otherwise},
\end{cases}
\end{equation}
where~$\hat{q}$ and~$e$ are given by~\eqref{zdefs},~\eqref{cou1} and~\eqref{edef}. 

Next, in Lemma~\ref{base}, we establish estimates on~$Z',Q'$ based on the a priori estimates for the Verlet integrator established in~\cite{MR4133372}. Note that we do not assume here that~$G$ and~$\hat{G}$ satisfy~\eqref{cou1},~\eqref{edef}. In particular, Lemma~\ref{base} applies in case of both synchronous and reflection coupling.

\begin{lemma}\label{base}
Let~$r^*\in(0,(1-\eta)^2/(T^2(1-\eta^2))]$, let~$x,y,v,w$ be~$\mathbb{R}^d$-valued r.v.'s and let~$G\sim \mathcal{N}(0,I_d)$ be independent of~$x,y,v,w,(u_{kh})_{k\in\mathbb{N}}$. Moreover, let~$z,\hat{q},Z',Q'$ be given by~\eqref{zdefs}, let~$\hat{G}\sim\mathcal{N}(0,I_d)$ and let~$K_{\cdot}$ be given by~\eqref{kdef} with~$\bar{G} = G-\hat{G}$. 
Suppose~\eqref{A1a} holds 
and that~$T\in h\mathbb{N}$ satisfies~$0<LT(T+h)\leq 1/16^2$. It holds a.s. that
\begin{align}
\abs{Z'} &\leq \frac{\gamma^{-1}\abs{z} + TK_{\hat{q}}}{T+\gamma^{-1}} + cLT(T+h)\bigg(\abs{z}+\frac{TK_{\hat{q}}}{2(T+\gamma^{-1})}\bigg),\label{base1}\\
\abs{Q'} &\leq K_{\hat{q}} + 2cL\bigg(\gamma^{-1}T + \frac{1}{2}T(T+h)\bigg)\bigg(\abs{z} + \frac{TK_{\hat{q}}}{2(T+\gamma^{-1})}\bigg).\label{base2}
\end{align}
where~$c=32/31$. The same inequalities hold with~$c$ replaced by~$c/2$ if the left-hand sides are replaced by~$\mathbb{E}[\abs{Z'}|x,y,v,w,G,\hat{G}]$ and~$\mathbb{E}[\abs{Q'}|x,y,v,w,G,\hat{G}] $ respectively. 
\end{lemma}
\begin{proof}
Recall~$v',w'$ as defined in~\eqref{ous}. 
Throughout the proof, we denote~$\bar{q}_{\cdot}=\bar{q}_{\cdot}(x,v',\bar{\theta})$,~$\bar{p}_{\cdot}=\bar{p}_{\cdot}(x,v',\bar{\theta})$,~$\bar{q}_{\cdot}'=\bar{q}_{\cdot}(y,w',\bar{\theta})$,~$\bar{p}_{\cdot}'=\bar{p}_{\cdot}(y,w',\bar{\theta})$. 
Let~$t\in[0,T]\cap h\mathbb{N}$. It holds that
\begin{align}
x-y + t(v'-w') &= x-y + \frac{t}{T+\gamma^{-1}}(T+\gamma^{-1})(\eta(v-w) + \sqrt{1-\eta^2}(G-\hat{G}))\nonumber\\
&= \bigg(1- \frac{t}{T+\gamma^{-1}}\bigg)z + \frac{t}{T+\gamma^{-1}} \Big(z + \eta(T+\gamma^{-1})(v-w)\nonumber\\
&\quad + (T+\gamma^{-1})\sqrt{1-\eta^2}\hat{G}_{\bar{q}}\Big). \label{fe}
\end{align}
Equation~\eqref{fe} and Lemma~\ref{basic} with~$c=16^2$ implies a.s. that
\begin{align}
\abs{\bar{q}_t - \bar{q}_t'} &\leq \max_{s\leq t}\abs{\bar{q}_s-\bar{q}_s' - (x-y)-s(v'-w')} + \abs{x-y + t(v'-w')}\nonumber\\
&\leq \frac{3}{2\cdot16^2-1}\max(\abs{z},\abs{z + t(v'-w')}) + \frac{(T-t+\gamma^{-1})\abs{z}}{T+\gamma^{-1}} + \frac{tK_{\hat{q}}}{T+\gamma^{-1}}\nonumber\\
&\leq \frac{3}{511}\bigg(\abs{z} + \frac{tK_{\hat{q}}}{T+\gamma^{-1}}\bigg) + \frac{(T-t+\gamma^{-1})\abs{z}}{T+\gamma^{-1}} + \frac{tK_{\hat{q}}}{T+\gamma^{-1}}\nonumber\\
&\leq \frac{514}{511}\bigg(\abs{z} + \frac{tK_{\hat{q}}}{T+\gamma^{-1}}\bigg). \label{base0}
\end{align}
Consider first the case~$u=0$, that is, the velocity Verlet integrator. Let~$\gamma_1\in (0,\infty]$ and recall the notation~$\hat{\theta} = (\theta_{ih})_{i\in\mathbb{N}\cap[0,T/h)},\hat{\theta}' =(\theta_{ih}')_{i\in\mathbb{N}\cap[0,T/h)}$. The velocity Verlet integrator gives
\begin{align}
Q_0'&:=\bar{q}_T - \bar{q}_T' + \gamma_1^{-1}(\bar{p}_T - \bar{p}_T')\nonumber\\
&=  x-y + \int_0^T (\bar{p}_{\floor{s}} - \bar{p}_{\floor{s}}') ds - \frac{h}{2}\int_0^T (b(\bar{q}_{\floor{s}},\theta_{\floor{s}}) - b (\bar{q}_{\floor{s}}',\theta_{\floor{s}}) ) ds \nonumber\\
&\quad + \gamma_1^{-1}\bigg(v' - w' - \frac{1}{2}\int_0^T \Big(b (\bar{q}_{\floor{s}},\theta_{\floor{s}}) - b (\bar{q}_{\floor{s}}',\theta_{\floor{s}})\nonumber\\
&\quad + b (\bar{q}_{\ceil{s}},\theta_{\floor{s}}') - b(\bar{q}_{\ceil{s}}',\theta_{\floor{s}}') \Big) ds \bigg),\label{q0}
\end{align}
so that expanding the right-hand side again 
yields
\begin{align}
Q_0'&= z + (T+\gamma_1^{-1}) (\eta (v-w) + \sqrt{1-\eta^2}\bar{G}) \nonumber\\
&\quad - \frac{1}{2}\int_0^T\int_0^{\floor{s}} 
\Big(b (\bar{q}_{\floor{r}},\theta_{\floor{r}}) - b (\bar{q}_{\floor{r}}',\theta_{\floor{r}}) \nonumber\\
&\quad + b (\bar{q}_{\ceil{r}},\theta_{\floor{r}}') - b (\bar{q}_{\ceil{r}}',\theta_{\floor{r}}')\Big) drds \nonumber\\
&\quad - \frac{h+\gamma_1^{-1}}{2} \int_0^T  \Big(b(\bar{q}_{\floor{s}},\theta_{\floor{s}}) - b (\bar{q}_{\floor{s}}',\theta_{\floor{s}}) \Big) ds \nonumber\\
&\quad - \frac{\gamma^{-1}}{2}\int_0^T \Big(b (\bar{q}_{\ceil{s}},\theta_{\floor{s}}') - b(\bar{q}_{\ceil{s}}',\theta_{\floor{s}}') \Big) ds.\label{theq}
\end{align}
For~$T\in h\mathbb{N}$, 
the double integral on the right-hand side of~\eqref{theq} may be bounded using the expressions
\begin{align*}
&\int_0^T\int_0^{\floor{s}} (\abs{\bar{q}_{\floor{r}} - \bar{q}_{\floor{r}}'} + 
\abs{\bar{q}_{\ceil{r}} - \bar{q}_{\ceil{r}}'} ) drds\\
&\quad= \int_h^T h\sum_{j = 1}^{\floor{s}/h} \Big(\abs{\bar{q}_{(j-1)h} - \bar{q}_{(j-1)h}'} + \abs{\bar{q}_{jh} - \bar{q}_{jh}'} \Big)ds\\
&\quad= h^2\sum_{i = 1}^{T/h -1}\sum_{j=1}^i  (\abs{\bar{q}_{(j-1)h} - \bar{q}_{(j-1)h}'} + \abs{\bar{q}_{jh} - \bar{q}_{jh}'}).
\end{align*}
Therefore the right-hand side of~\eqref{theq} can be bounded as 
\begin{align}
&\abs{Q_0'} - \abs{z + (T+\gamma_1^{-1}) (\eta (v-w) + \sqrt{1-\eta^2}\bar{G})} \nonumber\\
&\quad\leq \frac{Lh^2}{2}\sum_{i =1}^{T/h - 1} \sum_{j = 1}^i \abs{\bar{q}_{(j-1)h} - \bar{q}_{(j-1)h}'}  +  \frac{Lh^2}{2}\sum_{i =1}^{T/h - 1} \sum_{j = 1}^i \abs{\bar{q}_{jh} - \bar{q}_{jh}'}\nonumber\\
&\qquad + \bigg( \frac{h+\gamma_1^{-1}}{2}\bigg) Lh\sum_{i = 0}^{T/h - 1}\abs{ \bar{q}_{ih} - \bar{q}_{ih}'} + \frac{\gamma_1^{-1}}{2}Lh\sum_{i = 1}^{T/h }\abs{ \bar{q}_{ih} - \bar{q}_{ih}'}.\label{gug}
\end{align}
By~\eqref{base0}, inequality~\eqref{gug} implies a.s. that
\begin{align}
&\abs{Q_0'} - \abs{z + (T+\gamma_1^{-1}) (\eta (v-w) + \sqrt{1-\eta^2}\bar{G})} \nonumber\\
&\quad\leq  \frac{32}{31}Lh\bigg(\sum_{i = 1}^{T/h-1} h\sum_{j = 1}^i  \bigg(\abs{z} + \frac{jhK_{\hat{q}}}{T+\gamma^{-1}}\bigg)\nonumber\\
&\qquad + \bigg(\frac{h+\gamma_1^{-1}}{2}\bigg) \sum_{j=0}^{T/h-1}  \bigg[\abs{z} +\frac{jhK_{\hat{q}}}{T+\gamma^{-1}} \bigg] + \frac{\gamma_1^{-1}}{2}\sum_{j=0}^{T/h-1} \bigg[\abs{z} +\frac{(j+1)hK_{\hat{q}}}{T+\gamma^{-1}} \bigg]\bigg).\label{hfw}
\end{align}
Consequently, it holds a.s. that
\begin{align}
&L^{-1}(\abs{Q_0'} - \abs{z + (T+\gamma_1^{-1}) (\eta (v-w) + \sqrt{1-\eta^2}\bar{G})})\nonumber\\
&\quad\leq \frac{32}{31}\bigg(\frac{1}{2}T(T-h) + (h+\gamma_1^{-1})T\bigg)\abs{z}\nonumber\\
&\qquad + \frac{32}{31}\bigg(\frac{(T-h)(T+h)}{6} + \frac{hT+2\gamma_1^{-1}T}{4}\bigg)\frac{TK_{\hat{q}}}{T+\gamma^{-1}}. \label{le}
\end{align}
By substituting~$\gamma_1^{-1} = 0$ into~\eqref{le}, using~\eqref{fe} with~$t=T$ and using~$(T-h)(T+h)/6+Th/4 \leq T(T+h)/4$, 
we obtain~\eqref{base1}. By instead substituting~$\gamma_1^{-1} = \gamma^{-1}$ into~\eqref{le}, 
we obtain~\eqref{base2}. 
Now for the other case~$u=1$ (the randomized midpoint algorithm), we may replace~\eqref{base0} by the fact that for~$t\in[0,T-h]\cap h\mathbb{N}$, it holds that
\begin{align*}
&\abs{\bar{q}_t-\bar{q}_t' + u_th(\bar{p}_t-\bar{p}_t')}\\
&\quad\leq \abs{\bar{q}_t-\bar{q}_t' + u_th(\bar{p}_t-\bar{p}_t' - (h/2)(b(\bar{q}_t+u_th\bar{p}_t,\theta_t) \\
&\qquad-b(\bar{q}_t'+u_th\bar{p}_t',\theta_t))) } + (Lh^2/2)\abs{\bar{q}_t -\bar{q}_t' + u_th\bar{p}_t - u_th\bar{p}_t'}\\
&\quad = \abs{\bar{q}_{t+u_{\floor{t}}h}-\bar{q}_{t+u_{\floor{t}}h}'} + (1/64)\abs{\bar{q}_t -\bar{q}_t' + u_th\bar{p}_t - u_th\bar{p}_t'},
\end{align*}
which, by~\eqref{base0}, yields
\begin{align}
\abs{\bar{q}_t-\bar{q}_t' + u_th(\bar{p}_t-\bar{p}_t')} &\leq (64/63)(514/511)(\abs{z} + (t+u_{\floor{t}}h)K_{\hat{q}}/(T+\gamma^{-1}))\nonumber\\
&\leq (32/31)(\abs{z} + (t+u_{\floor{t}}h)K_{\hat{q}}/(T+\gamma^{-1})).\label{ghai}
\end{align}
The proof then follows in the same way for the randomized midpoint integrator as for the velocity Verlet integrator. 
Namely, we obtain~\eqref{hfw} with the two square brackets replaced by~$\abs{z} + (j+u_{jh})hK_{\hat{q}}/(T+\gamma^{-1})$. This implies~\eqref{le} if~$\abs{Q_0'}$ is replaced by~$\mathbb{E}[\abs{Q_0'}|x,y,v,w,G,\hat{G}]$. Otherwise, using~$u_{jh}\leq 1$ for all~$j\in\mathbb{N}$ and~$T+h\leq 2T$, the same bound~\eqref{le} holds with double the right-hand side.
\end{proof}

\subsection{Contraction}\label{contractionsec}

In the following, we consider
\begin{equation}\label{gdef}
\gamma = \frac{1-\eta}{\eta T},
\end{equation}
which implies~$\hat{q} \propto q$ in~\eqref{zdefs}. In the case where~$\eta = e^{-\bar{\gamma}T}$ for some~$\bar{\gamma}>0$, this choice of~$\gamma$ is to enforce~$1+\gamma T = e^{\bar{\gamma}T}$. Here, in the continuous time limit~$T\rightarrow 0$,~\eqref{cou1} corresponds to the reflection coupling in~\cite{MR3980913}. In the opposing case where~$\eta = 0$, this reduces in~\eqref{zdefs} to~$q=T\hat{q} = z$ and the coupling~\eqref{cou1} agrees substantially to that considered in~\cite[equation~(21)]{MR4133372}. If instead~$\gamma T = 1-\eta$, we still have~$2T\hat{q} = z$ when~$\eta=0$ but the range of~\eqref{Gdef2} is reduced. 

For~$\gamma\in(0,\infty]$ given by~\eqref{gdef}, the assertions of Lemma~\ref{base} imply
\begin{subequations}\label{wvdef}
\begin{align}
\mathbb{E}[\abs{Z'}|x,y,v,w,G,\hat{G}] &\leq  \eta \abs{z} +(1-\eta)K_{\hat{q}}+ (16/31)LT(T+h)\zeta,\\
\mathbb{E}[\abs{Q'}|x,y,v,w,G,\hat{G}] &\leq K_{\hat{q}} + (32/31)L(\gamma^{-1}T + T(T+h)/2)\zeta,
\end{align}
\end{subequations}
where~$z,\hat{q}$ are given by~\eqref{zdefs}, 
\begin{equation}\label{dhat}
\zeta = \abs{z} + (1-\eta)K_{\hat{q}}/2.
\end{equation}

\noindent
For any~$\bar{\alpha}>0$, inequalities~\eqref{wvdef} imply
\begin{equation}\label{omme}
\mathbb{E}[\abs{Q'}+\bar{\alpha}\abs{Z'}|x,y,v,w,G,\hat{G}] \leq C_qK_{\hat{q}} + C_z\bar{\alpha}\abs{z},
\end{equation}
where~$C_q,C_z$ are given by
\begin{align}
C_q &= 1+\bar{\alpha}(1-\eta) + (16/31)(1-\eta)L(\gamma^{-1}T+(1+\bar{\alpha})T(T+h)/2),\label{gcdef}\\
C_z &= \eta + (32/31)\bar{\alpha}^{-1}L(\gamma^{-1}T + T^2 + \bar{\alpha}T(T+h)/2 ).\label{erg}
\end{align}
If~$C_q,C_z <1$ hold, then even with synchronous coupling,~\eqref{omme} implies contraction. However, the former constant~\eqref{gcdef} always satisfies~$C_q>1$. We will choose~$\bar{\alpha}$ later so that~$C_z\leq1-cLT^2/(1-\eta)$ and~$C_q\leq1+c^{-1}LT^2/(1-\eta)$ for some constant~$c>0$. This expression for~$C_z$ ensures contraction in the~$z$ direction. On the other hand, we take advantage of the reflection coupling in~$K_{\hat{q}}$ defined in~\eqref{kdef} in order to deal with the fact that~$C_q>1$. As usual, a suitable concave function is required to prove contraction for reflection coupling, which we introduce next.

Let~$g,\hat{R}>0$ and let~$f_0:[0,\infty)\rightarrow[0,\infty)$ be given by
\begin{equation}\label{f0def}
f_0(x)=\int_0^{\min(x,\hat{R})}e^{-gs}ds.
\end{equation}
For~$x,v,y,w\in\mathbb{R}^d$, if~$\abs{q}+\bar{\alpha}\abs{z}\leq\hat{R}$ holds, then inequality~\eqref{omme} implies by the tower property and Jensen's inequality that
\begin{align}
&\mathbb{E}[(f_0(\abs{Q'}+\bar{\alpha}\abs{Z'}) - f_0(\abs{q}+\bar{\alpha}\abs{z}))]\nonumber\\
&\quad\leq g^{-1}\mathbb{E}[(e^{-g(\abs{q}+\bar{\alpha}\abs{z})}-e^{-g(\abs{Q'}+\bar{\alpha}\abs{Z'})})]\nonumber\\
&\quad=g^{-1}e^{-g(\abs{q}+\bar{\alpha}\abs{z})}\mathbb{E}[(1-e^{-g(\abs{Q'}+\bar{\alpha}\abs{Z'}-\abs{q}-\bar{\alpha}\abs{z})})]\nonumber\\
&\quad\leq g^{-1}e^{-g(\abs{q}+\bar{\alpha}\abs{z})} (1-e^{- g(C_z-1)\bar{\alpha}\abs{z} - g(C_q-1)\abs
{q}}\mathbb{E}[e^{-gC_q(K_{\hat{q}}-\abs{q})}]).\label{uexp}
\end{align}
The main idea of the proof is to write down the explicit integral form of the unresolved expectation appearing on the right-hand side of~\eqref{uexp} and to estimate it. This is the content of the following Lemma~\ref{lemexes}. The constants involved in the calculations of Lemma~\ref{lemexes} are left quite general and their specific choices leading to the announced results are justified afterwards.
\begin{lemma}\label{lemexes}
Let~$\hat{T}=T\sqrt{1-\eta^2}/(1-\eta)$,~$r^*\in(0,\hat{T}^{-2}]$,~$G\sim\mathcal{N}(0,I_d)$ and let~$K_{\cdot}$ be given by~\eqref{kdef} and~\eqref{Gdef2} with some arbitrary~$\hat{q}\in\mathbb{R}^d$ and its unit vector~$e$. 
Let~$r=r(\hat{q}) = \abs{\hat{q}}/(r^*(T+\gamma^{-1})\sqrt{1-\eta^2})$. For any $\hat{g}>0$,~$C>1$,~$0<c_1<1$,~$c_2>0$, it holds that
\begin{align*}
&\mathbb{E}[e^{-\hat{g}C(K_{\hat{q}}-r(\hat{q}))}]\\
&\quad\geq \begin{cases}
1+ 4c_1\hat{g}\hat{T}\int_{-c_2+\log(c_1)/(2\hat{g}\hat{T})}^{-c_2}\Phi(u)du (e^{\hat{g}r^*\hat{T}^2r} -1 )& \textrm{if } 0\leq r < 2c_2/(r^*\hat{T})\\
1+ 4c_1\hat{g}\hat{T}\int_{-c_2+
\log(c_1)/(2\hat{g}\hat{T})}^{-c_2}\Phi(u)du (e^{2c_2\hat{g}C\hat{T}} -1 ) 
&\textrm{otherwise}.
\end{cases}
\end{align*}
\end{lemma}
\begin{proof}
Let~$g_0:\mathbb{R}^d\rightarrow[0,\infty)$ be given by~$g_0(\hat{q}) = \mathbb{E}[e^{-\hat{g}C(K_{\hat{q}}-r(\hat{q}))}]$ for all~$\hat{q}$. 
By definition of~$K_{\cdot}$, it holds that
\begin{align}
g_0(r)
&= \mathbb{E}[e^{-\hat{g}C(K_{\hat{q}}-r)}(\mathds{1}_{K_{\hat{q}}=(1-r^*\hat{T}^2)r} + \mathds{1}_{K_{\hat{q}}\neq(1-r^*\hat{T}^2)r})] \nonumber\\
&= 2e^{\hat{g}Cr^*\hat{T}^2r}\Phi\bigg(-\frac{r^*\hat{T}r}{2}\bigg) + \frac{1}{\sqrt{2\pi}}\int_{-\frac{r^*\hat{T}r}{2}}^{\infty} e^{-2\hat{g}C\hat{T}u}\Big[e^{-\frac{u^2}{2}} - e^{-\frac{(u+r^*\hat{T}r)^2}{2}}\Big]du.\label{uexpo}
\end{align}
The integral associated to the last term in the square brackets on the right-hand side of~\eqref{uexpo} has the form
\begin{align*}
&\frac{1}{\sqrt{2\pi}}\int_{-\frac{r^*\hat{T}r}{2}}^{\infty}e^{-2\hat{g}C\hat{T}u-\frac{(u+r^*\hat{T}r)^2}{2}} du\\
&\quad= \frac{1}{\sqrt{2\pi}}\int_{-\frac{r^*\hat{T}r}{2}}^{\infty}e^{-\frac{1}{2}(u+2\hat{g}C\hat{T}+r^*\hat{T}r)^2 - \frac{1}{2}(r^*\hat{T}r)^2 + \frac{1}{2}(2\hat{g}C\hat{T} + r^*\hat{T}r)^2} du \\
&\quad= e^{2(\hat{g}C\hat{T})^2 + 2\hat{g}Cr^*\hat{T}^2r}  \Phi(-2\hat{g}C\hat{T} - r^*\hat{T}r/2)),
\end{align*}
so that taking the derivative of~$g_0$ yields
\begin{align*}
\partial_{r}g_0(r) &= 2\hat{g}Cr^*\hat{T}^2e^{\hat{g}Cr^*\hat{T}^2r}\Phi\bigg(-\frac{r^*\hat{T}r}{2}\bigg) - \frac{r^*\hat{T}e^{\hat{g}Cr^*\hat{T}^2r}}{\sqrt{2\pi}}e^{-\frac{(r^*\hat{T}r)^2}{8}} \nonumber\\
&\quad + \frac{r^*\hat{T}}{2\sqrt{2\pi}}e^{\hat{g}Cr^*\hat{T}^2r-\frac{(r^*\hat{T}r)^2}{8}} - e^{2(\hat{g}C\hat{T})^2 + 2\hat{g}Cr^*\hat{T}^2r} \nonumber\\
&\quad \cdot\bigg(2\hat{g}Cr^*\hat{T}^2 \Phi\bigg(-2\hat{g}C\hat{T} -\frac{r^*\hat{T}r}{2}\bigg) - \frac{r^*\hat{T}}{2\sqrt{2\pi}}e^{-\frac{1}{2}(2\hat{g}C\hat{T}+\frac{r^*\hat{T}r}{2})^2}\bigg),
\end{align*}
where the terms without~$\Phi$ factors cancel each other to give
\begin{align}
\partial_r g_0(r)&= 2\hat{g}Cr^*\hat{T}^2e^{\hat{g}Cr^*\hat{T}^2r}\Phi\bigg(-\frac{r^*\hat{T}r}{2}\bigg)\nonumber\\
&\quad- 2\hat{g}Cr^*\hat{T}^2 e^{2(\hat{g}C\hat{T})^2 + 2\hat{g}Cr^*\hat{T}^2r} \Phi\bigg(-2\hat{g}C\hat{T} - \frac{r^*\hat{T}r}{2}\bigg)\label{uexp2}
\end{align}
for all~$r$. For the last term on the right-hand side of~\eqref{uexp2}, it holds that
\begin{align*}
&e^{2(\hat{g}C\hat{T})^2+2\hat{g}Cr^*\hat{T}^2r}\Phi\bigg(-2\hat{g}C\hat{T}-\frac{r^*\hat{T}r}{2}\bigg) \nonumber\\
&\quad= \frac{e^{\hat{g}Cr^*\hat{T}^2r}}{\sqrt{2\pi}} \int_{-\infty}^{-\frac{r^*\hat{T}r}{2}}e^{-\frac{u^2}{2} + 2u\hat{g}C\hat{T} + \hat{g}Cr^*\hat{T}^2r}du,
\end{align*}
which, substituting back into~\eqref{uexp2}, implies
\begin{equation}\label{exgra}
\partial_r g_0(r) = \frac{2\hat{g}Cr^*\hat{T}^2e^{\hat{g}Cr^*\hat{T}^2r}}{\sqrt{2\pi}}\int_{-\infty}^{-\frac{r^*\hat{T}r}{2}}e^{-\frac{u^2}{2}}(1-e^{\hat{g}C\hat{T}(2u+r^*\hat{T}r)})du. 
\end{equation}
Using~$-e^{Cu}\geq -e^{u}$ for~$u\leq 0$, then integrating-by-parts on the right-hand side of~\eqref{exgra}, yields for any~$0<c_1<1$ that
\begin{align}
\partial_r g_0(r) &\geq 2\hat{g}r^*\hat{T}^2 e^{\hat{g}r^*\hat{T}^2r}\int_{-\infty}^{-\frac{r^*\hat{T}r}{2}} 2\hat{g}\hat{T}\Phi(u) e^{\hat{g}\hat{T}(2u+r^*\hat{T}r)}du\nonumber\\
&\geq 4c_1\hat{g}^2r^*\hat{T}^3 e^{\hat{g}r^*\hat{T}^2r} \int_{-\frac{r^*\hat{T}r}{2} + \frac{\log(c_1)}{2\hat{g}\hat{T}}}^{-\frac{r^*\hat{T}r}{2}} \Phi(u) du.\label{hapw}
\end{align}
Estimating the integral on the right-hand side of~\eqref{hapw} gives for any~$c_2>0$ that
\begin{equation*}
\partial_r g_0(r) \geq 
\begin{cases}
4c_1\hat{g}^2r^*\hat{T}^3 e^{\hat{g}r^*\hat{T}^2r} \int_{-c_2+\log(c_1)/(2\hat{g}\hat{T})}^{-c_2}\Phi(u)du &\textrm{if } 0\leq r < 2c_2/(r^*\hat{T})\\
0 &\textrm{otherwise}, 
\end{cases}
\end{equation*}
which concludes by integrating w.r.t.~$r$.

\end{proof}
Concerning the choice of~$r^*$ with respect to small~$\hat{T}=T\sqrt{1-\eta^2}/(1-\eta)$, it makes sense to set~$r^*\sim O(1/\hat{T})$ or~$r^*\sim O(1/\hat{T}^2)$. However,~$r^*\sim O(1)$ is a bad choice because the bound in Lemma~\ref{lemexes} gives in that case~$\mathbb{E}[e^{-\hat{g}C(K_{\hat{q}}-r)}]-1\sim O(\hat{T}^3)$, unless~$\hat{g}$ is chosen to depend on~$\hat{T}$, which would lead to an exponentially poor dependence on~$\hat{T}$ in the~$\mathcal{W}_1$ contraction prefactor. In fact, the right-hand bound in Lemma~\ref{lemexes} is an increasing function in~$r^*$ at every fixed~$r$, so we take the maximum possible value~$r^*=1/\hat{T}^2$. 

\begin{theorem}\label{main2}
Let~$\hat{R}>0$, let~$g$ be given by
\begin{equation}\label{gdef??}
g=(2/5)\max(16L\hat{R},2\sqrt{L}),
\end{equation}
let~$x,y,v,w,G,\hat{G},\hat{q},z,Q',Z'$ be as in Lemma~\ref{base} 
and let~$q$ be given by~\eqref{zdefs}. Let~$r^*=(1-\eta)/(T^2(1+\eta))$, let~$\mathcal{U}\sim\mathcal{U}(0,1)$ be independent of~$x,y,v,w,G,\hat{G},(u_{kh})_{k\in\mathbb{N}}$ and assume~$\hat{G}$ satisfies~\eqref{cou1} with~\eqref{edef}. 
Moreover, let~$f_0$ be given by~\eqref{f0def} and~$\alpha = LT^2/(1-\eta)^2$. 
Suppose~\eqref{A1a} holds 
and assume~\eqref{carot}. 
If~$\abs{q}+1.09\alpha\abs{z}\leq \hat{R}$ holds, then it holds a.s. that
\begin{align*}
&\mathbb{E}[f_0(\abs{Q'}+1.09\alpha\abs{Z'})|x,y,v,w]\\
&\quad\leq \bigg(1- \frac{c_0LT^2}{(1-\eta)}\bigg) f_0(\abs{q}+1.09\alpha\abs{z}) - \frac{ge^{-g\hat{R}}T^2(1+\eta)}{32(1-\eta)}\cdot\min\bigg(\abs{\hat{q}},\frac{67}{50}\bigg),
\end{align*}
where~$c_0$ is given by
\begin{equation}\label{cor}
c_0 = \frac{g\hat{R}e^{-g\hat{R}}}{5(1-e^{-g\hat{R}})}\cdot\min\bigg(\frac{1}{4\alpha},1\bigg).
\end{equation}
\end{theorem}
\begin{proof}
Let~$\hat{T}=T\sqrt{1+\eta}/\sqrt{1-\eta}$ and~$\bar{\alpha}=c_{\alpha}LT^2/(1-\eta)^2$ for some~$c_{\alpha}>0$. 
Moreover, let~$c_0=1/16$ and let~$C_q,C_z$ be given by~\eqref{gcdef} and~\eqref{erg}. It holds that 
\begin{equation}\label{cq1i0}
0< C_q-1 \leq \frac{c_{\alpha}LT^2}{(1-\eta)} + \frac{16L}{31}(1-\eta)\bigg(\frac{\eta T^2}{1-\eta}+ T^2 + \frac{c_{\alpha}LT^3(T+h)}{2(1-\eta)^2}\bigg).
\end{equation}
By~\eqref{carot}, the last term on the right-hand side of~\eqref{cq1i0} satisfies
\begin{align*}
&\frac{16L}{31}(1-\eta)\bigg(\frac{\eta T^2}{1-\eta}+ T^2 + \frac{c_{\alpha}LT^3(T+h)}{2(1-\eta)^2}\bigg) \leq \frac{16L}{31}\bigg(T^2+ \frac{c_0^2c_{\alpha}T^2}{2(1-\eta)}\bigg),
\end{align*}
so that substituting back into~\eqref{cq1i0} gives
\begin{equation}\label{cq1i}
0\leq C_q-1 \leq C^*L\hat{T}^2,
\end{equation}
where~$C^*$ is given by
\begin{equation}\label{cstard}
C^* = c_{\alpha} + (16/31)(1+c_0^2c_{\alpha}/2).
\end{equation}
Moreover, let~$c_3=2/5$, then by definition~\eqref{gdef??} 
of~$g$,~\eqref{carot} and the assumptions on~$T,\eta$, it holds that
\begin{equation}\label{gtbma}
g\hat{T}=c_3\max(16\hat{R}L\hat{T},2\sqrt{L}\hat{T}) \leq c_3.
\end{equation}
Let~$c_1\in(0,1)$ and~$c_2\in(0,\infty)$. Given~\eqref{gtbma}, the prefactor appearing in the statement of Lemma~\ref{lemexes} may be approximated as
\begin{equation}\label{nuea}
4c_1\int_{-c_2+\frac{\log(c_1)}{2g\hat{T}}}^{-c_2} \Phi(u) du \geq 4c_1\int_{-c_2+\frac{\log(c_1)}{2c_3}}^{-c_2} \Phi(u) du =: \Phi^*.
\end{equation}
Note that we may integrate-by-parts the expression for~$\Phi^*$ in order to obtain a more explicit form. Lemma~\ref{lemexes} with~\eqref{nuea}, the assumption~$r^*\hat{T}^2=1$ and~$e^x\geq1+x$ \begin{align}
&e^{-g(C_q-1)\abs{q}}\mathbb{E}[e^{-gC_q(K_{\hat{q}} - \abs{q})}|x,y,v,w]\nonumber\\
&\quad\geq \begin{cases}
e^{-g(C_q-1)\abs{q}}(1+ \Phi^* g^2\hat{T}r) & \textrm{if }\abs{q}\leq 2c_2\hat{T}\\
e^{-g(C_q-1)\abs{q}}(1+ 2c_2\Phi^* g^2\hat{T}^2) & \textrm{if } 2c_2\hat{T} < r\leq \hat{R}.
\end{cases}\label{caz}
\end{align}
By the bound~\eqref{cq1i} on~$C_q-1$, inequality~\eqref{caz} implies for any~$\abs{q}\leq \hat{R}$ that
\begin{align}
&e^{-g(C_q-1)\abs{q}}\mathbb{E}[e^{-gC_q(K_{\hat{q}} - \abs{q})}|x,y,v,w]\nonumber\\
&\quad\geq \begin{cases}
(1-C^*L\hat{T}^2g\abs{q})(1+ \Phi^*g^2\hat{T}\abs{q}) & \textrm{if }\abs{q}\leq 2c_2\hat{T}\\
(1-C^*L\hat{T}^2g\abs{q})(1+ 2c_2\Phi^*g^2\hat{T}^2) & \textrm{if } 2c_2\hat{T} < \abs{q}\leq \hat{R}.
\end{cases}\label{topt3}
\end{align}
For the case~$\abs{q}\leq 2c_2\hat{T}$, 
the right-hand side of~\eqref{topt3} has the form
\begin{equation}\label{diwi1}
(1-C^*L\hat{T}^2g\abs{q})(1+ \Phi^*g^2\hat{T}\abs{q}) = 1 - C^*L\hat{T}^2g\abs{q} + \Phi^*g^2\hat{T}\abs{q} - C^*\Phi^*Lg^3\hat{T}^3\abs{q}^2.
\end{equation}
The positive term on the right-hand side of~\eqref{diwi1} may be rewritten as
\begin{align*}
g^2\hat{T}\abs{q} &= (1 + 1/8 - 1/8 -c_3C^*/16 + c_3C^*/16 )g^2\hat{T}\abs{q}\\ 
&\geq g^2\hat{T}\abs{q}/8 + 2c_3(7/8 -c_3C^*/16)\sqrt{L}g\hat{T}\abs{q} + c_3^2C^*L\hat{R}g\hat{T}\abs{q}, 
\end{align*}
which may be substituted back into~\eqref{diwi1}, then used with~\eqref{gtbma} and our assumptions on~$T,\eta$, to produce
\begin{align}
&(1-C^*L\hat{T}^2g\abs{q})(1+ \Phi^*g^2\hat{T}\abs{q}) \nonumber\\
&\quad\geq 1 - C^*L\hat{T}^2g\abs{q} + \Phi^*g^2\hat{T}\abs{q}/8 + c_3(7/4-c_3C^*/8)\Phi^*\sqrt{L}g\hat{T}\abs{q} \nonumber\\ 
&\qquad + c_3^2C^*\Phi^*L\hat{R}g\hat{T}\abs{q} - c_3^2C^*\Phi^*Lg\hat{T}\abs{q}^2\nonumber\\
&\quad\geq 1 - (C^*-(c_3/c_0)(7/4-c_3C^*/8)\Phi^*)L\hat{T}^2gr + \Phi^*g^2\hat{T}\abs{q}/8. \label{topt1} 
\end{align}
In the other case~$2c_2\hat{T} \leq \abs{q} \leq \hat{R}$, the right-hand side of~\eqref{topt3} satisfies
\begin{align}
&(1-C^*L\hat{T}^2g\abs{q})(1+ 2c_2\Phi^*g^2\hat{T}^2)\nonumber\\
&\quad= 1 - C^*L\hat{T}^2g\abs{q} + 2c_2\Phi^*g^2\hat{T}^2 - 2C^*c_2\Phi^*Lg^3\hat{T}^4\abs{q}.\label{jdj2}
\end{align}
By~\eqref{gtbma}, the last term on the right-hand side of~\eqref{jdj2} has the bound~$
Lg^3\hat{T}^4\abs{q}\leq c_3^2L g\hat{T}^2\abs{q}$. 
Therefore~\eqref{jdj2} together with~\eqref{gdef??} implies
\begin{align}
&(1-C^*L\hat{T}^2g\abs{q})(1+ 2c_2\Phi^*g^2\hat{T}^2)\nonumber\\ 
&\quad\geq 1-C^*(1+2c_2c_3^2\Phi^*)L\hat{T}^2g\abs{q} + 2c_2\Phi^*g^2\hat{T}^2\nonumber\\
&\quad\geq 1-(C^*(1+2c_2c_3^2\Phi^*)\abs{q} - 28c_2c_3\Phi^*\hat{R})L\hat{T}^2g + c_2\Phi^*g^2\hat{T}^2/4\nonumber\\
&\quad\geq 1-(C^*(1+2c_2c_3^2\Phi^*) - 28c_2c_3\Phi^*)L\hat{T}^2g\abs{q} + c_2\Phi^*g^2\hat{T}^2/4.\label{topt2}
\end{align}
Gathering~\eqref{topt1},~\eqref{topt2} and~\eqref{topt3} gives for all~$0\leq \abs{q}\leq \hat{R}$ that
\begin{align}
&e^{-g(C_q-1)\abs{q}}\mathbb{E}[e^{-gC_q(K_{\hat{q}} - \abs{q})}|x,y,v,w]\nonumber\\
&\quad\geq 1+\min\bigg( \frac{c_3}{c_0}\bigg(\frac{7}{4}-\frac{C^*}{8}\bigg)\Phi^*-C^*, 28c_2c_3\Phi^* - C^*(1+2c_2c_3^2\Phi^*) \bigg) 
L\hat{T}^2g\abs{q}\nonumber\\
&\qquad+ \Phi^*g^2\hat{T}^2\min(\abs{\hat{q}},2c_2)/8.\label{twk1}
\end{align}
On the other hand, for the contraction rate associated to~$\abs{z}$ arising from~\eqref{uexp}, it holds by definition~\eqref{erg} of~$C_z$ and~$\bar{\alpha}=c_{\alpha}LT^2/(1-\eta)^2$ 
that 
\begin{equation}\label{yda0}
C_z-1 = -(1-\eta) + (32/31) ( \eta (1-\eta)/c_{\alpha} + (1-\eta)^2/c_{\alpha} + LT(T+h)/2).
\end{equation}
Inequality~\eqref{yda0} implies for all~$z\in\mathbb{R}^d$ that
\begin{align}
C_z-1&\leq -(1-\eta) + (32/31) (1-\eta)/c_{\alpha} + (16/31)c_0^2(1-\eta)\nonumber\\
&= \bigg(\frac{4\cdot32}{31c_{\alpha}}-4\bigg(1-\frac{16c_0^2}{31}\bigg)\bigg)\cdot\frac{(1-\eta)^2}{4LT^2}\cdot\frac{LT^2}{1-\eta},\label{yda}
\end{align}
which implies
\begin{equation*}
e^{-g(C_z-1)\bar{\alpha}\abs{z}} \geq 1 - \bigg(\frac{128}{31c_{\alpha}}-4\bigg(1-\frac{16c_0^2}{31}\bigg)\bigg)\cdot\frac{1}{4\alpha}\cdot\frac{LT^2}{1-\eta}g\bar{\alpha}\abs{z},
\end{equation*}
and, together with~\eqref{twk1}, that
\begin{align}
&1 - e^{-g(C_z-1)\bar{\alpha}\abs{z}-g(C_q-1)\abs{q}}\mathbb{E}[e^{-gC_q(K_{\hat{q}} - \abs{q})}|x,y,v,w]\nonumber\\
&\quad\leq -\bigg[\bigg(4-\frac{64c_0^2}{31} - \frac{128}{31c_{\alpha}}\bigg) \frac{\bar{\alpha}\abs{z}}{4\alpha} + \min\bigg( \frac{c_3}{c_0}\bigg(\frac{7}{4}-\frac{C^*}{8} \bigg)\Phi^* - C^*, 28c_2c_3\Phi^*  \nonumber\\
&\qquad - C^*(1+2c_2c_3^2\Phi^*) \bigg)(1+\eta)\abs{q} \bigg]\cdot\frac{gLT^2}{1-\eta} - \frac{\Phi^*g^2\hat{T}^2}{8}\min(\abs{\hat{q}},2c_2).\label{nopt}
\end{align}
Recall the definitions~$c_0=1/16$,~$c_3=2/5$,~\eqref{cstard} for~$C^*$ and~\eqref{nuea} for~$\Phi^*$. Numerically optimizing the expression in the square brackets on the right-hand side of~\eqref{nopt} with respect to~$c_1\in(0,1)$,~$c_2,c_{\alpha}\in(0,\infty)$ and with~$\eta=0$,~$4\alpha=1$ suggests the values
\begin{equation}\label{tvals}
c_1 = 0.58,\qquad c_2 = 0.67 ,\qquad c_{\alpha} = 1.09.
\end{equation}
For the choice~\eqref{tvals}, inequality~\eqref{nopt} implies
\begin{align}\label{tvc}
&1 - e^{-g(C_z-1)\bar{\alpha}\abs{z}-g(C_q-1)\abs{q}}\mathbb{E}[e^{-gC_q(K_{\hat{q}} - \abs{q})}|x,y,v,w] \nonumber\\
&\quad\leq -\frac{(\abs{q}+\bar{\alpha}\abs{z})gLT^2}{5(1-\eta)}\min\bigg(\frac{1}{4\alpha},1\bigg) - \frac{1}{32}g^2\hat{T}^2\min\bigg(\frac{r}{\hat{T}},2c_2\bigg).
\end{align}
Substituting~\eqref{tvc},~\eqref{tvals} 
into the right-hand side of~\eqref{uexp} yields
\begin{align*}
&\mathbb{E}[f_0(\abs{Q'}+\bar{\alpha}\abs{Z'}) - f_0(\abs{q}+\bar{\alpha}\abs{z})|x,y,v,w]\\
&\quad\leq -e^{-g(\abs{q}+\bar{\alpha}\abs{z})}[L(\abs{q}+\bar{\alpha}\abs{z})\min(1/(4\alpha),1)/5 \\
&\qquad+ g(1+\eta)\min(\abs{\hat{q}},67/50)/32]T^2/(1-\eta)\\
&\quad= -\frac{g(\abs{q}+\bar{\alpha}\abs{z}) e^{-g(\abs{q}+\bar{\alpha}\abs{z})}}{1-e^{-g(\abs{q}+\bar{\alpha}\abs{z})}}f_0(\abs{q}+\bar{\alpha}\abs{z})\cdot \frac{LT^2}{10(1-\eta)}\cdot\min\bigg(\frac{1}{4\alpha},1\bigg)\\
&\qquad- ge^{-g(\abs{q}+\bar{\alpha}\abs{z})}\cdot\frac{T^2(1+\eta)}{32(1-\eta)}\cdot\min\bigg(\abs{\hat{q}},\frac{67}{50}\bigg),
\end{align*}
which implies the assertion.
\end{proof}

\section{Convex region and large velocity differences}\label{convex}
Here, a Lyapunov function is used to obtain contractivity outside the region considered previously. The core argument is inspired by those of~\cite{lc23}, where synchronous coupling is used to obtain contractivity in modified Euclidean norms for convex potentials. In particular, we make use of the eigenvalue argument there that, in essence, allows one to
\begin{itemize}
\item consider any~$b(x,\theta)\cdot v$ and~$x\cdot v$ terms in the same regard when obtaining a Lyapunov inequality,
\item remove dependence of the stepsize on the strength of convexity at infinity ($m$ in Assumption~\ref{A1}) and on the friction.
\end{itemize}
The price for using this approach is that firstly~$\nabla b(x,\theta)$ is assumed to be symmetric as in Assumption~\ref{A1} and secondly that a lower bound on the value of a friction-like parameter 
is imposed. It was shown in~\cite[Proposition~4]{monmarche2021sure} that a lower bound of this kind is necessary in non-Gaussian convex cases when obtaining Wasserstein contraction corresponding to a modified Euclidean distance via synchronous couplings. As a consequence, we do not expect in general a convergence rate at an order better than~$O(m)$. 
To our knowledge, the improved rate of~$O(\sqrt{m})$ has only been shown to hold in Gaussian cases~\cite{gouraud2023hmc} or in continuous time, convex cases~\cite{Cao_2023}. In both instances, it is obtained by optimizing friction to values past the restrictions placed in this section. See also the recent work~\cite{pmlr-v195-zhang23a} on this matter. 

More concretely, in this section we consider 
the same dynamics as in~\eqref{ous}, but setting in~\eqref{ou1} almost surely~$\hat{G}=G$. 
In addition, for~$z\in\mathbb{R}^{2d}$ and any positive definite matrix~$M\in\mathbb{R}^{2d\times 2d}$, let~$\|z\|_M=\sqrt{z^{\top}Mz}$. To describe the approach in more detail, we use the notation that~$\mathcal{O}$ denotes the autoregressive velocity step~$v\leftarrow \eta v + \sqrt{1-\eta^2}G$, 
where~$G$ is a standard~$d$-dimensional Gaussian variable,~$\mathcal{A}$ denotes a half Hamiltonian step in position given by~$x\leftarrow x + \frac{h}{2}v$ 
and~$\mathcal{B}$ denotes a half Hamiltonian step in velocity given by~$v\leftarrow v - \frac{h}{2} b(x,\theta)$. The strategy in this section is to study the individual changes in modified norm after the steps~$\mathcal{AB}$ and~$\mathcal{BA}$ respectively, at first without the~$\mathcal{O}$ step involved. It will be shown in Proposition~\ref{lab} that the procedure~$\mathcal{AB}$, where~$\mathcal{A}$ is applied first, does not contract or expand the modified norm, up to a factor in solely the velocity variable. This is consistent with~\cite[Theorem~4.9]{lc23}, where there is a restriction on the stepsize with respect to the friction parameter, in contrast to the lack thereof here. In Section~7.2 of the aforementioned work, the large friction limit of~$\mathcal{OAB}$ is explained to be a random walk. As a consequence, since there is no analogous upper bound restriction on the friction here, nor an upper bound on the stepsize with respect to the friction parameter, we would not expect a contraction or expansion in the modified norm up to factors in the velocity variable, just as for synchronous couplings of the random walk. On the other hand, it will be shown in Proposition~\ref{lba} that~$\mathcal{BA}$, where~$\mathcal{B}$ is applied first, does give a contraction in modified norm up to factors in the velocity variable. This contraction will be of the same order (in every parameter in the algorithm) as those obtained for~`$\mathcal{OBABO}$' and~`$\mathcal{BAOAB}$' in~\cite{lc23}, with qualitatively the same restrictions on the stepsize. The factor in the velocity variable that results from obtaining these contraction results is then eliminated at the application of the~$\mathcal{O}$ step. Putting together these results from Section~\ref{lpd} gives a contraction for~$\mathcal{O}(\mathcal{BAAB})^K$ as required.

In the nonconvex case, given any two initial positions arbitrarily far apart, the points at which~$\nabla\!_xb(\cdot,\theta)$ is relevant in the algorithm will appear inside the ball of radius~$R$ anyway for certain (large) initial velocity differences. In these instances, any assumption that initial positions are far apart 
is not useful. Instead, we rely on the difference between the initial velocities being relatively large and thus that the impact of the change in factor for the velocity variable is large. This is the content of Section~\ref{rlvd}.

\subsection{Large position differences}\label{lpd}
As mentioned, Propositions~\ref{lba} and~\ref{lab} show contraction up to velocity factors in modified norm for~$\mathcal{BA}$ and~$\mathcal{AB}$ respectively. They make use of the assumption that position differences are large through inequality~\eqref{A1b}. 
All of the proofs in this section may be found in 
Appendix~\ref{conapp}. 
\begin{prop}\label{lba}
Let Assumption~\ref{A1} hold, let~$\eta_0,\eta_1\in[0,1]$ satisfy~$0<\eta_0-\eta_1\leq\frac{1}{2}$ and~$c\in\mathbb{R}$,~$M\in\mathbb{R}^{2d\times 2d}$ be given by
\begin{equation}\label{mdef}
c=\frac{m\eta_0h^2}{16(\eta_0 - \eta_1)},\qquad M = \begin{pmatrix} 1 & \frac{h}{2(\eta_0-\eta_1)}\\ \frac{h}{2(\eta_0-\eta_1)} & \frac{h^2}{2(\eta_0-\eta_1)^2}\end{pmatrix}\otimes I_d.
\end{equation}
Assume~$
Lh^2\leq (\eta_0-\eta_1)^2$. It holds that
\begin{align}
&\bigg\|\begin{pmatrix} 
I_d & 0\\
0 & \eta_1 I_d
\end{pmatrix}
\begin{pmatrix}
\bar{q}-\bar{q}' + \frac{h}{2}(\bar{p}-\bar{p}') - \frac{h^2}{4}(b(\bar{q},\theta) - b(\bar{q}',\theta))\\
\bar{p} - \bar{p}' - \frac{h}{2}(b(\bar{q},\theta) - b(\bar{q}',\theta))
\end{pmatrix} \bigg\|_M^2 \nonumber \\
&\quad\leq (1-c)\bigg\| \begin{pmatrix}
I_d & 0\\
0 & \eta_0 I_d
\end{pmatrix}
\begin{pmatrix}
\bar{q} - \bar{q}'\\
\bar{p} - \bar{p}'
\end{pmatrix}\bigg\|_M^2 
\label{lbaeq}
\end{align}
for all~$\theta\in\Theta$,~$\bar{q},\bar{q}'\in\mathbb{R}^d$ satisfying~$\abs{\bar{q}-\bar{q}'}\geq 4R(1+\frac{L}{m})$ and~$\bar{p},\bar{p}'\in\mathbb{R}^d$.
\end{prop}

\begin{prop}\label{lab}
Let Assumption~\ref{A1} hold,~$\eta_0,\eta_1\in[0,1]$ satisfy~$0<\eta_0-\eta_1\leq\frac{1}{2}$ and let~$M\in\mathbb{R}^{2d}$ be given by~\eqref{mdef}. Assume~$0<Lh^2\leq (\eta_0-\eta_1)^2$. It holds that
\begin{align}
&\bigg\|\begin{pmatrix}
I_d & 0\\
0 & \eta_1I_d
\end{pmatrix}
\begin{pmatrix}
\bar{q} - \bar{q}' + \frac{h}{2}(\bar{p} - \bar{p}')\\
\bar{p} - \bar{p}' - \frac{h}{2}(b(\bar{q}+\frac{h}{2}\bar{p},\theta) - b(\bar{q}' + \frac{h}{2}\bar{p}',\theta))
\end{pmatrix}
\bigg\|_M\nonumber\\
&\quad\leq \bigg\|\begin{pmatrix}
I_d & 0\\
0 & \eta_0I_d
\end{pmatrix}
\begin{pmatrix}
\bar{q} - \bar{q}'\\
\bar{p} - \bar{p}'
\end{pmatrix}
\bigg\|_M \label{labeq}
\end{align}
for all~$\theta\in\Theta$,~$\bar{q},\bar{q}',\bar{p},\bar{p}'\in\mathbb{R}^d$ satisfying~$\abs{\bar{q} - \bar{q}' + \frac{h}{2}(\bar{p} - \bar{p}')}\geq 4R(1+\frac{L}{m})$.
\end{prop}

In the randomized midpoint case, only the~$\mathcal{BA}$ step needs to be dealt with, but the proof must be adapted accordingly. This is given in the following.
\begin{prop}\label{lbaR}
Let Assumption~\ref{A1} hold, let~$\bar{u}\in(0,h)$,~$\eta_0,\eta_1\in[0,1]$ satisfy~$0<\eta_0-\eta_1\leq1$ and~$c\in\mathbb{R}$,~$M\in\mathbb{R}^{2d\times 2d}$ be given by~\eqref{mdef}. 
Assume~$
Lh^2\leq (\eta_0-\eta_1)^2$. It holds that
\begin{align}
&\bigg\|\begin{pmatrix} 
I_d & 0\\
0 & \eta_1 I_d
\end{pmatrix}
\begin{pmatrix}
\bar{q}-\bar{q}' + \frac{h}{2}(\bar{p}-\bar{p}') - \frac{h^2}{4}(b(\bar{q}+\bar{u}\bar{p},\theta) - b(\bar{q}'+\bar{u}\bar{p}',\theta))\\
\bar{p} - \bar{p}' - \frac{h}{2}(b(\bar{q}+\bar{u}\bar{p},\theta) - b(\bar{q}'+\bar{u}\bar{p}',\theta))
\end{pmatrix} \bigg\|_M^2 \nonumber \\
&\quad\leq (1-c)\bigg\| \begin{pmatrix}
I_d & 0\\
0 & \eta_0 I_d
\end{pmatrix}
\begin{pmatrix}
\bar{q} - \bar{q}'\\
\bar{p} - \bar{p}'
\end{pmatrix}\bigg\|_M^2 
\label{lbaeqR}
\end{align}
for all~$\theta\in\Theta$,~$\bar{q},\bar{q}',\bar{p},\bar{p}'\in\mathbb{R}^d$ satisfying~$\abs{\bar{q}+\bar{u}\bar{p}-\bar{q}'-\bar{u}\bar{p}'}\geq 4R(1+\frac{L}{m})$ and~$\bar{p},\bar{p}'\in\mathbb{R}^d$.
\end{prop}

\subsection{Relatively large velocity differences}\label{rlvd}

In Propositions~\ref{lba2} and~\ref{lab2}, we follow largely the same approach as for the proofs in Section~\ref{lpd}. However, 
a different region of~$(\bar{q}-\bar{q}',\bar{p}-\bar{p}')$ space is considered. Consequently, the matrices~$H$ and~$\bar{H}$ 
given by~\eqref{mvth} and~\eqref{barH} 
are not necessarily positive definite. In order to obtain inequalities~\eqref{lbaeq} and~\eqref{labeq}, some positive~$\abs{\bar{p}-\bar{p}'}^2$ term is used to compensate for the lack of such~$\abs{\bar{q}-\bar{q}'}^2$ terms through restriction in the region of~$(\bar{q}-\bar{q}',\bar{p}-\bar{p}')$ space under consideration.
In this section, the additional assumption~$Lh^2\leq1/16^2$ will be made. This is not necessary to obtain contraction, but it is assumed for other parts of the paper and the constants here improve as a result. 
All of the proofs in this section may be found in 
Appendix~\ref{conapp}. 

\begin{prop}\label{lba2}
Let Assumption~\ref{A1} hold,~$\eta_0,\eta_1\in[0,1]$ satisfy~$0<\eta_0-\eta_1\leq \min(\eta_1,1/2)$ and let~$c\in\mathbb{R}$,~$M\in\mathbb{R}^{2d\times 2d}$ be given by~\eqref{mdef}.
Assume~$0<Lh^2\leq \min((\eta_0-\eta_1)^2,1/16^2)$. Inequality~\eqref{lbaeq} holds for all~$\bar{q},\bar{q}',\bar{p},\bar{p}'\in\mathbb{R}^d$ satisfying~$\abs{\bar{p} - \bar{p}'} \geq \sqrt{(17/4)L}\abs{\bar{q} - \bar{q}'}$. 
\end{prop}

\begin{prop}\label{lab2}
Let Assumption~\ref{A1} hold,~$\eta_0,\eta_1\in[0,1]$ satisfy~$0<\eta_0-\eta_1\leq \frac{1}{2}$ and let~$M\in\mathbb{R}^{2d\times2d}$ be given by~\eqref{mdef}. Assume~$0<Lh^2\leq \min((\eta_0-\eta_1)^2,1/16^2)$. Inequality~\eqref{labeq} holds for all~$\bar{q},\bar{q}',\bar{p},\bar{p}'\in\mathbb{R}^d$ satisfying~$\abs{\bar{p} - \bar{p}'} \geq \sqrt{(17/4)L}\abs{\bar{q} - \bar{q}'}$.
\end{prop}

\begin{prop}\label{lba2R}
Let Assumption~\ref{A1} hold, let~$\bar{u}\in(0,h)$,~$\eta_0,\eta_1\in[0,1]$ satisfy~$0\leq\eta_1<\eta_0\leq 1$ and let~$c\in\mathbb{R}$,~$M\in\mathbb{R}^{2d\times 2d}$ be given by~\eqref{mdef}.
Assume~$0<Lh^2\leq \min((\eta_0-\eta_1)^2,2/16^2)$. Inequality~\eqref{lbaeqR} holds for all~$\bar{q},\bar{q}',\bar{p},\bar{p}'\in\mathbb{R}^d$ satisfying~$\abs{\bar{p} - \bar{p}'} \geq 4\sqrt{L}\abs{\bar{q} - \bar{q}'}$. 
\end{prop}

\subsection{Contraction and otherwise for full iterations}\label{caoffi}
In this section, the results from Sections~\ref{lpd} and~\ref{rlvd} 
are combined to obtain Theorem~\ref{convthm}. A large part of the proof is to verify that given initially large position differences and relatively large velocity differences, the position and velocity differences stay large and relatively large respectively over the approximated Hamiltonian trajectory. In addition to Theorem~\ref{convthm}, for initial position and velocity differences that do not fall into the regions studied in Sections~\ref{lpd} and~\ref{rlvd}, the degree to which expansion may occur over iterations of reflectively coupled (as in Section~\ref{secsg}) SGgHMC chains is given in Propositions~\ref{expo},~\ref{expo2} and Corollary~\ref{refmod}. 
Let~$\bar{M}\in\mathbb{R}^{2d\times2d}$ be given by
\begin{equation}\label{Mdef}
\bar{M} = \begin{pmatrix} 1 & \gamma^{-1}\\ \gamma^{-1} & 2\gamma^{-2}\end{pmatrix}\otimes I_d.
\end{equation}
In Theorem~\ref{convthm}, the modified Euclidean norm associated with~$\bar{M}$ is that which is shown to be contractive under iterations of synchronously coupled SGgHMC chains. Recall that the parameter~$u$ determines whether the velocity Verlet integrator ($u=0$) is used or the randomized midpoint method ($u=1$) is used.
\begin{theorem}\label{convthm}
Let Assumption~\ref{A1} hold. Let~$x,y,v,w$ be~$\mathbb{R}^d$-valued r.v.'s,~$G,\hat{G}\sim\mathcal{N}(0,I_d)$ be independent of~$x,y,v,w,(u_{kh})_{k\in\mathbb{N}}$ and let~$X',Y',V',W'$ be given by~\eqref{ous}. 
Assume~$4LT^2/(1-\eta)^2\leq 1$ and 
\begin{equation}\label{tars}
L(T+h)^2\leq 1/16^2
\end{equation}
Assume~$\hat{G}=G$ holds almost surely. If~$u=0$,~$c_1=3$,~$c_2=6$, then for a.a.~$\omega\in\Omega$ such that~$x,y,v,w$ satisfy either
\begin{equation}\label{evz}
\eta\abs{v-w} \geq c_1\sqrt{L}\abs{x-y}
\end{equation}
or
\begin{equation}\label{qzk2}
\abs{x-y} \geq c_2R(1+L/m),
\end{equation}
it holds a.s. that
\begin{equation}\label{somvo}
\mathbb{E}\bigg[\bigg\|
\begin{pmatrix}
X' - Y'\\
V' - W'
\end{pmatrix} \bigg\|_{\bar{M}}^2\bigg|x,y,v,w,G,\hat{G}\bigg]
\leq \bigg(1-\frac{mT^2}{16(1-\eta)}\bigg)\bigg\| 
\begin{pmatrix}
x - y\\
v - w
\end{pmatrix}\bigg\|_{\bar{M}}^2,
\end{equation}
where~$\bar{M}$ is given by~\eqref{Mdef}. Moreover, if~$u=1$,~$c_1=8$ and~$c_2=10$, then the same assertion holds.
\end{theorem}
\begin{proof}
First, let~$u=0$. Let~$s\in[0,T]$ and recall the notation~\eqref{ous},~\eqref{zdefs}. Consider first the case where~\eqref{evz} holds. 
For any~$\hat{\theta},\hat{\theta}'\in\Theta^{T/h}$ and~$\bar{\theta} = (\hat{\theta},\hat{\theta}')$, Lemma~\ref{basic} with~$c=16^2$ in particular implies a.s. that
\begin{equation}\label{lem64b}
\max_{\bar{s}\in[0,T]}\abs{\bar{q}_{\bar{s}}(x,v',\bar{\theta}) - \bar{q}_{\bar{s}}(y,w',\bar{\theta})} \leq (64/63)\max(\abs{x-y},\abs{x-y+(v'-w')T}).
\end{equation}
Moreover,~\eqref{lem64b} implies
\begin{equation}\label{lem64c}
\sqrt{L}\abs{\bar{q}_s(x,v',\bar{\theta}) - \bar{q}_s(y,w',\bar{\theta})} \leq (64/63)(\sqrt{L}\abs{x-y} + \eta\abs{v-w}/16).
\end{equation}
On the other hand, given~\eqref{tars}, 
inequality~\eqref{lem64} with~$c=32$ holds, which implies a.s. that 
\begin{align}
\abs{\bar{p}_s(x,v',\bar{\theta}) - \bar{p}_s(y,w',\bar{\theta})}  &\geq \eta\abs{v - w} - \abs{\bar{p}_s(x,v',\bar{\theta}) - \bar{p}_s(y,w',\bar{\theta}) - (v' - w')}\nonumber\\
&\geq \eta\abs{v - w} - 3(64/63)LT(\abs{z} + \eta T\abs{v-w})\nonumber\\
&\geq (62/63)\eta\abs{v-w} - (12/63)\sqrt{L}\abs{z}.\label{sps}
\end{align}
Let~$r^*=(17/4)^{1/2}$. A numerical calculation gives
\begin{equation*}
c_1\bigg(\frac{62}{63} - \frac{64}{63}\cdot\frac{1}{16}\bigg(\frac{r^*}{1-r^*/32} + \frac{1}{32}\bigg) \bigg) \geq \frac{12}{63} + \frac{64}{63}\bigg(\frac{r^*}{1-r^*/32} + \frac{1}{32}\bigg),
\end{equation*}
therefore, we may apply~\eqref{evz} on the right-hand side of~\eqref{sps} to obtain a.s. that
\begin{equation}\label{sps2}
\abs{\bar{p}_s(x,v',\bar{\theta}) - \bar{p}_s(y,w',\bar{\theta})}
\geq \frac{64}{63}\bigg(\frac{r^*}{1-r^*/32} + \frac{1}{32}\bigg)\bigg(\frac{1}{16}\eta\abs{v-w} + \sqrt{L}\abs{z}\bigg).
\end{equation}
Substituting~\eqref{sps2} into the right-hand side of~\eqref{lem64c} yields
\begin{align}
\abs{\bar{p}_s(x,v',\bar{\theta}) - \bar{p}_s(y,w',\bar{\theta})} &\geq \bigg[\frac{r^*}{1-r^*/32} + \frac{1}{32}\bigg]\sqrt{L}\abs{\bar{q}_s(x,v',\bar{\theta}) - \bar{q}_s(y,w',\bar{\theta})} \nonumber\\
&\geq r^*\sqrt{L}\abs{\bar{q}_s(x,v',\bar{\theta}) - \bar{q}_s(y,w',\bar{\theta})}. \label{spr}
\end{align}
On the other hand, if~\eqref{evz} does not hold, then together with~\eqref{tars}, it holds that
\begin{equation*}
\abs{z + \eta s(v - w)} \geq \abs{z} - T\eta\abs{(v - w)} \geq \abs{z} - c_1\sqrt{L}T\abs{z} \geq (1-c_1/16)\abs{z},
\end{equation*}
which, by 
\eqref{lem64a} with~$c=16^2$ and~\eqref{tars}, implies a.s. for~$s\in h\mathbb{N}\cap[h,T]$ that
\begin{align}
&\abs{\bar{q}_s(x,v',\bar{\theta}) - \bar{q}_s(y,w',\bar{\theta})}\nonumber\\
&\quad\geq \abs{z + s(v' - w')} - \abs{\bar{q}_s(x,v',\bar{\theta}) - \bar{q}_s(y,w',\bar{\theta}) - z - s(v'-w')}\nonumber\\
&\quad\geq \abs{z + \eta s(v - w)} - (3/511)\max(\abs{z}, \abs{z + s(v' - w')})\nonumber\\
&\quad\geq \bigg(1-\frac{c_1}{16}-\frac{3}{511}\bigg)\abs{z}.\label{ccv}
\end{align}
Inequalities~\eqref{ccv} and~\eqref{qzk2} imply
\begin{equation}\label{qake}
\abs{\bar{q}_s(x,v',\bar{\theta}) - \bar{q}_s(y,w',\bar{\theta})} \geq c_2(1-c_1/16-3/511)R(1+L/m) \geq 4R(1+L/m).
\end{equation}
In order to make use of inequalities~\eqref{spr} and~\eqref{qake} for Propositions~\ref{lba} or~\ref{lba2}, for~$i\in\frac{1}{2}\mathbb{N}\cap[0,K]$, let~$\hat{\eta}_i\in [0,1]$ be given by~$\hat{\eta}_0 = 1$ and
\begin{equation}\label{etad}
\hat{\eta}_i - \hat{\eta}_{i+\frac{1}{2}} = (1-\eta)/(2K).
\end{equation}
For any~$i\in\frac{1}{2}\mathbb{N}\cap[0,K]$,~\eqref{etad} implies by assumption that
\begin{equation}\label{etar}
Lh^2/(\hat{\eta}_i - \hat{\eta}_{i+\frac{1}{2}})^2) = 4LT^2/(1-\eta)^2 \leq 1.
\end{equation}
For any~$i\in \mathbb{N}\cap[0,K]$, if the full step
\begin{align*}
\tilde{q}_{ih}&:=\bar{q}_{ih}(x,v',\bar{\theta}),\qquad\tilde{q}_{ih}':=\bar{q}_{ih}(y,w',\bar{\theta}),\\
\tilde{p}_{ih}&:=\bar{p}_{ih}(x,v',\bar{\theta}),\qquad\tilde{p}_{ih}':=\bar{p}_{ih}(y,w',\bar{\theta})
\end{align*} 
along the trajectory satisfy~\eqref{spr} with~$s=ih$, then Proposition~\ref{lba2} may be applied with~$\eta_0 = \hat{\eta}_i$,~$\eta_1 = \hat{\eta}_{i+\frac{1}{2}}$ and~$\bar{q} = \tilde{q}_{ih}$,~$\bar{q}' = \tilde{q}_{ih}'$,~$\bar{p} = \tilde{p}_{ih}$,~$\bar{q}' = \tilde{p}_{ih}'$.
Otherwise, if instead of~\eqref{spr}, inequality~\eqref{qake} is satisfied, then Proposition~\ref{lba} may be applied with the same~$\eta_0,\eta_1,\bar{q},\bar{q}',\bar{p},\bar{p}'$. In both cases, inequality~\eqref{lbaeq} holds with~\eqref{mdef} and the aforementioned substitutions; explicitly, for any~$i\in \mathbb{N}\cap[0,K]$, if either~\eqref{evz} or~\eqref{qzk2} holds, then it holds a.s. that
\begin{align}
&\bigg\|\begin{pmatrix} 
I_d & 0\\
0 & \hat{\eta}_{i+\frac{1}{2}} I_d
\end{pmatrix}
\begin{pmatrix}
\tilde{q}_{ih}-\tilde{q}_{ih}' + \frac{h}{2}(\tilde{p}_{ih}-\tilde{p}_{ih}') - \frac{h^2}{4}(b(\tilde{q}_{ih},\theta_{ih}) - b(\tilde{q}_{ih}',\theta_{ih}))\\
\tilde{p}_{ih} - \tilde{p}_{ih}' - \frac{h}{2}(b(\tilde{q}_{ih},\theta_{ih}) - b(\tilde{q}_{ih}',\theta_{ih}))
\end{pmatrix} \bigg\|_M^2 \nonumber \\
&\quad\leq \bigg(1 - \frac{mTh}{8(1-\eta)} + \frac{mih^2}{8} \bigg)\bigg\| \begin{pmatrix}
I_d & 0\\
0 & \hat{\eta}_i I_d
\end{pmatrix}
\begin{pmatrix}
\tilde{q}_{ih} - \tilde{q}_{ih}'\\
\tilde{p}_{ih} - \tilde{p}_{ih}'
\end{pmatrix}\bigg\|_M^2, 
\label{lbaeq2}
\end{align}
where~$\theta_{ih}$ is given by~$\hat{\theta} = (\theta_{ih})_{i\in\mathbb{N}\cap [0,T/h)}$ and~$M$ is given by~\eqref{mdef}.
Next, to see that Propositions~\ref{lab},~\ref{lab2} may be applied to the half steps in the trajectory, consider first the case when~\eqref{spr} holds. For any~$i\in \mathbb{N}\cap[0,K]$,~\eqref{spr},~\eqref{A1a} and~\eqref{tars} imply,
\begin{align}
&\abs{\tilde{p}_{ih} - \tilde{p}_{ih}' - (h/2)(b(\tilde{q}_{ih},\theta_{ih}) - b(\tilde{q}_{ih}',\theta_{ih}))} \nonumber\\
&\quad\geq \bigg(\sqrt{L}\bigg(\frac{r^*}{1-r^*/32} + \frac{1}{32}\bigg) - \frac{Lh}{2}\bigg)\abs{\tilde{q}_{ih} - \tilde{q}_{ih}'}\nonumber\\
&\quad\geq \frac{r^*}{1-r^*/32}\sqrt{L}\abs{\tilde{q}_{ih} - \tilde{q}_{ih}'}.\label{czr}
\end{align}
Moreover, it holds that
\begin{align*}
\sqrt{L}\abs{\tilde{q}_{ih} - \tilde{q}_{ih}'} &\geq \sqrt{L}\abs{\tilde{q}_{ih} - \tilde{q}_{ih}' + (h/2)(\tilde{p}_{ih} - \tilde{p}_{ih}' - (h/2)(b(\tilde{q}_{ih},\theta_{ih}) - b(\tilde{q}_{ih}',\theta_{ih}))}\\
&\quad - \sqrt{L}(h/2)\abs{\tilde{p}_{ih} - \tilde{p}_{ih}' - (h/2)(b(\tilde{q}_{ih},\theta_{ih}) - b(\tilde{q}_{ih}',\theta_{ih}))},
\end{align*}
which, substituting back into the right-hand side of~\eqref{czr} and using~\eqref{tars}, implies
\begin{align*}
&\bigg(1+\frac{r^*/32}{1-r^* /32}\bigg)\abs{\tilde{p}_{ih} - \tilde{p}_{ih}' - (h/2)(b(\tilde{q}_{ih},\theta_{ih}) - b(\tilde{q}_{ih}',\theta_{ih}))}\\
&\quad\geq \frac{r^*}{1-r^*/32}\sqrt{L}\abs{\tilde{q}_{ih} - \tilde{q}_{ih}' + (h/2)(\tilde{p}_{ih} - \tilde{p}_{ih}' - (h/2)(b(\tilde{q}_{ih},\theta_{ih}) - b(\tilde{q}_{ih}',\theta_{ih})))}.
\end{align*}
In other words, it holds that
\begin{align*}
&\abs{\tilde{p}_{ih} - \tilde{p}_{ih}' - (h/2)(b(\tilde{q}_{ih},\theta_{ih}) - b(\tilde{q}_{ih}',\theta_{ih}))}\\
&\quad \geq \sqrt{L}r^*\abs{\tilde{q}_{ih} - \tilde{q}_{ih}' + (h/2)(\tilde{p}_{ih} - \tilde{p}_{ih}' - (h/2)(b(\tilde{q}_{ih},\theta_{ih}) - b(\tilde{q}_{ih}',\theta_{ih})))}.
\end{align*}
Therefore for any~$i\in \mathbb{N}\cap[0,K]$, if~\eqref{spr} is satisfied with~$s=ih$, then again since~$Lh^2\leq L(T+h)^2\leq 1/16^2$, Proposition~\ref{lab2} may be applied with~$\eta_0 = \hat{\eta}_{i+\frac{1}{2}}$,~$\eta_1 = \hat{\eta}_{i+1}$,
\begin{equation}\label{deq}
\bar{q} = \tilde{q}_{ih} + (h/2)\tilde{p}_{ih} - (h/2)^2b(\tilde{q}_{ih},\theta_{ih}),\qquad\bar{p} = \tilde{p}_{ih} - (h/2)b(\tilde{q}_{ih},\theta_{ih})
\end{equation}
and~$\bar{q}',\bar{p}'$ given by~\eqref{deq} with~$\bar{q}',\bar{p}',\tilde{q}_{ih}',\tilde{p}_{ih}'$ replacing~$\bar{q},\bar{p},\tilde{q}_{ih},\tilde{p}_{ih}$ respectively.
If for any~$i\in \mathbb{N}\cap[0,K]$,~\eqref{spr} with~$s=ih$ does not hold and inequality~\eqref{qake} is satisfied with~$s=ih$, then it holds by~\eqref{A1a} and applying the negation of~\eqref{spr} with~$s=ih$ that
\begin{align*}
&\abs{\tilde{q}_{ih} - \tilde{q}_{ih}' + h(\tilde{p}_{ih} - \tilde{p}_{ih}' - (h/2)(b(\tilde{q}_{ih},\theta_{ih}) - b(\tilde{q}_{ih}',\theta_{ih})))}\\
&\quad\geq \abs{\tilde{q}_{ih} - \tilde{q}_{ih}'} - h\abs{\tilde{p}_{ih} - \tilde{p}_{ih}'} -(h^2/2)\abs{b(\tilde{q}_{ih},\theta_{ih}) - b(\tilde{q}_{ih}',\theta_{ih})}\\
&\quad\geq \bigg(1-\bigg(\frac{r^*}{1-r^*/32}+\frac{1}{32}\bigg)\sqrt{L}h - \frac{Lh^2}{2}\bigg)\abs{\tilde{q}_{ih} - \tilde{q}_{ih}'},
\end{align*}
then applying~\eqref{tars} and~\eqref{qake} with~$s=ih$ gives
\begin{align*}
&\abs{\tilde{q}_{ih} - \tilde{q}_{ih}' + h(\tilde{p}_{ih} - \tilde{p}_{ih}' - (h/2)(b(\tilde{q}_{ih},\theta_{ih}) - b(\tilde{q}_{ih}',\theta_{ih})))}\\
&\quad\geq \bigg(1-\bigg(\frac{r^*}{1-r^*/32}+\frac{1}{32}\bigg)\frac{1}{32} - \frac{1}{32^2}\bigg)\abs{\tilde{q}_{ih} - \tilde{q}_{ih}'}\\
&\quad\geq \bigg(1-\bigg(\frac{r^*}{1-r^*/32}+\frac{1}{32}\bigg)\frac{1}{32} - \frac{1}{32^2}\bigg)\cdot c_2(1-c_1/16-1/63)R(1+L/m)\\
&\quad\geq 4R(1+L/m).
\end{align*}
Therefore in this case, Proposition~\ref{lab} may be applied with~$\eta_0 = \hat{\eta}_{i+\frac{1}{2}}$,~$\eta_1 = \hat{\eta}_{i+1}$,~$\bar{q},\bar{p}$ given by~\eqref{deq} and~$\bar{q}',\bar{p}'$ given by~\eqref{deq} with~$\bar{q}',\bar{p}',\tilde{q}_{ih}',\tilde{p}_{ih}'$ replacing~$\bar{q},\bar{p},\tilde{q}_{ih},\tilde{p}_{ih}$ respectively. In either case, by Proposition~\ref{lab} and~\ref{lab2}, inequality~\eqref{labeq} holds 
the aforementioned substitutions. Together with~\eqref{lbaeq2}, by the obvious induction in~$i\in \mathbb{N}\cap[0,K]$, if either~\eqref{evz} or~\eqref{qzk2} holds, then it holds a.s. that
\begin{align*}
&\mathbb{E}\bigg[\bigg\|
\begin{pmatrix}
\bar{q}_{T}(x,v',\bar{\theta})-\bar{q}_{T}(y,w',\bar{\theta})\\
\eta(\bar{p}_{T}(x,v',\bar{\theta}) - \bar{p}_{T}(y,w',\bar{\theta})) 
\end{pmatrix} \bigg\|_M^2\bigg|x,y,v,w,G,\hat{G}\bigg]\\
&\quad\leq \bigg(1-\frac{mT^2}{8(1-\eta)}+\frac{mT(T-h)}{16}\bigg)\bigg\| 
\begin{pmatrix}
x - y\\
v' - w'
\end{pmatrix}\bigg\|_M^2\\
&\quad\leq \bigg(1-\frac{mT^2}{16(1-\eta)} \bigg)\bigg\| 
\begin{pmatrix}
x - y\\
\eta(v - w)
\end{pmatrix}\bigg\|_M^2.
\end{align*}
The proof concludes for~$u=0$ by writing
\begin{equation*}
\begin{pmatrix}
I_d & 0\\
0 & \eta I_d
\end{pmatrix}
\begin{pmatrix}
I_d & \frac{TI_d}{1-\eta}\\
\frac{TI_d}{1-\eta} & \frac{2T^2I_d}{(1-\eta)^2}
\end{pmatrix}
\begin{pmatrix}
I_d & 0\\
0 & \eta I_d
\end{pmatrix} =
\begin{pmatrix}
1 & \frac{\eta T}{1-\eta}\\
\frac{\eta T}{1-\eta} & \frac{2\eta^2T^2}{(1-\eta)^2}
\end{pmatrix}\otimes I_d = \bar{M}.
\end{equation*}
For~$u=1$,~$c_1=8$,~$c_2=10$, the proof follows in a similar way, with the differences as follows. Let~$r^*$ instead be~$4$. Either inequality~\eqref{spr} or~\eqref{qake} holds by similar calculations using instead that either~$\eta\abs{v-w}\geq 8\sqrt{L}\abs{x-y}$ or~$\abs{x-y}\geq 10R(1+L/m)$. We set instead
\begin{equation}\label{Rvr}
\hat{\eta}_i-\hat{\eta}_{i+\frac{1}{2}} = (1-\eta)/K,
\end{equation}
so that
\begin{equation}\label{Rvr2}
\frac{L(2h)^2}{(\hat{\eta}_i-\hat{\eta}_{i+\frac{1}{2}})^2} = \frac{4LT^2}{(1-\eta)^2}\leq 1.
\end{equation}
If~\eqref{spr} holds, then we may apply Proposition~\ref{lba2R} with~$h/2$ replaced by~$h$ and with~$\bar{q} = \tilde{q}_{ih}$,~$\bar{q}' =\tilde{q}_{ih}'$,~$\bar{p}=\tilde{p}_{ih}$,~$\bar{p}'=\tilde{p}_{ih}'$ and~$\bar{u} = u_{ih}h$,~$\eta_0=\hat{\eta}_i$,~$\eta_1=\hat{\eta}_{i+\frac{1}{2}}$. If~\eqref{spr} does not hold (its negation holds) and instead~\eqref{qake} holds, then it holds that
\begin{align*}
&\abs{\tilde{q}_{ih} - \tilde{q}_{ih}' + u_{ih}h(\tilde{p}_{ih} - \tilde{p}_{ih}')} \\
&\quad\geq \abs{\tilde{q}_{ih} - \tilde{q}_{ih}'} -u_{ih}h\abs{\tilde{p}_{ih} - \tilde{p}_{ih}'}\\
&\quad\geq (1-u_{ih}(r^*/(1-r^*/32) + 1/32)\sqrt{L}h) \abs{\tilde{q}_{ih} - \tilde{q}_{ih}'}\\
&\quad\geq c_2(1-c_1/16-3/511)(1-u_{ih}(r^*/(1-r^*/32) + 1/32)/32)R(1+L/m)\\
&\quad\geq 4R(1+L/m),
\end{align*}
so that we may apply Proposition~\ref{lbaR} with again~$h/2$ replaced by~$h$ and with~$\bar{q} = \tilde{q}_{ih}$,~$\bar{q}' =\tilde{q}_{ih}'$,~$\bar{p}=\tilde{p}_{ih}$,~$\bar{p}'=\tilde{p}_{ih}'$ and~$\bar{u} = u_{ih}h$,~$\eta_0=\hat{\eta}_i$,~$\eta_1=\hat{\eta}_{i+\frac{1}{2}}$. The applications of Propositions~\ref{lbaR} and~\ref{lba2R} conclude in the same way as before.
\end{proof}

The following Propositions~\ref{expo},~\ref{expo2},~\ref{RvRth} have proofs that make use of the settings and strategies already introduced 
in the 
proofs of Proposition~\ref{lba},~\ref{lab},~\ref{lba2} and~\ref{lab2}. 
These proofs may be found in 
Appendix~\ref{conapp}. 

\begin{prop}\label{expo}
Let~$\eta_0,\eta_1\in[0,1]$ satisfy~$\eta_0<\eta_1$. Let~$c\in\mathbb{R}$ be given by
\begin{equation}\label{cbdef}
\bar{c}= 3L\eta_0h^2/(\eta_0 - \eta_1)
\end{equation}
and~$M$ given by~\eqref{mdef}. 
Assume~\eqref{A1a} and~$
Lh^2\leq (\eta_0-\eta_1)^2$. It holds that
\begin{align}
&\bigg\|\begin{pmatrix} 
I_d & 0\\
0 & \eta_1 I_d
\end{pmatrix}
\begin{pmatrix}
\bar{q}-\bar{q}' + \frac{h}{2}(\bar{p}-\bar{p}') - \frac{h^2}{4}(b(\bar{q},\theta) - b(\bar{q}',\theta))\\
\bar{p} - \bar{p}' - \frac{h}{2}(b(\bar{q},\theta) - b(\bar{q}',\theta))
\end{pmatrix} \bigg\|_M^2 \nonumber \\
&\quad\leq (1+\bar{c})\bigg\| \begin{pmatrix}
I_d & 0\\
0 & \eta_0 I_d
\end{pmatrix}
\begin{pmatrix}
\bar{q} - \bar{q}'\\
\bar{p} - \bar{p}'
\end{pmatrix}\bigg\|_M^2 
\label{lbaeq00}
\end{align}
for all~$\theta\in\Theta$,~$\bar{q},\bar{q}',\bar{p},\bar{p}'\in\mathbb{R}^d$.
\end{prop}

\begin{prop}\label{expo2}
Let~$\eta_0,\eta_1\in[0,1]$ satisfy~$0<\eta_0-\eta_1\leq\eta_1$,~$c\in\mathbb{R}$ be given by~$\bar{c}=2L\eta_0h^2/(\eta_0 - \eta_1)$ 
and~$M\in\mathbb{R}^{2d\times 2d}$ be given by~\eqref{mdef}. 
Assume~\eqref{A1a} and~$
Lh^2\leq (\eta_0-\eta_1)^2$. It holds that
\begin{align}
&\bigg\|\begin{pmatrix} 
I_d & 0\\
0 & \eta_1 I_d
\end{pmatrix}
\begin{pmatrix}
\bar{q}-\bar{q}' + \frac{h}{2}(\bar{p}-\bar{p}')\\
\bar{p} - \bar{p}' - \frac{h}{2}(b(\bar{q} + \frac{h}{2}\bar{p},\theta) - b(\bar{q}'+ \frac{h}{2}\bar{p}',\theta))
\end{pmatrix} \bigg\|_M^2 \nonumber \\
&\quad\leq (1+\bar{c})\bigg\| \begin{pmatrix}
I_d & 0\\
0 & \eta_0 I_d
\end{pmatrix}
\begin{pmatrix}
\bar{q} - \bar{q}'\\
\bar{p} - \bar{p}'
\end{pmatrix}\bigg\|_M^2 
\end{align}
for all~$\theta\in\Theta$,~$\bar{q},\bar{q}',\bar{p},\bar{p}'\in\mathbb{R}^d$.
\end{prop}

Next, we extend Proposition~\ref{expo} to cover iterations of the randomized midpoint version of the algorithm ($u=1$).
\begin{prop}\label{RvRth}
Let~$\eta_0,\eta_1\in[0,1]$ satisfy~$\eta_0<\eta_1$ and let~$\bar{u}\in(0,h)$. Let~$\bar{c}\in\mathbb{R}$ be given by~\eqref{cbdef} and~$M$ be given by~\eqref{mdef}. Assume~\eqref{A1a} and~$Lh^2\leq \min((\eta_0-\eta_1)^2,1/16^2)$. Inequality~\eqref{lbaeq00} holds with~$b(\bar{q},\theta),b(\bar{q}',\theta)$ replaced by~$b(\bar{q}+\bar{u}\bar{p},\theta)$ and~$b(\bar{q}'+\bar{u}\bar{p}',\theta)$ respectively for all~$\theta\in\Theta$,~$\bar{q},\bar{q}',\bar{p},\bar{p}' \in\mathbb{R}^d$.
\end{prop}

Before combining the previous propositions, we give next an elementary result about the second moments of~$K_{\cdot}$ that will be useful in quantifying the noisy part of the dynamics.
\begin{lemma}\label{teos}
Let~$r^*=(1-\eta)^2/(T^2(1-\eta^2))$, let~$G\sim\mathcal{N}(0,I_d)$,~$\mathcal{U}\sim\mathcal{U}(0,1)$ be independent and let~$K_{\cdot}$ be given by~\eqref{kdef} and~\eqref{Gdef2} with some arbitrary~$\hat{q}\in\mathbb{R}^d$ and its unit vector~$e$. 
It holds that
\begin{equation*}
\mathbb{E}[\abs{K_{\hat{q}}}^2] = \frac{T^2(1-\eta^2)}{(1-\eta)^2} \bigg((4+\abs{\hat{q}}^2)\bigg(\Phi\bigg(\frac{\abs{\hat{q}}}{2}\bigg) - \Phi\bigg(-\frac{\abs{\hat{q}}}{2}\bigg)\bigg) + 4\hat{q}\varphi_{0,1}\bigg(\frac{\abs{\hat{q}}}{2}\bigg)\bigg).
\end{equation*}
\end{lemma}
\begin{proof}
Fix~$\hat{q}\in\mathbb{R}^d$. By definitions~\eqref{kdef} and~\eqref{Gdef2}, it holds that 
\begin{equation*}
\mathbb{E}[\abs{K_{\hat{q}}}^2] = \frac{T^2(1-\eta^2)}{(1-\eta)^2}\int_{-\frac{\abs{\hat{q}}}{2}}^{\infty} (\abs{\hat{q}} + 2g)^2  (\varphi_{0,1}(g) - \varphi_{0,1}(g+\abs{\hat{q}})) dg,
\end{equation*}
which, by a change in variable, yields
\begin{equation}\label{keo}
\frac{(1-\eta)^2}{T^2(1-\eta^2)}\mathbb{E}[\abs{K_{\hat{q}}}^2] = \int_0^{\infty} 4g^2\bigg[\varphi_{0,1}\bigg(g-\frac{\abs{\hat{q}}}{2}\bigg) - \varphi_{0,1}\bigg(g+\frac{\abs{\hat{q}}}{2}\bigg)\bigg]dg.
\end{equation}
In the rest of the proof, let~$\hat{q}_2:=\frac{\abs{\hat{q}}}{2}$ and let~$\varphi$ denote~$\varphi_{0,1}$. For the first term in the square bracket on the right-hand side of~\eqref{keo}, integration by parts yields
\begin{align*}
&\int_0^{\infty} g^2\varphi(g-\hat{q}_2)dg\\
&\quad= \int_0^{\infty} ((g-\hat{q}_2)^2 + 2g\hat{q}_2 - \hat{q}_2^2)\varphi(g-\hat{q}_2)dg\\
&\quad= -\hat{q}_2\varphi(-\hat{q}_2) + \Phi(\hat{q}_2) + 2\hat{q}_2\int_0^{\infty}(g-\hat{q}_2)\varphi(g-\hat{q}_2)dg + \hat{q}_2^2\Phi(\hat{q}_2)\\
&\quad= \hat{q}_2\varphi(-\hat{q}_2) + (1+\hat{q}_2^2)\Phi(\hat{q}_2).
\end{align*}
Similarly, for the second term in the square bracket on the right-hand side of~\eqref{keo}, it holds that
\begin{align*}
&\int_0^{\infty} g^2\varphi(g+\hat{q}_2)dg\\
&\quad= \int_0^{\infty} ((g+\hat{q}_2)^2 - 2g\hat{q}_2 - \hat{q}_2^2)\varphi(g+\hat{q}_2)dg\\
&\quad= \hat{q}_2\varphi(-\hat{q}_2) + \Phi(-\hat{q}_2) - 2\hat{q}_2\int_0^{\infty}(g+\hat{q}_2)\varphi(g+\hat{q}_2)dg + \hat{q}_2^2\Phi(-\hat{q}_2)\\
&\quad= -\hat{q}_2\varphi(-\hat{q}_2) + (1+\hat{q}_2^2)\Phi(-\hat{q}_2),
\end{align*}
from which the assertion follows.
\end{proof}

Putting together Propositions~\ref{expo},~\ref{expo2} and Lemma~\ref{teos}, we obtain the following result on full iterations of a SGgHMC pair chain when the coupling is that which is studied in Section~\ref{secsg}.
\begin{corollary}\label{refmod}
Let Assumption~\ref{A1} hold. Assume~$4LT^2\leq (1-\eta)^2$ and~\eqref{carot}. Let~$x,y,v,w$ be~$\mathbb{R}^d$-valued r.v.'s, let~$G,\hat{G}\sim\mathcal{N}(0,I_d)$ be independent of~$x,y,v,w,(u_{kh})_{k\in\mathbb{N}}$ and let~$X',Y',V',W'$ be given by~\eqref{ous}. 
Suppose~$\hat{G}$ satisfies a.s. that
\begin{equation}\label{ggk}
\hat{G} = G - \bar{G},
\end{equation}
where~$\bar{G}$ is given by~\eqref{Gdef2} and~$\hat{q},e$ given by~\eqref{cou1},~\eqref{edef} and~\eqref{zdef5} with~$r^*=(1-\eta)^2/(T^2(1-\eta^2))$ and~$\gamma$ given by~\eqref{gdef}. 
It holds a.s. that
\begin{equation}\label{mcon2}
\mathbb{E}\bigg[\bigg\|
\begin{pmatrix}
X' - Y'\\
V' - W'
\end{pmatrix} \bigg\|_{\bar{M}}^2\bigg|x,y,v,w\bigg] \leq \bigg(1+\frac{10LT^2}{1-\eta}\bigg)\bigg\| 
\begin{pmatrix}
x - y\\
v - w
\end{pmatrix}\bigg\|_{\bar{M}}^2 + \hat{c}
\end{equation}
where~$\bar{M}$ is given by~\eqref{Mdef} and the constant~$\hat{c}$ is given by~$\hat{c} = \frac{2T^2(1+\eta)}{1-\eta}\max(\frac{8\abs{\hat{q}}}{\sqrt{2\pi}},4)$. 
Otherwise, if~$\hat{G}=G$ holds a.s., then~\eqref{mcon2} holds a.s. with~$\hat{c}=0$.
\end{corollary}
\begin{proof}
Recall the notation~\eqref{ous}. First, consider the case~$u=0$ (velocity Verlet). For~$i\in\frac{1}{2}\mathbb{N}\cap[0,K]$, let~$\hat{\eta}_i\in[0,1]$ be given by~$\hat{\eta}_0 = 1$ and~\eqref{etad}. By the assumption~$4LT^2\leq (1-\eta)^2$, inequality~\eqref{etar} holds. Therefore Propositions~\ref{expo} and~\ref{expo2} with~$\eta_0=\hat{\eta}_i$,~$\eta_1=\hat{\eta}_{i+\frac{1}{2}}$ imply a.s. that
\begin{align}
&\mathbb{E}\bigg[\bigg\|
\begin{pmatrix}
\bar{q}_{T}(x,v',\bar{\theta})-\bar{q}_{T}(y,w',\bar{\theta})\\
\eta(\bar{p}_{T}(x,v',\bar{\theta}) - \bar{p}_{T}(y,w',\bar{\theta})) 
\end{pmatrix} \bigg\|_M^2\bigg|x,y,v,w,G,\hat{G}\bigg]\nonumber\\
&\quad\leq \bigg(1+\frac{10LT^2}{1-\eta}\bigg)\bigg\| 
\begin{pmatrix}
x - y\\
v' - w'
\end{pmatrix}\bigg\|_M^2,\label{kcr}
\end{align}
where~$M$ is given by~\eqref{mdef}. In the case~$u=1$ (randomized midpoint), let instead~$\hat{\eta}_i$ be given by~\eqref{Rvr}. Again by assumption inequality~\eqref{Rvr2} holds. Therefore Proposition~\ref{RvRth} with~$h/2$ replaced by~$h$ and again~$\eta_0=\hat{\eta}_i$,~$\eta_1=\hat{\eta}_{i+\frac{1}{2}}$ implies a.s. that~\eqref{kcr} holds. 
For the right-hand side of~\eqref{kcr}, 
it holds a.s. that
\begin{align}
\bigg\|\begin{pmatrix}
z\\
v' - w'
\end{pmatrix}\bigg\|_M^2 &= 
\abs{z}^2 + \frac{2T}{1-\eta}z\cdot(\eta(v-w)+\sqrt{1-\eta^2}(G-\hat{G}))\nonumber\\
&\quad+ \frac{2T^2}{(1-\eta)^2}\abs{\eta(v-w)+\sqrt{1-\eta^2}(G-\hat{G})}^2\nonumber\\
&= 
\bigg\|\begin{pmatrix}
z\\
v - w
\end{pmatrix}\bigg\|_{\bar{M}}^2
+ \frac{2T^2(1-\eta^2)}{(1-\eta)^2}\abs{G-\hat{G}}^2\nonumber\\
&\quad+ \frac{2T\sqrt{1-\eta^2}}{1-\eta}(z+2\gamma^{-1}(v-w))\cdot(G-\hat{G}). \label{pwk}
\end{align}
In the case of~$\hat{G}=G$, the assertion follows. If instead equation~\eqref{ggk} holds, then by~$\mathbb{E}[\bar{G}|q] = \mathbb{E}[G-\hat{G}|q] = 0$, equation~\eqref{pwk} implies a.s. that
\begin{equation}\label{cm1}
\mathbb{E}\bigg[\bigg\|\begin{pmatrix}
z\\
v' - w'
\end{pmatrix}\bigg\|_M^2\bigg|x,y,v,w \bigg]
- 
\bigg\|\begin{pmatrix}
z\\
v - w
\end{pmatrix}\bigg\|_{\bar{M}}^2 = \frac{2T^2(1-\eta^2)}{(1-\eta)^2}\mathbb{E}[\abs{\bar{G}}^2|q].
\end{equation}
Moreover, using~\eqref{kdef}, Lemma~\ref{teos} and~\eqref{ggk}, it holds a.s. that 
\begin{align}
\mathbb{E}[\abs{\bar{G}}^2|q]
&= \mathbb{E}[\abs{\hat{q} + \bar{G}}^2|q] - \abs{\hat{q}}^2 - 2\abs{\hat{q}}\cdot\mathbb{E}[ \bar{G}|q]\nonumber\\
&= (1-\eta)^2/(T^2(1-\eta^2))\mathbb{E}[\abs{K_{\hat{q}}}^2|q] - \abs{\hat{q}}^2 - 2\abs{\hat{q}}\cdot\mathbb{E}[e\cdot G - e\cdot\hat{G}]\nonumber\\
&=(4+\abs{\hat{q}}^2)(\Phi(\abs{\hat{q}}/2) - \Phi(-\abs{\hat{q}}/2)) + 4\abs{\hat{q}}\varphi_{0,1}(\abs{\hat{q}}/2) - \abs{\hat{q}}^2\nonumber\\
&\leq 8\abs{\hat{q}}/(\sqrt{2\pi}). 
\label{jero}
\end{align}
On the other hand, if equation~\eqref{ggk} holds, then it holds 
that~$\abs{\bar{G}} = \abs{e\cdot e\bar{G}} = \abs{e \cdot(G - \hat{G})}$, 
which implies
\begin{equation*}
\mathbb{E}[\abs{\bar{G}}^2|q] = 2 - 2\mathbb{E}[(G\cdot e)(\hat{G}\cdot e)|q] = 2\mathbb{E}[G\cdot e\bar{G}|q] \leq 2\mathbb{E}[\abs{G\cdot e}^2|q] + \frac{1}{2}\mathbb{E}[\abs{\bar{G}}^2|q]
\end{equation*}
and therefore~$\mathbb{E}[\abs{\bar{G}}^2|q] \leq 4$. 
Consequently, together with~\eqref{kcr},~\eqref{cm1},~\eqref{jero} and the inequality~$1-\eta^2=(1+\eta)(1-\eta)$, the assertion~\eqref{mcon2} follows.
\end{proof}

\section{Global contraction}\label{global}
The main result in this section is Theorem~\ref{coth}, where both couplings from Sections~\ref{ncsec} and~\ref{convex} are used to obtain a contraction in an appropriate semimetric under the assumption~$4LT^2\leq (1-\eta)^2$. The semimetric used is inspired by the class of additive metrics studied in~\cite{MR3939573}. A related metric also appears in~\cite{schuh2022global}. 
In addition, without assuming~$4LT^2\leq (1-\eta)^2$ and instead assuming the existence of a Lyapunov function, contraction in another suitable semimetric, more akin to that in~\cite[equation~(33)]{MR4133372}, is given in Theorem~\ref{coth2}. 
Thereafter, some consequences of (only) Theorem~\ref{coth} are given in Section~\ref{conseqs}.

\subsection{Semimetric contraction}\label{semco}
To state our results, notation from the previous sections is assumed; for the reader's convenience, some notation is explicitly recalled. Some new notation is also introduced. Recall that~$u\in\{0,1\}$ determines whether the velocity Verlet or randomized midpoint integrator is considered. 
Let~$\alpha=LT^2/(1-\eta)^2$,~$\gamma$ be given by~\eqref{gdef} and fix~$\hat{R}$ to be
\begin{equation}\label{rhdef}
\hat{R}= \max((6+4u)(1+(3+5u)\alpha^{\frac{1}{2}} + 1.09\alpha)R(1+L/m),(4L)^{-\frac{1}{2}}). 
\end{equation}
Let~$g$ be given by~\eqref{gdef??}.  Let~$\epsilon^*>0$ be given by
\begin{equation}\label{epidef}
\epsilon^* = \min\bigg(\frac{e^{-g\hat{R}}}{
101(2\mathds{1}_{\eta\neq 0}/\alpha^2 + \mathds{1}_{\eta=0})\hat{R}},\frac{\sqrt{2\pi}ge^{-g\hat{R}}}{1024}\bigg),
\end{equation}
let~$\bar{M}$ be given by~\eqref{Mdef} and 
let~$\rho^*:\mathbb{R}^{2d}\times\mathbb{R}^{2d}\rightarrow[0,\infty)$ be given by
\begin{equation}\label{rsdef}
\rho^*((x,v),(y,w)) = f_0(\abs{q}+1.09\alpha \abs{z})+\epsilon^*\| (x-y,v-w) \|_{\bar{M}}^2
\end{equation}
for all~$x,y,v,w\in\mathbb{R}^d$, 
where~$q,z$ are given by~\eqref{zdefs} and~$f_0$ is defined by~\eqref{f0def}. 
\begin{theorem}\label{coth}
Let Assumption~\ref{A1} hold. Let~$x,y,v,w$ be~$\mathbb{R}^d$-valued r.v.'s,~$G,\hat{G}\sim\mathcal{N}(0,I_d)$ be independent of~$x,y,v,w,(u_{kh})_{k\in\mathbb{N}}$ and let~$X',Y',V',W'$ be given by~\eqref{ous}.
In addition, assume~$4LT^2\leq (1-\eta)^2$ and~\eqref{carot}.
If it holds a.s. that
\begin{equation}\label{kcou}
\hat{G} = \begin{cases}
G&\textrm{if }\abs{q} + 1.09\alpha\abs{z} \geq \hat{R} 
\\
G - \bar{G}&\textrm{if }\abs{q} + 1.09\alpha\abs{z} < \hat{R},
\end{cases}
\end{equation}
where~$\bar{G}$ satisfies~\eqref{Gdef2} with~$\hat{q},e$ given by~\eqref{zdefs},~\eqref{cou1},~\eqref{edef} and~$r^*=(1-\eta)/(T^2(1+\eta))$, then it holds a.s. that
\begin{equation}\label{coy}
\mathbb{E}[\rho^*((X',V'),(Y',W'))|x,y,v,w] \leq (1-c)\rho^*((x,v),(y,w)),
\end{equation}
where~$c$ is given by
\begin{equation}\label{cothrate}
c=\frac{mT^2}{e^{g\hat{R}}(1-\eta)} \cdot \frac{\min(4/(2\mathds{1}_{\eta\neq 0}/\alpha^2 + \mathds{1}_{\eta=0}),1)}{6592}.
\end{equation}
\end{theorem}
\begin{remark}\label{rrcom}
We may compare this convergence rate~\eqref{cothrate} with~\cite[equation~(32)]{MR4133372} in terms of the dependence on (large) nonconvexity radius. Recall from Remark~\ref{Rcom} that the appropriate value for comparison is~$R'=4R(1+L/m)$ as in Remark~\ref{Rcom}. In the best case in~\cite[equation~(27)]{MR4133372}, one may take~$T+h=1/(16LR')$. In that work, this choice of~$T$ gives a convergence rate bounded by
\begin{equation}\label{hera}
\frac{1}{10}\min\bigg(1,\frac{1}{4}mT^2(1+16L(R')^2)e^{-8L(R')^2}\bigg)e^{-32L(R')^2}.
\end{equation}
On the other hand, for large~$R'$ and~$u=0$ (velocity Verlet integrator), the convergence rate~\eqref{cothrate} is equal to
\begin{equation}\label{cothrp}
C(R')^2e^{-(2/5)16L(1.5(1+3\alpha^{\frac{1}{2}} + 1.09\alpha)R')^2}
\end{equation}
for some constant~$C>0$ independent of~$R'=4R(1+L/m)$. Therefore, for~$\alpha^{\frac{1}{2}}< 1/5$, the rate~\eqref{cothrp} has improved dependence on~$R'$ than~\eqref{hera}. 
\end{remark}
\begin{proof}
Recall the notation~\eqref{zdefs}. 
The probability space is split into different regions based on the values of~$x,y,v,w$. Moreover, if~$u=0$ (velocity Verlet), then let~$c_1=3$,~$c_2=6$. Otherwise if~$u=1$ (randomized midpoint), then let~$c_1=8$,~$c_2=10$. Firstly, assume
\begin{equation}\label{ae1}
\abs{q} + 1.09\alpha\abs{z} < \hat{R}.
\end{equation}
In this case,~$\hat{G}=G-\bar{G}$ a.s. by assumption. Theorem~\ref{main2} and Corollary~\ref{refmod} imply a.s. that
\begin{align}
&\mathbb{E}[\rho^*((X',V'),(Y',W'))|x,y,v,w]\nonumber\\
&\quad= \mathbb{E}[f_0(\abs{Q'} + 1.09\alpha \abs{Z'})+\epsilon^*\|(Z',V'-W') \|_{\bar{M}}^2|x,y,v,w]\nonumber\\
&\quad\leq \bigg(1-\frac{c_0LT^2}{1-\eta}\bigg)f_0(\abs{q}+1.09\alpha \abs{z}) + \epsilon^*\bigg(1+\frac{10LT^2}{1-\eta}\bigg)\bigg\|\begin{pmatrix}
z\\
v-w
\end{pmatrix}\bigg\|_{\bar{M}}^2\nonumber\\
&\qquad- \frac{ge^{-g\hat{R}}T^2(1+\eta)}{32(1-\eta)}\cdot\min\bigg(\abs{\hat{q}},\frac{67}{50}\bigg) + \frac{2\epsilon^* T^2(1+\eta)}{1-\eta} 
\cdot\min\bigg(\frac{8\abs{\hat{q}}}{\sqrt{2\pi}},4\bigg),\label{frh}
\end{align}
where~$c_0$ is given by~\eqref{cor}. 
Moreover, 
by direct calculation and by~\eqref{ae1}, it holds that
\begin{equation}\label{suc}
\bigg\|\begin{pmatrix}
z\\
v-w
\end{pmatrix}\bigg\|_{\bar{M}}^2 
=\abs{q}^2 + \abs{q - z}^2 
\leq 3\abs{q}^2 + 2\abs{z}^2
\leq 3\hat{R}\abs{q} + 2\hat{R}\abs{z}/(1.09\alpha),
\end{equation}
the right-hand side of which may be bounded, using~$\alpha\leq 1/4$, as
\begin{equation*}
3\hat{R}\abs{q} + 2\hat{R}\abs{z}/(1.09\alpha) = 3\hat{R}\abs{q} + 2(\hat{R}/(1.09\alpha)^2)1.09\alpha\abs{z} \leq (2\hat{R}/\alpha^2)(\abs{q} + 1.09\alpha\abs{z}).
\end{equation*}
In addition, if~$\eta=0$, then the left-hand side of~\eqref{suc} is equal to~$\abs{z}^2\leq \hat{R}(\abs{q} +1.09\alpha\abs{z})$. 
Therefore it holds, by definition~\eqref{epidef} of~$\epsilon^*$ and definition~\eqref{f0def} of~$f_0$, that
\begin{align}
&\epsilon^*\bigg(\frac{10LT^2}{1-\eta} 
+ \frac{LT^2}{10(1-\eta)}
\bigg)\bigg\|\begin{pmatrix}
z\\
v-w
\end{pmatrix}\bigg\|_{\bar{M}}^2\nonumber\\
&\quad\leq \frac{\epsilon^* LT^2}{1-\eta}
\cdot\frac{101}{10}
\bigg(\frac{2}{\alpha^2}\mathds{1}_{\eta\neq0} + \mathds{1}_{\eta=0}\bigg) \cdot
\frac{g\hat{R}^2}{1-e^{-g\hat{R}}}f_0(\abs{q} + 1.09\alpha\abs{z})\nonumber\\
&\quad\leq \frac{LT^2}{1-\eta}\cdot\frac{e^{-g\hat{R}}}{10}\cdot\frac{g\hat{R}}{1-e^{-g\hat{R}}}f_0(\abs{q} + 1.09\alpha\abs{z}).\label{frh2}
\end{align}
In addition, for the last two terms on the right-hand side of~\eqref{frh}, it holds that
\begin{equation*}
\min(\abs{\hat{q}},67/50) 
\geq (\sqrt{2\pi}/8)\min(8\abs{\hat{q}}/(\sqrt{2\pi}),4),
\end{equation*}
which by definition~\eqref{epidef} of~$\epsilon^*$ implies
\begin{align}
\frac{2\epsilon^*T^2(1+\eta)}{1-\eta}\cdot\min\bigg(\frac{8\abs{\hat{q}}}{\sqrt{2\pi}},4\bigg) &\leq \frac{\sqrt{2\pi}ge^{-g\hat{R}}}{256}\cdot\frac{T^2}{1-\eta}\cdot \min\bigg(\frac{8\abs{\hat{q}}}{\sqrt{2\pi}},4\bigg)\nonumber\\
&\leq \frac{ge^{-g\hat{R}}T^2}{32(1-\eta)}\cdot \min\bigg(\abs{\hat{q}},\frac{67}{50}\bigg).\label{hwo2}
\end{align}
Gathering~\eqref{frh2} and~\eqref{hwo2}, inserting into the right-hand side of~\eqref{frh} and recalling the definition~\eqref{cor} of~$c_0$ yields a.s. that
\begin{align*}
&\mathbb{E}[\rho^*((X',V'),(Y',W'))|x,y,v,w]\\
&\quad\leq \rho^*((x,v),(y,w)) -\frac{g\hat{R}e^{-g\hat{R}}LT^2}{10(1-e^{-g\hat{R}})(1-\eta)}f_0(\abs{q}+1.09\alpha\abs{z})\\
&\qquad
-\frac{\epsilon^* LT^2}{10(1-\eta)}
\bigg\|\begin{pmatrix}
z\\
v-w
\end{pmatrix}\bigg\|_{\bar{M}}^2,
\end{align*}
which, by~$\bar{x}e^{-\bar{x}}/(1-e^{-\bar{x}})\leq 1$ for all~$\bar{x}\geq 0$, implies
\begin{align}
&\mathbb{E}[\rho^*((X',V'),(Y',W'))|x,y,v,w]\nonumber\\
&\quad\leq \bigg(1-\frac{g\hat{R}e^{-g\hat{R}}LT^2}{10(1-e^{-g\hat{R}})(1-\eta)}\bigg)\rho^*((x,v),(y,w)).\label{ae1r}
\end{align}
Now instead of~\eqref{ae1}, assume
\begin{equation}\label{ae4}
\abs{q} + 1.09\alpha\abs{z} \geq \hat{R}.
\end{equation}
In this case~$\hat{G}=G$ holds a.s. by assumption. If both~$\eta\abs{v-w}<c_1\sqrt{L}\abs{z}$ and~$\abs{z}< c_2 R(1+L/m)$ hold, then it holds 
that
\begin{equation*}
\abs{q} = \bigg|z + \frac{\eta T}{1-\eta}(v-w)\bigg| < \bigg(1+\frac{c_1\sqrt{L}T}{1-\eta}\bigg)\abs{z} = (1+c_1\alpha^{\frac{1}{2}})\abs{z}
\end{equation*}
and consequently, 
by definition~\eqref{rhdef} of~$\hat{R}$,
\begin{equation}
\abs{q} + 1.09\alpha\abs{z} < (1+c_1\alpha^{\frac{1}{2}}+1.09\alpha)\cdot c_2R(1+L/m) \leq \hat{R}.
\end{equation}
Therefore, either~$\eta\abs{v-w}\geq c_1\sqrt{L}\abs{z}$ or~$\abs{z}\geq c_2R(1+L/m)$ holds. 
Theorem~\ref{convthm} may then be applied to obtain, using also~$\abs{q} + 1.09\alpha\abs{z}\geq\hat{R}$, a.s. that
\begin{align}
&\mathbb{E}[\rho^*((X',V'),(Y',W'))|x,y,v,w]\nonumber\\
&\quad\leq f_0(\abs{q} + 1.09\alpha\abs{z}) 
+ \epsilon^*\bigg(1-\frac{mT^2}{16(1-\eta)}\bigg)\bigg\|\begin{pmatrix}
z\\
v-w
\end{pmatrix}\bigg\|_{\bar{M}}^2.\label{tiw}
\end{align}
Moreover, the assumption~$4LT^2\leq(1-\eta)^2$ 
implies
\begin{align}
(\abs{q} + 1.09\alpha\abs{z})^2 &\leq \bigg(1+\frac{24(1.09/4)^2}{5}\bigg)\abs{q}^2 + \bigg(\frac{1.09^2}{16} + \frac{5}{24}\bigg)\abs{z}^2\nonumber\\
&\leq \bigg(1+\frac{3(1.09)^2}{10}\bigg)\abs{q}^2 + \bigg(\frac{1.09^2}{16} + \frac{5}{24}\bigg)\bigg(6\abs{z-q}^2 + \frac{6}{5}\abs{q}^2 \bigg)\nonumber\\
&\leq 2(\abs{q}^2+\abs{z-q}^2)\nonumber\\
&= 2\|(z,v-w) \|_{\bar{M}}^2,\label{ruq}
\end{align}
which, by~\eqref{ae4}, implies in particular that~$\abs{q} + 1.09\alpha\abs{z} \leq\hat{R}^{-1}(\abs{q} + 1.09\alpha\abs{z})^2 \leq 2\hat{R}^{-1}\|(z,v-w) \|_{\bar{M}}^2$. 
Therefore, for any~$\bar{c}\geq 0$, 
inequality~\eqref{tiw} implies
\begin{align}
&\mathbb{E}[\rho^*((X',V'),(Y',W'))|x,y,v,w]\nonumber\\
&\quad\leq \bigg( 1 - \frac{\epsilon^*\bar{c} mT^2}{16(1-\eta)} \bigg)f_0(\abs{q} + 1.09\alpha\abs{z}) + \frac{\epsilon^*\bar{c} mT^2 f_0(\hat{R})}{16(1-\eta)\hat{R}}(\abs{q} + 1.09\alpha \abs{z})\nonumber\\
&\qquad+ \epsilon^*\bigg(1-\frac{mT^2}{16(1-\eta)}\bigg)\bigg\|\begin{pmatrix}
z\\
v-w
\end{pmatrix}\bigg\|_{\bar{M}}^2\nonumber\\
&\quad\leq \bigg( 1 - \frac{\epsilon^*\bar{c} mT^2}{16(1-\eta)} \bigg)f_0(\abs{q} + 1.09\alpha\abs{z})\nonumber\\
&\qquad+ \epsilon^*\bigg[ 1-\frac{mT^2}{16(1-\eta)} + \frac{\bar{c} mT^2 f_0(\hat{R})}{8(1-\eta)\hat{R}^2}\bigg]\bigg\|\begin{pmatrix}
z\\
v-w
\end{pmatrix}\bigg\|_{\bar{M}}^2.\label{tji}
\end{align}
In particular if~$\bar{c}$ is given by~$\bar{c} = \hat{R}^2g/(2(1-e^{-g\hat{R}}) + \epsilon^*\hat{R}^2 g)$, 
then the expression in the square bracket on the right-hand side of~\eqref{tji} satisfies
\begin{align*}
1-\frac{mT^2}{16(1-\eta)} + \frac{2\bar{c} mT^2 f_0(\hat{R})}{16(1-\eta)\hat{R}^2} &= 1-\frac{mT^2}{16(1-\eta)}\bigg(1-\frac{2(1-e^{-g\hat{R}})}{2(1-e^{-g\hat{R}}) + \epsilon^*\hat{R}^2g}\bigg)\\
&\leq 1-\frac{\epsilon^* \bar{c}mT^2}{16(1-\eta)}.
\end{align*}
Together with the contraction rate given by~\eqref{ae1r}, this implies that~\eqref{coy} holds with~$c$ given by
\begin{equation}\label{cothrate0}
c= (T^2/(1-\eta))\min(Lg\hat{R}/(10(e^{g\hat{R}}-1)),m_0),
\end{equation}
where
\begin{align}
m_0&:=\frac{m\epsilon^* g\hat{R}^2}{16(2(1-e^{-g\hat{R}})+\epsilon^*g\hat{R}^2)} \nonumber\\
&= \frac{m}{16e^{g\hat{R}}}\cdot\frac{\min((101(2\mathds{1}_{\eta\neq 0}/\alpha^2+\mathds{1}_{\eta=0})g\hat{R})^{-1},\sqrt{2\pi}/1024) g^2\hat{R}^2}{2(1-e^{-g\hat{R}}) + \min((101(2\mathds{1}_{\eta\neq 0}/\alpha^2+\mathds{1}_{\eta=0})g\hat{R})^{-1},\sqrt{2\pi}/1024)g^2\hat{R}^2}.\label{m0def}
\end{align}
In order to bound the expression~\eqref{m0def} for~$m_0$, note first that~$\hat{R}\in [(4L)^{\frac{1}{2}},\infty)$ by definition~\eqref{rhdef} and (recall~\eqref{gdef}) so~$g\hat{R} = (4/5)\max(8L\hat{R}^2,\sqrt{L}\hat{R}) \in [8/5,\infty)$. Moreover, for any constant~$\bar{\epsilon}>0$, the functions
\begin{equation*}
[0,\infty)\ni x\mapsto \frac{\bar{\epsilon} x}{2(1-e^{-x}) + \bar{\epsilon} x} \quad\textrm{and}\quad 
[0,\infty)\ni x\mapsto \frac{\bar{\epsilon} x^2}{2(1-e^{-x}) + \bar{\epsilon} x^2} 
\end{equation*}
are both non-decreasing, which is readily verified by checking that their first derivatives at zero and their second derivatives are non-negative.
Therefore~\eqref{m0def} may bounded as
\begin{equation}\label{m0b}
m_0 \geq \frac{m}{16e^{g\hat{R}}} \cdot \frac{\min(8/(505(2\mathds{1}_{\eta\neq 0}/\alpha^2 + \mathds{1}_{\eta=0})),\sqrt{2\pi}/640)}{2(1-e^{-8/5}) + \min(8/(505(2\mathds{1}_{\eta\neq 0}/\alpha^2 + \mathds{1}_{\eta=0})),\sqrt{2\pi}/640)}
\end{equation}
The other argument in the minimum appearing in~\eqref{cothrate0} may be bounded below as~$Lg\hat{R}/(10(e^{g\hat{R}}-1)) \geq m(8/5)/(10e^{g\hat{R}})$ and this is always greater than the right-hand side of~\eqref{m0b}. Therefore, by multiplying top and bottom in~\eqref{m0b} and some simple bounds, the proof concludes.
\end{proof}
The next result gives a contraction rate in case~$4LT^2>(1-\eta)^2$. Here it is not possible to use the Lyapunov function arising from synchronous coupling and the modified norm. Instead, we assume as in~\cite[Section~2.5.3]{MR4133372} the existence of a Lyapunov function, which has already been established elsewhere in certain cases. Consequently, in place of the additive semimetric~\eqref{rsdef}, we take a multiplicative semimetric.

Let~$(G_i)_{i\in\mathbb{N}}$ 
be an i.i.d. sequence 
with~$G_i\sim\mathcal{N}(0,I_d)$. 
For any~$i\in\mathbb{N}$, let~$\hat{\theta}_i,\hat{\theta}_i'$ be~$\Theta^{T/h}$-valued r.v.'s independent 
of~$(G_i)_i$. 
Let~$(\hat{u}_i)_{i\in\mathbb{N}} = ((u_{kh}^{(i)})_{k\in\mathbb{N}})_{i\in\mathbb{N}}$ be an independent sequence such that for any~$i\in\mathbb{N}$,~$(u_{kh}^{(i)})_{k\in\mathbb{N}}$ is an independent sequence of r.v.'s independent 
of~$(\hat{\theta}_i)_i,(\hat{\theta}_i')_i,(G_i)_i$ and such that~$u_{kh}^{(i)}\sim\mathcal{U}(0,1)$ for all~$k\in\mathbb{N}$. Denote~$\bar{\theta}_i:=(\hat{\theta}_i,\hat{\theta}_i',u,\hat{u}_i)$. For any~$\mathbb{R}^d$-valued r.v.'s~$x$ and~$v$ both independent 
of~$(G_i)_i,(\bar{\theta}_i)_i,((u_{kh}^{(i)})_k)_i$ and any~$i\in\mathbb{N}$, let~$X^{(i)},V^{(i)}$ be given by~$X^{(0)}=x$,~$V^{(0)}=v$, then inductively by~\eqref{ous} with~$X^{(i+1)},V^{(i+1)},X^{(i)},V^{(i)},G_i,\hat{G}_i,\bar{\theta}_i$ replacing~$X',V',x,v,G,\hat{G},\bar{\theta}$ respectively. 
Let~$Z^{(i)},Q^{(i)},\hat{q}^{(i)}$ be given by~\eqref{zdefs} and let~$\hat{G}_i$ be given by~\eqref{cou1},~\eqref{edef} all with the obvious replacements, with~$r^*=(1-\eta)/(T^2(1+\eta))$, with~$\gamma$ given by~\eqref{gdef} and with~$\mathcal{U}$ replaced by an i.i.d. sequence~$(\mathcal{U}_i)_i$ independent of~$(G_i)_i,(\bar{\theta}_i)_i,((u_{kh}^{(i)})_k)_i,x,v$. 
For any~$\mathbb{R}^d$-valued r.v.'s~$y,w\in\mathbb{R}^d$ independent 
of~$(G_i)_i,(\hat{G}_i)_i,(\mathcal{U}_i)_i,(\bar{\theta}_i)_i,((u_{kh}^{(i)})_k)_i$, 
let~$Y^{(i)},W^{(i)}$ be defined analogously as for~$X^{(i)},V^{(i)}$ but with~$\hat{G}_i$ in place of~$G_i$.
\begin{assumption}\label{A2}
Inequalities~\eqref{A1a} and~\eqref{carot} hold. 
Moreover, there exist~$c_1\in(0,1]$,~$c_2\in\mathbb{R}$,~$V:\mathbb{R}^d\times\mathbb{R}^d\rightarrow [0,\infty)$ and~$n\in\mathbb{N}\setminus\{0\}$ such that for any~$x,v\in\mathbb{R}^d$, it holds that~$\mathbb{E}[V(X^{(n)},V^{(n)})] \leq (1-c_1) V(x,v) + c_2$, 
and the set~$\{x,v\in\mathbb{R}^d:V(x,v)\leq 4c_2/c_1\}$ is compact.
\end{assumption}
By Theorem~5 in~\cite{durmus2023uniform}, Assumption~\ref{A2} is satisfied for~$K=1$,~$\eta\neq 0$. By Theorem~3 in~\cite{camrud2023second}, it is also satisfied in a case where the force field~$b$ is conservative (there exists~$U$ such that~$b=\nabla U$) and deterministic ($\Theta=\{\theta\}$). 

\begin{remark}\label{deprem}
In Assumption~\ref{A2}, no information is given about the dependence of~$c_1,c_2$ on the parameters~$m,R,d$, which is existing in the aforementioned references. In particular, in contrast to the Lyapunov inequality~\eqref{mcon2} satisfied by the modified norm, the constant~$c_2$ 
typically depends explicitly on~$d$. Moreover, a further step-size restriction may be 
implicitly assumed by taking Assumption~\ref{A2}. As a result, the one-step convergence rate generally worsens in terms of the dependency on~$m$ from that present in Theorem~\ref{coth}, despite the~$O(\sqrt{m})$ rate for certain friction values in the Gaussian and continuous time cases as mentioned at the beginning of Section~\ref{convex}. 
\end{remark}

Under Assumption~\ref{A2}, let~$\epsilon$ be given by
\begin{equation}\label{epidef2}
\epsilon = c_0/(4c_2),
\end{equation}
where~$c_0$ is defined by~\eqref{cor}, 
let~$\hat{R}'$ be given by
\begin{equation}
\hat{R}' = 4\sup\{(1+\alpha)\abs{x},\gamma^{-1}\abs{v}:V(x,v)\leq 4c_2/c_1\},
\end{equation}
where~$\alpha= 1.09LT^2/(1-\eta)^2$, 
let~$f_1:[0,\infty)\rightarrow[0,\infty)$ be given by~\eqref{f0def} with~$f_1,\hat{R}'$ replacing~$f_0,\hat{R}$ 
and let~$\rho:\mathbb{R}^{2d}\times\mathbb{R}^{2d}\rightarrow[0,\infty)$ be given by
\begin{equation*}
\rho((x,v),(y,w)) = \sqrt{f_1(\abs{q} + 1.09\alpha\abs{z})(1 + \epsilon V(x,v) + \epsilon V(y,w))}.
\end{equation*}
The proof of the next Theorem~\ref{coth2} follows closely to that of Theorem~2.7 in~\cite{MR4133372}.
\begin{theorem}\label{coth2}
Let Assumption~\ref{A2} hold. 
Assume~\eqref{A1a},~\eqref{carot} and~$4LT^2> (1-\eta)^2$. Let~$x,y,v,w$ be~$\mathbb{R}^d$-valued r.v.'s independent of~$(G_i)_i,(\hat{G}_i)_i,(\bar{\theta}_i)_i,((u_{kh}^{(i)})_k)_i$. It holds a.s. that
\begin{equation*}
\mathbb{E}[\rho((X^{(n)},V^{(n)}),(Y^{(n)},W^{(n)}))|x,y,v,w] \leq(1-\min(c_0,c_1)/8)\rho((x,v),(y,w)),
\end{equation*}
where~$c_0$ is given by~\eqref{cor}.
\end{theorem}
\begin{proof}
Firstly, assume
\begin{equation}\label{as1}
\abs{q} + 1.09\alpha\abs{z}\leq\hat{R}'.
\end{equation}
By conditional H\"older inequality, it holds a.s. that
\begin{align}
&\mathbb{E}[\rho((X^{(n)},V^{(n)}),(Y^{(n)},W^{(n)}))|x,y,v,w]\nonumber\\
&\quad\leq \mathbb{E}[f_1(\abs{Q^{(n)}} + 1.09\alpha \abs{Z^{(n)}})|x,y,v,w]^{\frac{1}{2}}\nonumber\\
&\qquad\cdot\mathbb{E}[1+\epsilon V(X^{(n)},V^{(n)}) + \epsilon V(Y^{(n)},W^{(n)})|x,y,v,w]^{\frac{1}{2}}.\label{rfv}
\end{align}
For each~$i\in[1,n-1]\cap\mathbb{N}$, in either cases~$\abs{Q^{(i)}} + 1.09\alpha \abs{Z^{(i)}} \leq \hat{R}'$ and~$\abs{Q^{(i)}} + 1.09\alpha \abs{Z^{(i)}} > \hat{R}'$, by either 
Theorem~\ref{main2} or~$\max f_1 = f_1(\abs{Q^{(i)}}+1.09\alpha\abs{Z^{(i)}})$, it holds a.s. that~$\mathbb{E}[f_1(\abs{Q^{(i+1)}} + 1.09\alpha \abs{Z^{(i+1)}})|X^{(i)},Y^{(i)},V^{(i)},W^{(i)}] \leq f_1(\abs{Q^{(i)}} + 1.09\alpha \abs{Z^{(i)}})$. 
By the tower property and again 
Theorem~\ref{main2}, this implies a.s. that
\begin{equation}
\mathbb{E}[f_1(\abs{Q^{(n)}} + 1.09\alpha \abs{Z^{(n)}})|x,y,v,w] \leq (1-c_0)f_1(\abs{q} + 1.09\alpha \abs{z}).
\end{equation}
Therefore, by Assumption~\ref{A2} and the definition~\eqref{epidef2} of~$\epsilon$, inequality~\eqref{rfv} implies a.s. that
\begin{align}
&\mathbb{E}[\rho((X^{(n)},V^{(n)}),(Y^{(n)},W^{(n)}))|x,y,v,w]\nonumber\\
&\quad\leq (1-c_0)^{\frac{1}{2}}(1+2\epsilon c_2)^{\frac{1}{2}}\rho((x,v),(y,w))\nonumber\\
&\quad\leq(1-c_0/4)\rho((x,v),(y,w)).\label{fcr}
\end{align}
Instead of~\eqref{as1}, assume
\begin{equation}\label{as2}
\abs{q} + 1.09\alpha\abs{z}>\hat{R}'.
\end{equation}
If all of the inequalities
\begin{align}
(1+1.09\alpha)\abs{x}&\leq\hat{R}'/4,\qquad(1+1.09\alpha)\abs{y}\leq\hat{R}'/4,\nonumber\\
\gamma^{-1}\abs{v}&\leq\hat{R}'/4,\qquad\gamma^{-1}\abs{w}\leq\hat{R}'/4\label{ote}
\end{align}
hold, then it holds that
\begin{equation*}
\abs{q}+1.09\alpha\abs{z} \leq (1+1.09\alpha)(\abs{x} + \abs{y}) + \gamma^{-1}(\abs{v} + \abs{w}) \leq \hat{R}',
\end{equation*}
which contradicts~\eqref{as2}. Consequently, one of the inequalities in~\eqref{ote} fails to hold and
\begin{align*}
&\mathbb{E}[1+\epsilon V(X^{(n)},V^{(n)}) + \epsilon V(Y^{(n)},W^{(n)})|x,y,v,w]\\
&\quad\leq 1+\epsilon((1-c_1) V((x,v)) + (1-c_1)V((y,w)) + 2c_2)\\
&\quad\leq 1+\epsilon\bigg(\bigg(1-\frac{c_1}{4}\bigg) (V((x,v)) + V((y,w))) - c_2\bigg)\\
&\quad\leq \bigg(1-\min\bigg(\frac{c_1}{4},\epsilon c_2\bigg)\bigg)(1+\epsilon V(x,v) + \epsilon V(y,w)).
\end{align*}
Therefore, by~$\max f_1 = f_1(\abs{q}+1.09\alpha\abs{z})$ and the definition~\eqref{epidef2} of~$\epsilon$, it holds that
\begin{align*}
\mathbb{E}[\rho((X^{(n)},V^{(n)}),(Y^{(n)},W^{(n)}))|x,y,v,w] &\leq (1-\min(c_1/4,\epsilon c_2))^{\frac{1}{2}}\rho((x,v),(y,w))\\
&\leq (1-\min(c_1/8,c_0/8))\rho((x,v),(y,w)).
\end{align*}
Together with the contraction rate given in~\eqref{fcr}, the proof concludes.
\end{proof}

\section{Consequences for empirical averages}\label{conseqs}
Only the consequences of Theorem~\ref{coth} are given, where~$4LT^2\leq (1-\eta)^2$ holds. The main underlying Corollary~\ref{maincore} shows that~$\mathcal{W}_1$ contraction is possible, that is, a convergence bound on~$\mathcal{W}_1$ with a right-hand side in terms of only~$\mathcal{W}_1$. 
As consequence of this contraction result, in Corollary~\ref{gaucor_old} of Section~\ref{gausec}, Gaussian concentration bounds on ergodic averages are given. 
In addition, for the case where~$b$ is a stochastic approximation of some~$\nabla U$, a bound on the bias of the estimator to the average w.r.t. the probability measure with density proportional to~$e^{-U}$ is given in Section~\ref{bisec} by Corollary~\ref{bicor_old}.

For the rest of this section, 
recall the notations at the beginning of Section~\ref{semco}, the definitions just before Assumption~\ref{A2} 
and that~$\mathcal{W}_1,\mathcal{W}_2$ denote the~$L^1$,~$L^2$ Wasserstein distance with respect to the twisted Euclidean metric~$\hat{d}$ given by~\eqref{tmet}. Moreover, let~$\mathcal{W}_{2,e}$ denote the~$L^2$ Wasserstein distance w.r.t. the distance~$\mathbb{R}^{2d}\times\mathbb{R}^{2d}\ni((x,v),(y,w))\mapsto \sqrt{\abs{x-y}^2 + \abs{v-w}^2/L}$. 
For any distribution~$\nu$ of an~$\mathbb{R}^{2d}$-valued r.v.~($\mathbb{R}^d$-valued when~$\gamma=\infty$) that is independent of~$(G_i)_i,(\hat{G}_i)_i,(\bar{\theta}_i)_i,((u_{kh}^{(i)})_k)_i$, recall~$\pi_n(\nu)$ to denote the distribution after~$n$ iterations of SGgHMC; more precisely,~$\pi_n(\nu)$ denotes the distribution of~$(X^{(i)},V^{(i)})$ for~$(x,v)\sim \nu$~(and of~$X^{(i)}$ for~$x\sim\nu$ if~$\gamma=\infty$). 

By the triangle inequality, the~$L_1$ Wasserstein distance (w.r.t. the twisted metric given by~\eqref{tmet} in particular) may be split into a sum of many smaller intermediate distances. Using modified distance as a Lyapunov function in~\eqref{rsdef} allows one to capitalize on this. More specifically, 
we rely 
in the following Corollary~\ref{maincore} 
on the Lyapunov term being second order in the distance. 
\begin{corollary}\label{maincore}
Let Assumption~\ref{A1} hold, let~$x,y,v,w$ be~$\mathbb{R}^d$-valued r.v.'s independent of~$(G_i)_i,(\hat{G}_i)_i,(\bar{\theta}_i)_i,((u_{kh}^{(i)})_k)_i$ such that~$\mathbb{E}[\abs{x}^2+\abs{y}^2+\abs{v}^2+\abs{w}^2]<\infty$ holds. 
Assume~$4LT^2\leq (1-\eta)^2$ and~\eqref{carot}. 
For any~$n\in\mathbb{N}$, it holds that
\begin{align*}
\mathcal{W}_1(\pi_n(\delta_{(x,v)}),\pi_n(\delta_{(y,w)})) &\leq \bigg(\frac{g}{2\log(2)\epsilon^*}\bigg)^{\frac{1}{2}}(1-c)^n \mathcal{W}_1(\delta_{(x,v)},\delta_{(y,w)}),\\
\mathcal{W}_2(\pi_n(\delta_{(x,v)}),\pi_n(\delta_{(y,w)}))^2 &\leq \frac{2(1-c)^n}{\epsilon^*} \mathcal{W}_{\rho^*}(\delta_{(x,v)},\delta_{(y,w)}) 
\end{align*}
where~$\delta_{(x,v)},\delta_{(y,w)}$ denote the distributions of~$(x,v)$ and~$(y,w)$ respectively,~$c,g$ are given by~\eqref{cothrate},~\eqref{gdef??},~\eqref{rhdef} and~$\rho^*$,~$\epsilon^*$ and~$f_0$ are given by~\eqref{rsdef},~\eqref{epidef} and~\eqref{f0def} respectively along with~\eqref{gdef??}.
\end{corollary}
\begin{proof}
For the rest of the proof, fix~$n\in\mathbb{N}$. 
It holds that
\begin{align}
&W_1(\pi_n(\delta_{(x,v)}),\pi_n(\delta_{(y,w)})) \nonumber\\
&\quad\leq \mathbb{E}[\mathbb{E}[1.09\alpha\abs{Z^{(n)}} + \abs{Q^{(n)}}|X^{(n-1)},V^{(n-1)},Y^{(n-1)},W^{(n-1)}]]\nonumber\\
&\quad\leq \mathbb{E}\bigg[\mathbb{E}\bigg[\bigg(\frac{1.09\alpha\abs{Z^{(n)}} + \abs{Q^{(n)}}}{\rho^*((X^{(n)},V^{(n)}),(X^{(n)},V^{(n)}))}\bigg)\nonumber\\
&\qquad\cdot\rho^*((X^{(n)},V^{(n)}),(Y^{(n)},W^{(n)}))\bigg|X^{(n-1)},V^{(n-1)},Y^{(n-1)},W^{(n-1)}\bigg]\bigg],\label{wfi}
\end{align}
where~$0/0:=0$ here. 
Using~\eqref{ruq}, denoting~$\hat{X}^{(n)} = 1.09\alpha\abs{Z^{(n)}} + \abs{Q^{(n)}}$, it holds that
\begin{align}
\frac{\hat{X}^{(n)}}{\rho^*((X^{(n)},V^{(n)}),(Y^{(n)},W^{(n)}))} &= \frac{\hat{X}^{(n)}}{f_0(\hat{X}^{(n)}) + \epsilon^*\|(Z^{(n)}, V^{(n)} - W^{(n)})\|_{\bar{M}}^2}\nonumber\\
&\leq \frac{\hat{X}^{(n)}}{f_0(\hat{X}^{(n)}) + \epsilon^*(\hat{X}^{(n)})^2/2}.\label{zai}
\end{align}
Since it holds that~$f_0(\bar{x})=(1-e^{-g\bar{x}})/g\geq \min(\bar{x}/2,\log(2)/g)$, the right-hand side of~\eqref{zai} may be bounded above by the expression
\begin{equation*}
\max((1/2+\epsilon^*\hat{X}^{(n)}/2)^{-1},\hat{X}^{(n)}(\log(2)/g + \epsilon^*(\hat{X}^{(n)})^2/2)^{-1}),
\end{equation*}
which may in turn be bounded above to obtain
\begin{equation*}
\frac{\hat{X}^{(n)}}{\rho^*((X^{(n)},V^{(n)}),(Y^{(n)},W^{(n)}))}\leq \max\bigg(2,\sqrt{\frac{g}{2\log (2)\epsilon^*}}\bigg)\leq\sqrt{\frac{g}{2\log (2)\epsilon^*}}
\end{equation*}
where we have used the definitions~\eqref{gdef??} of~$g$ and~\eqref{epidef} of~$\epsilon^*$. 
Therefore, by choosing appropriately between reflection and synchonous coupling, inequality~\eqref{wfi} and Theorem~\ref{coth} imply that 
\begin{align}
&\mathcal{W}_1(\pi_n(\delta_{(x,v)}),\pi_n(\delta_{(y,w)})) \nonumber\\
&\quad\leq \bigg(\frac{g}{2\log(2)\epsilon^*}\bigg)^{\frac{1}{2}}(1-c)\mathbb{E}[\rho^*((X^{(n-1)},V^{(n-1)}),(Y^{(n-1)},W^{(n-1)}))]\nonumber\\
&\quad\leq \bigg(\frac{g}{2\log(2)\epsilon^*}\bigg)^{\frac{1}{2}}(1-c)^n\mathbb{E}[\rho^*((x,v),(y,w))].\label{tae}
\end{align}
For any~$j,N\in\mathbb{N}$ with~$j< N$, by applying inequality~\eqref{tae} with~$x,y,v,w$ replaced respectively by
\begin{align*}
x&\leftarrow x_j:=x+j(y-x)/N,\qquad y\leftarrow x_{j+1}:=x+(j+1)(y-x)/N,\\
v&\leftarrow v_j:=v+j(v-w)/N,\qquad w\leftarrow v_{j+1}:=v+(j+1)(v-w)/N,
\end{align*}
the triangle inequality yields
\begin{align}
\mathcal{W}_1(\pi_n(\delta_{(x,v)}),\pi_n(\delta_{(y,w)})) &\leq \sum_{j=0}^{N-1}\mathcal{W}_1(\pi_n(\delta_{(x_j,v_j)}),\pi_n(\delta_{(x_{j+1},v_{j+1})}))\nonumber\\
&\leq \sum_{j=0}^{N-1}\bigg(\frac{g}{2\log(2)\epsilon^*}\bigg)^{\frac{1}{2}}(1-c)^n\mathbb{E}[\rho^*((x_j,v_j),(x_{j+1},v_{j+1}))].\label{nte}
\end{align}
Since~\eqref{nte} holds for any~$N$, by considering~$N\rightarrow\infty$ and the forms~\eqref{rsdef},~\eqref{f0def} of~$\rho^*,f_0$, it holds by~$f_0(\bar{x})\leq \bar{x}$, dominated convergence and the assumption~$\mathbb{E}[\abs{x}^2+\abs{y}^2 + \abs{v}^2+\abs{w}^2]<\infty$ that
\begin{equation}\label{w1t}
\mathcal{W}_1(\pi_n(\delta_{(x,v)}),\pi_n(\delta_{(y,w)}))\leq (g/(2\log(2)\epsilon^*))^{\frac{1}{2}}(1-c)^n \mathbb{E}[\abs{q}+\alpha\abs{z}].
\end{equation}
Therefore, since~\eqref{w1t} holds for any coupling between~$\delta_{(x,v)}$ and~$\delta_{(y,w)}$, the first assertion follows.
For the second assertion, by~\eqref{ruq}, it holds for any~$\bar{x},\bar{y},\bar{v},\bar{w}\in\mathbb{R}^d$ with~$\bar{q},\bar{z}\in\mathbb{R}^d$ defined correspondingly by~\eqref{zdefs} that~$(\abs{\bar{q}} + 1.09\alpha\abs{\bar{z}})^2 \leq 2\|(\bar{z},\bar{v}-\bar{w}) \|_{\bar{M}}^2 \leq 2(\epsilon^*)^{-1}\rho^*((\bar{x},\bar{v}),(\bar{y},\bar{w}))$ 
and therefore
\begin{align*}
\mathcal{W}_2(\pi_n(\delta_{(x,v)}),\pi_n(\delta_{(y,w)}))^2 &\leq 2(\epsilon^*)^{-1}\mathcal{W}_{\rho^*}(\pi_n(\delta_{(x,v)}),\pi_n(\delta_{(y,w)}))\\
&\leq 2(\epsilon^*)^{-1}(1-c)^n\mathbb{E}[\rho^*((x,v),(y,w))],
\end{align*}
whereby taking the infimum over all couplings concludes the proof.
\end{proof}

\subsection{Gaussian concentration}\label{gausec}


The main goal of this section is to prove Corollary~\ref{gaucor_old}, which gives the announced Corollary~\ref{gaucor}. In order to do this, we prove in Lemma~\ref{conlem} that conditional distributions of the SGgHMC chain satisfy log-Sobolev inequalities and thus transportation cost-information inequalities (by Otto-Villani). These inequalities are consequences of the fact that the conditional distributions of SGgHMC are pushforwards of Gaussian distributions by Lipschitz maps and that they are absolutely continuous, as shown in Lemma~\ref{bbcon} and Lemma~\ref{abcon}. All of the proofs in this section may be found in Appendix~\ref{gauapp}.

Let~$\gamma_{0,\sigma^2}^d$ denote the Gaussian measure on~$\mathbb{R}^d$ with mean~$0$ and variance~$\sigma^2I_d$. 
For any~$i\in\{1,2,3\}$, let the function~$h^{(i)}:\mathbb{R}^{2d}\rightarrow\mathbb{R}^{2d}$ be 
given by~$h^{(i)}(x,v)=(\bar{q}_T(x,L^{\frac{1}{2}}v,\bar{\theta}_i),L^{-\frac{1}{2}}\bar{p}_T(x,L^{\frac{1}{2}}v,\bar{\theta}_i))$. Moreover, let~$h^{(0)}:\mathbb{R}^{3d}\rightarrow\mathbb{R}^{2d}$ be given by~$h^{(0)}(x,v,g) = (x,\eta v + g)$ 
and for any~$x,v\in\mathbb{R}^d$, let~$\hat{h}_{x,v}:\mathbb{R}^d\times\mathbb{R}^d\rightarrow\mathbb{R}^d\times\mathbb{R}^d$,~$\hat{h}_{x,v}':\mathbb{R}^d\times\mathbb{R}^d\times\mathbb{R}^d\rightarrow\mathbb{R}^d\times\mathbb{R}^d$ be given by
\begin{align*}
\hat{h}_{x,v}(g,g')&= h^{(2)}(h^{(0)}(h^{(1)}(x,\eta v+g),g')),\\
\hat{h}_{x,v}'(g,g',g'')&= h^{(3)}(h^{(0)}(h^{(2)}(h^{(0)}(h^{(1)}(x,\eta v+g),g')),g'')).
\end{align*}
In the special case~$\eta=0$, it will be more useful to consider for any~$x\in\mathbb{R}^d$ the mapping~$\hat{h}_x^{(0)}:\mathbb{R}^d\rightarrow\mathbb{R}^d$ given by~$\hat{h}_x^{(0)}(g)=\textrm{proj}_1(h^{(1)}(x,g))$, 
where~$\textrm{proj}_1:\mathbb{R}^{2d}\rightarrow\mathbb{R}^d$ is the projection to the first~$d$ coordinates given by~$\textrm{proj}_1((\bar{x},\bar{v}))=\bar{x}$. 
\begin{lemma}\label{bbcon}
Assume~\eqref{A1a} and~\eqref{carot}. For any~$x,v\in\mathbb{R}^d$,~$\nabla \hat{h}_{x,v},\nabla \hat{h}_{x,v}'$ exist almost everywhere and for any~$z_1,z_2\in\mathbb{R}^{2d}$,~$z_1'\in\mathbb{R}^{3d}$, they satisfies a.e. that
\begin{equation}\label{ddev}
\abs{(z_1\cdot\nabla) (\hat{h}_{x,v}\cdot z_2)}^2 \leq 4\abs{z_1}^2\abs{z_2}^2,\qquad
\abs{(z_1'\cdot\nabla) (\hat{h}_{x,v}'\cdot z_2)}^2 \leq 8\abs{z_1'}^2\abs{z_2}^2.
\end{equation}
If~$\eta=0$, then for any~$x\in\mathbb{R}^d$,~$\nabla\hat{h}_x^{(0)}$ exists a.e. and satisfies a.e. that
\begin{equation}\label{ddev2}
\abs{(z_1\cdot\nabla)(\hat{h}_x^{(0)}\cdot z_2)}^2\leq 2LT^2 \abs{z_1}^2\abs{z_2}^2.
\end{equation}
\end{lemma}

\begin{lemma}\label{abcon}
Assume~\eqref{A1a} and~\eqref{carot}. 
For any~$x,v\in\mathbb{R}^d$, the pushforward of~$\gamma_{0,(1-\eta^2)/L}^{2d}$ by the map~$\hat{h}_{x,v}$ and the pushforward of~$\gamma_{0,(1-\eta^2)/L}^{3d}$ by the map~$\hat{h}_{x,v}'$ are absolutely continuous. Moreover, if~$\eta=0$, then the pushforward of~$\gamma_{0,1/L}^d$ by~$\hat{h}_x^{(0)}$ is absolutely continuous.
\end{lemma}

The measures in Lemma~\ref{abcon} are of course essentially the distributions of the SGgHMC chains after two and three iterations (one iteration in the~$\eta=0$ case), see the accompanying Remark~\ref{conrem}. The next Lemma~\ref{conlem} states that they satisfy the~$L^1$ transportation cost-information inequality~\cite{MR2078555} on~$(\mathbb{R}^{2d},\hat{d})$ (and on~$(\mathbb{R}^d,\hat{d})$ in case~$\eta=0$). 

For a measure~$\mu$ on~$\mathbb{R}^{2d}$, let~$\mu^*$ denote the pushforward of~$\mu$ by~$\mathbb{R}^{2d}\ni(\bar{x},\bar{v})\mapsto(\bar{x},L^{\frac{1}{2}}\bar{v})$. For a probability measure~$\mu$ and constant~$\bar{C}>0$, we say that~$\mu\in T_1(\bar{C})$ if for any probability measure~$\mu'$, it holds that~$\mathcal{W}_1(\mu,\mu') \leq \sqrt{2\bar{C}\textrm{Ent}(\mu'|\mu)}$.
\begin{lemma}\label{conlem}
Assume~\eqref{A1a} and~\eqref{carot}. For any~$x,v\in\mathbb{R}^d$, let~$\nu_{x,v}$ denote the pushforward of~$\gamma_{0,(1-\eta^2)/L}^{2d}$ by the map~$\hat{h}_{x,v}$ and let~$\bar{\nu}_{x,v}$ denote the pushforward of~$\gamma_{0,(1-\eta^2)/L}^{3d}$ by~$\hat{h}_{x,v}'$. Moreover, let~$\nu_x$ denote the pushforward of~$\gamma_{0,1/L}^d$ by the map~$\hat{h}_x^{(0)}$. It holds that
\begin{enumerate}[label=(\roman*)]
\item~$\nu_{x,v}^*\in T_1(4(1.09/4+1)^2(1-\eta^2)/L)$, \label{conlem1}
\item~$\bar{\nu}_{x,v}^*\in T_1(8(1.09/4+1)^2(1-\eta^2)/L)$. \label{conlem2}
\end{enumerate}
If~$\eta=0$, then it holds that~$\nu_x\in T_1((1.09/4+1)^2T^2)$.
\end{lemma}
\begin{remark}\label{conrem}
In the construction of~$\nu_{x,v}^*$ and~$\nu_x$, the parameters~$\bar{\theta}_1,\bar{\theta}_2$ are arbitrary. Therefore for any~$x,v\in\mathbb{R}^n$ and~$i\in\mathbb{N}$, the measure~$\nu_{x,v}^*$ may be identified accordingly as the regular conditional distribution for~$(X^{(i+2)},V^{(i+2)})$ given~$(X^{(i)},V^{(i)}) = (x,v)$ as in~\cite[Chapter~4.2-4.3, p.79]{MR1163370}\footnote{Alternatively,~$\nu_{x,v}^*$ is~$\mu((x,v),\cdot)$ for the probability kernel~$\mu$ given by Theorem~8.5 in~\cite{MR4226142} with~$\xi=(X^{(i)},V^{(i)})$,~$\eta = (X^{(i+2)},V^{(i+2)})$.}. If~$\eta=0$, then for suitable~$\bar{\theta}_1$,~$\nu_x$ is the regular conditional distribution for~$X^{(i+1)}$ given~$X^{(i)}=x$.
\end{remark}

We are now ready to precisely state concentration inequalities for empirical averages of the chain. The following Corollary~\ref{gaucor_old} is a more detailed version of Corollary~\ref{gaucor}. Its proof strategy is mostly known. The assumption~$N_0\neq0$ made in Corollary~\ref{gaucor_old} is a technical one, which arises from our requirement for the absolute continuity of conditional measures as in Lemma~\ref{abcon}. 

\begin{corollary}\label{gaucor_old}
Let Assumption~\ref{A1} hold, let~$x,v\in\mathbb{R}^d$ 
and let~$f:\mathbb{R}^d\rightarrow\mathbb{R}$ be Lipschitz. 
Assume~$4LT^2\leq(1-\eta)^2$ and~\eqref{carot}. Moreover, let~$c,\epsilon^*$ be given as in Corollary~\ref{maincore} and~$g$ be as in~\eqref{gdef??}. 
Let~$N_0\in\mathbb{N}\setminus\{0\}$ and if~$\eta=0$, then let~$N_0\in\mathbb{N}$. For any~$N\in\mathbb{N}$,~$t>0$, it holds a.s. that
\begin{equation*}
\mathbb{P}\bigg(\frac{1}{N}\sum_{i=N_0+1}^{N_0+N} (f(X^{(i)}) - \mathbb{E}[f(X^{(i)})]) > t\bigg) \leq \exp\bigg(-\frac{ (1-\eta)Nc^2t^2}{CT^2\|f\|_{\textrm{Lip}(\hat{d})}^2(1+(cN)^{-1})}\bigg),
\end{equation*}
where~$C$ is given by
\begin{equation}\label{csi_old}
C= \begin{cases}
19g(1+\eta)/(\alpha\epsilon^*) & \textrm{if }\eta\neq 0\\
3g/\epsilon^* & \textrm{if }\eta=0.
\end{cases}
\end{equation}
\end{corollary}

\subsection{Empirical average bias}\label{bisec}
In this section, 
we fix~$U\in C^2(\mathbb{R}^d)$ satisfying~$\nabla U(0) = 0$ and~\eqref{A10},~\eqref{A1a} both with~$b(\cdot,\theta)$ replaced by~$\nabla U$. Moreover, recall that~$\mu$ denotes the probability measure with density proportional to~$\mathbb{R}^{2d}\ni(x,v)\mapsto e^{-U(x)-\abs{v}^2/2}$ and recall that~$\hat{d}$ is given by~\eqref{tmet}.
Corollary~\ref{invd} below gives the~$L^1$ Wasserstein distance w.r.t. the twisted metric~$\hat{d}$ between the target measure~$\mu$ and the invariant measure~$\tilde{\mu}$ for the unadjusted, non-stochastic gHMC chain. Proposition~\ref{invd2} then gives the same distance between~$\tilde{\mu}$ and~$\pi_n(\tilde{\mu})$ given suitable assumptions on the stochastic gradient. 
Note that 
Corollary~\ref{maincore} and~\cite[Proposition~2.10]{MR2078555} imply the existence of unique invariant measures for~$\bar{n}$ and~$\bar{n}+1$ successive non-stochastic SGgHMC iterations for large enough~$\bar{n}$, which must coincide and imply the existence of~$\tilde{\mu}$ (in the application of~\cite[Proposition~2.10]{MR2078555}, the~$T_1(C)$ property follows by and similarly to Lemma~\ref{abcon},~\ref{conlem}). The proofs of Corollary~\ref{invd}, Proposition~\ref{invd2} and Corollary~\ref{bicor_old} follow closely to previous proofs appearing in the literature. 
These proofs can be found in Appendix~\ref{spw}. 

\begin{corollary}\label{invd}
Let Assumption~\ref{A1} hold. 
Assume~$4LT^2\leq (1-\eta)^2$ and~\eqref{carot}. 
In the case~$u=0$, where the velocity Verlet integrator is used, it holds that
\begin{equation*}
\mathcal{W}_1(\mu,\tilde{\mu})\leq 
16\bigg(\frac{2g}{\log(2)\epsilon^*}\bigg)^{\sqrt{L}T/c} \cdot h\max\bigg(\frac{dL}{m},\frac{(LR')^2}{m}\bigg)^{\frac{1}{2}}
\end{equation*}
where~$c,\epsilon^*,f_0$ are given as in Corollary~\ref{maincore} and~$R'=4R(1+\frac{L}{m})$. 
In the case~$u=1$, where the randomized midpoint integrator is used, it holds that
\begin{equation*}
\mathcal{W}_1(\mu,\tilde{\mu})\leq 144\sqrt{6}T^{-1}\bigg(\frac{2g}{\log(2)\epsilon^*}\bigg)^{\sqrt{L}T/c}\cdot L^{-1/4}h^{3/2}\max\bigg(\frac{dL}{m},\frac{(LR')^2}{m}\bigg)^{\frac{1}{2}}.
\end{equation*}
For~$u=0$, if in addition there exists~$L_2>0$ such that~$\abs{\nabla^2 U(x)-\nabla^2 U(x')}\leq L_2 \abs{x-x'}$ for all~$x,x'\in\mathbb{R}^d$, then it holds that
\begin{equation*}
\mathcal{W}_1(\mu,\tilde{\mu})\leq 66\bigg(1+\frac{L_2}{L^{3/2}}\bigg)\bigg(\frac{2g}{\log(2)\epsilon^*}\bigg)^{\sqrt{L}T/c} \cdot \sqrt{L}h^2\max\bigg(\frac{dL}{m},\frac{(LR')^2}{m}\bigg).
\end{equation*}
\end{corollary}

\begin{assumption}\label{A3}
There exist~$p\in\mathbb{N}\setminus\{0\}$, a measurable index set~$\Theta_0$ and a family of functions~$(U_{\vartheta})_{\vartheta\in\Theta_0}$ satisfying
\begin{enumerate}
\item for any~$x\in\mathbb{R}^d$, the function~$\vartheta\mapsto\nabla U_{\vartheta}(x)$ is measurable,
\item for any~$\vartheta\in\Theta_0$, it holds that~$U_{\vartheta}\in C^2(\mathbb{R}^d)$ and~$\abs{\nabla U_{\vartheta}(x)-\nabla U_{\vartheta}(y)}\leq L\abs{x-y}$ for all~$x,y\in\mathbb{R}^d$,
\item for each~$\theta\in\Theta$, there exist~$\vartheta_1 
,\dots,\vartheta_p\in\Theta_0$ satisfying~$b(x,\theta) = \frac{1}{p}\sum_{i=1}^p \nabla U_{\vartheta_i}(x)$ for all~$x\in\mathbb{R}^d$.
\end{enumerate}
\end{assumption}
Let~$C_U>0$. Under Assumption~\ref{A3}, a~$\Theta$-valued r.v.~$\theta$ will be called unbiased if all of the following conditions hold:
\begin{enumerate}
\item there exist measurable~$\vartheta_1,\dots,\vartheta_p:\Theta\rightarrow\Theta_0$ from Assumption~\ref{A3} such that~$\vartheta_1(\theta),\dots,\vartheta_p(\theta)$ are pairwise independent,
\item it holds that~$\mathbb{E}[\nabla U_{\vartheta_i(\theta)}(x)] = \nabla U(x)$, for all~$x\in\mathbb{R}^d$ and~$i\in\mathbb{N}\cap[1,p]$ (or equivalently,~$i\in\{1\}$),
\item it holds that~$\mathbb{E}[\abs{\nabla U_{\vartheta_1(\theta)}(x) - \nabla U(x)}^2] < C_U$ for all~$x\in\mathbb{R}^d$. 
\end{enumerate}

\begin{prop}\label{invd2}
Let Assumption~\ref{A3},~\eqref{A1a} and~\eqref{carot} hold. Suppose~$(\hat{\theta}_i)_i,(\hat{\theta}_i')_i$ are sequences of tuples of unbiased~$\Theta$-valued r.v.'s. 
For any~$n\in\mathbb{N}$, it holds that~$\mathcal{W}_{2,e}(\pi_n(\tilde{\mu}),\tilde{\mu})\leq (C_U\sqrt{L}h(1+3\sqrt{L}h)^{nK}/p)^{\frac{1}{2}}$.
\end{prop}

Gathering the previous results, we can now state a more precise version of Corollary~\ref{bicor} to provide an explicit estimate on the bias of the estimator as in Corollary~\ref{gaucor_old}.

\begin{corollary}\label{bicor_old}
Let Assumptions~\ref{A1},~\ref{A3} hold, 
let~$(x,v)\in\mathbb{R}^{2d}$ 
and~$f:\mathbb{R}^d\rightarrow\mathbb{R}$ be Lipschitz. Assume~$4LT^2\leq(1-\eta)^2$,~\eqref{carot} and that~$(\hat{\theta}_i)_i,(\hat{\theta}_i')_i$ are sequences of tuples of unbiased~$\Theta$-valued r.v.'s. If~$u=0$ (velocity Verlet integrator case), then for any~$N_0,N\in\mathbb{N}$, it holds that
\begin{align*}
&\frac{1}{\|f\|_{\textrm{Lip}(\hat{d})}}\bigg|\mathbb{E}\bigg[\frac{1}{N}\sum_{i=0}^{N-1}\bigg(f(X^{(N_0+i)}) - \int fd\mu\bigg)\bigg]\bigg|\\
&\quad\leq \frac{c^*(1-c)^{N_0}}{Nc}\Big(\mathbb{E}[\abs{(x,v)}_{\hat{d}}
] + 4\sqrt{3}(d^*)^{\frac{1}{2}} + B\Big)\\
&\qquad+ 2\bigg(\frac{C_U\sqrt{L}h}{p}\bigg)^{\frac{1}{2}}\cdot\frac{\exp(\frac{3}{2}\sqrt{L}T(N_0+N))}{\frac{3}{2}N\sqrt{L}T} + B,
\end{align*}
where~$d^*,c^*,B$ are given by
\begin{equation*}
d^* = \frac{\max(L(R')^2,d)}{m},\quad c^* = \bigg(\frac{g}{2\log(2)\epsilon^*}\bigg)^{\frac{1}{2}},\quad B = 16\sqrt{L}h(d^*)^{\frac{1}{2}}(2c^*)^{2\sqrt{L}T/c},
\end{equation*}
the constant~$R'$ is given by~$R'= 4R(1+\frac{L}{m})$,~$c,\epsilon^*$,~$g$ are given by~\eqref{cothrate},~\eqref{epidef},~\eqref{rhdef} and~\eqref{gdef??} and~$f_0$ is given by~\eqref{f0def}.
If instead~$u=1$ (randomized midpoint case), the same assertion holds but with~$B$ given by
\begin{equation*}
B = 144\sqrt{6}L^{1/4}T^{-1}h^{3/2}(d^*)^{\frac{1}{2}}(2c^*)^{2\sqrt{L}T/c}.
\end{equation*}
In case~$u=0$, if in addition there exists~$L_2>0$ such that~$\abs{\nabla^2 U(x) - \nabla^2 U(x)}\leq L_2\abs{x-x'}$ for all~$x,x'\in\mathbb{R}^d$, then the same assertion 
holds but with the definition of~$B$ replaced by
\begin{equation*}
B = 66(L^{3/2}+L_2)h^2d^*(2c^*)^{2\sqrt{L}T/c}.
\end{equation*}
\end{corollary}

\begin{appendix}

\section{Proofs for concentration}\label{gauapp}

\begin{proof}[Proof of Lemma~\ref{bbcon}]
By Lemma~\ref{basic} with~$c=1/16^2$,~$h^{(i)}$ satisfies for any~$x,y,v,w\in\mathbb{R}^d$,~$i=1,2$ that
\begin{align*}
&\abs{\bar{q}_T(x,L^{\frac{1}{2}}v,\bar{\theta}_i) - \bar{q}_T(y,L^{\frac{1}{2}}w,\bar{\theta}_i)}^2 + L^{-1}\abs{\bar{p}_T(x,L^{\frac{1}{2}}v,\bar{\theta}_i) - \bar{p}_T(y,L^{\frac{1}{2}}w,\bar{\theta}_i)}^2 \\
&\quad\leq 2^{\frac{1}{3}}(\abs{x-y}^2 + \abs{v-w}^2),
\end{align*}
which is that~$h^{(i)}$ is Lipschitz with constant~$2^{1/6}$. 
Moreover, the mapping~$h^{(0)}:\mathbb{R}^{3d}\rightarrow\mathbb{R}^{2d}$ given by~$h^{(0)}(x,v,g) = (x,\eta v + g)$ satisfies for any~$x,y,v,w,g,g'\in\mathbb{R}^d$ that
\begin{align*}
\abs{h^{(0)}(x,v,g)-h^{(0)}(y,w,g')}^2 &= \abs{x-y}^2 + \abs{\eta (v-w) + g-g'}^2\\
&\leq \abs{x-y}^2 + (1+\eta^2)\abs{v-w}^2 + (1+\eta^2)\abs{g-g'}^2,
\end{align*}
so that~$h^{(0)}$ is Lipschitz with constant~$\sqrt{2}$. 
Therefore, for any~$x,v\in\mathbb{R}^d$,~$\hat{h}_{x,v}$ is Lipschitz with constant~$2$ and~$\hat{h}_{x,v}'$ is Lipschitz with constant~$2\sqrt{2}$. The first assertion follows by Rademacher's theorem. 
In case~$\eta=0$, Lemma~\ref{basic} with~$c=1/16^2$ gives for any~$x,y,v,w\in\mathbb{R}^d$ that~$\abs{\bar{q}_T(x,L^{\frac{1}{2}}v,\bar{\theta}_1) - \bar{q}_T(y,L^{\frac{1}{2}}w,\bar{\theta}_1)} \leq \sqrt{2}(\abs{x-y} + \sqrt{L}T\abs{v-w})$, 
so that~$\hat{h}_x^{(0)}$ is Lipschitz with constant~$\sqrt{2L}T$, which concludes again by Rademacher's theorem. 
\end{proof}

\begin{proof}[Proof of Lemma~\ref{abcon}]
Fix~$x,v\in\mathbb{R}^d$. Let~$\hat{h}':\mathbb{R}^d\times\mathbb{R}^d\rightarrow\mathbb{R}^d\times\mathbb{R}^d$ be given by~$\hat{h}'(g,g')= h^{(0)}(h^{(1)}(x,L^{-\frac{1}{2}}\eta v+g),g')$, so that~$\hat{h}_{x,v}=h^{(2)}\circ\hat{h}'$. 
For the assertion about the pushforward of~$\gamma_{0,(1-\eta^2)/L}^{2d}$, by Corollary~4.3 in~\cite{MR4687848}, it suffices to prove that the matrices~$\nabla h^{(2)}$ and~$\nabla \hat{h}'$ are invertible (in the sense of matrices) everywhere. 

For any~$\bar{\theta}_0=(\theta,\theta',r,r_0)\in\Theta\times\Theta\times\{0,1\}\times[0,1]$, let~$A_{\bar{\theta}_0}:\mathbb{R}^d\times\mathbb{R}^d\rightarrow\mathbb{R}^{2d\times 2d}$,~$A_{\bar{\theta}_0}:\mathbb{R}^d\times\mathbb{R}^d\times\mathbb{R}^d\rightarrow\mathbb{R}^{2d\times 2d}$ be given by 
\begin{align*}
A_{\bar{\theta}_0}(\bar{y},\bar{z}) &= \begin{pmatrix}
0 & L^{\frac{1}{2}}hI_d \\
-L^{-\frac{1}{2}}\frac{h}{2}(1+r)\nabla b(\bar{y},\theta) -L^{-\frac{1}{2}}\frac{h}{2}(1-r)\nabla b(\bar{z},\theta') & 0
\end{pmatrix},\\
A_{\bar{\theta}_0}'(\bar{x},\bar{y},\bar{z}) &= \begin{pmatrix}
-\frac{h^2}{2}\nabla b(\bar{y},\theta) & - L^{\frac{1}{2}}rr_0\frac{h^3}{2}\nabla b(\bar{y},\theta)\\
L^{-\frac{1}{2}}(h^3/4)(1-r)\nabla b(\bar{z},\theta')\nabla b(\bar{x},\theta) & - rr_0\frac{h^2}{2}(1+r)\nabla b(\bar{y},\theta)
\end{pmatrix},
\end{align*}
Below, we omit in the notation the dependence on~$\bar{\theta}_0$ and write~$A,A'$ in place of~$A_{\bar{\theta}_0},A_{\bar{\theta}_0}'$. 
Let~$\abs{\cdot}_1$ denote the norm on~$\mathbb{R}^{2d}$ given by~$\abs{(x,v)}_1 = \abs{x} + \abs{v}$ for~$x,v\in\mathbb{R}^d$. 
For any~$\bar{x},\bar{y},\bar{z}$ and~$\bar{\theta}_0$, the operator norm~$\abs{A+A'}_1$ of~$(\mathbb{R}^{2d},\abs{\cdot}_1)\ni \hat{x}\mapsto (A(\bar{y},\bar{z}) + A'(\bar{x},\bar{y},\bar{z}))\hat{x}\in(\mathbb{R}^{2d},\abs{\cdot}_1)$ satisfies by~\eqref{A1a} and~\eqref{carot} that
\begin{equation}\label{aanorm}
\abs{A + A'}_1\leq L^{\frac{1}{2}}h + Lh^2 + L^{\frac{3}{2}}h^3/2 <1.
\end{equation}
By definition~\eqref{hameqs} of the one-step integrator~$(\bar{q}_h,\bar{p}_h)$, 
for any~$\bar{\theta}
\in\{\bar{\theta}_1,\bar{\theta}_2\}$, 
there exist~$y_0,y_1,y_0'\in\mathbb{R}^d$ such that the mapping~$\bar{h}:\mathbb{R}^d\times\mathbb{R}^d\rightarrow\mathbb{R}^d\times\mathbb{R}^d$ given by~$\bar{h}(\bar{x},\bar{v})=(\bar{q}_h(\bar{x},L^{\frac{1}{2}}\bar{v},\bar{\theta}),L^{-\frac{1}{2}}\bar{p}_h(\bar{x},L^{\frac{1}{2}}\bar{v},\bar{\theta}))$ satisfies
\begin{equation*}
\nabla\bar{h}(x,v) = 
\begin{pmatrix}
\nabla\!_x\bar{q}_h & L^{\frac{1}{2}}\nabla\!_v\bar{q}_h\\
L^{-\frac{1}{2}}\nabla\!_x\bar{p}_h & \nabla\!_v\bar{p}_h
\end{pmatrix}(x,L^{\frac{1}{2}}v,\bar{\theta})
= I_{2d} + A(y_1,y_0') + A'(y_0,y_1,y_0')
\end{equation*}
for some~$\bar{\theta}_0=(\theta,\theta',r,r_0)\in\Theta\times\Theta\times\{0,1\}\times[0,1]$.
Therefore by~\eqref{aanorm},~$\nabla \bar{h}(x,v)$ is invertible 
(with inverse~$\sum_{i=0}^{\infty} (-A-A')^i$). 
Similarly and by repeated applications of the chain rule,~$\nabla h^{(2)}(\bar{x},\bar{v})$ is invertible for any~$\bar{x},\bar{v}\in\mathbb{R}^d$. Moreover, it holds by definition for any~$g,g'\in\mathbb{R}^d$ that
\begin{equation}\label{peba}
\nabla\!_g\hat{h}'(g,g')=
\begin{pmatrix}
I_d & 0\\
0 & \eta I_d
\end{pmatrix}
\nabla h^{(1)}(x,L^{-\frac{1}{2}}\eta v + g)
\begin{pmatrix}
0\\ I_d
\end{pmatrix}
\end{equation}
and there exists~$(y_i,y_i')_{i\in[0,K]\cap\mathbb{N}}$ with~$y_i,y_i'\in\mathbb{R}^d$ for all~$i$ such that 
\begin{equation}\label{nabh}
\nabla h^{(1)}(x,L^{-\frac{1}{2}}\eta v + g) = \prod_{i=0}^{K-1}( I_{2d} + A(y_{i+1},y_i') +
A'(y_i,y_{i+1},y_i')).
\end{equation}
for some finite sequence of~$\bar{\theta}_0$. 
We proceed to show that the expression in~\eqref{peba} is a small enough perturbation of~$(L^{\frac{1}{2}}TI_d\ \eta I_d)^{\top}$, 
using that the latter is its precise form if in~\eqref{nabh} we neglect~$A'$ and the lower triangular parts of~$A$. 

For any~$i$, let~$A_i=A(y_{i+1},y_i')$,~$A_i'=A'(y_i,y_{i+1},y_i')$ and, with abuse of notation, let~$\abs{\cdot}_1$ denote the operator norm on~$(\mathbb{R}^{2d},\abs{\cdot}_1)$. Moreover, for any~$M,M'\in\mathbb{R}^{d\times d}$, let~$\abs{(M\ M')^{\top}}_{\infty} = \max(\abs{M},\abs{M'})$, where~$\abs{M}$ is the operator norm of~$M$ on~$(\mathbb{R}^d,\abs{\cdot})$. In particular, for any~$\bar{M}\in\mathbb{R}^{2d\times 2d}$, we have~$\abs{\bar{M}(0\ I_d)^{\top}}_{\infty}\leq \abs{\bar{M}}_1$.
Therefore, it holds by~\eqref{nabh} that
\begin{align}
&\bigg\vert\nabla h^{(1)}(x,L^{-\frac{1}{2}}\eta v+g)\begin{pmatrix} 0\\ I_d\end{pmatrix} - \begin{pmatrix}L^{\frac{1}{2}}TI_d\\I_d\end{pmatrix}\bigg\vert_{\infty}\nonumber\\
&\quad\leq \bigg\vert \prod_{i=0}^{K-1}(I_{2d}+A_i)\begin{pmatrix} 0\\ I_d\end{pmatrix}- \begin{pmatrix}L^{\frac{1}{2}}TI_d\\I_d\end{pmatrix}\bigg\vert_{\infty}\nonumber\\
&\qquad+ 
\bigg\vert \bigg(\prod_{i=0}^{K-1}(I_{2d} + A_i + A_i') - \prod_{i=0}^{K-1}(I_{2d} + A_i)\bigg)\bigg\vert_1.\label{njd}
\end{align}
By definition of~$A,A'$ and~\eqref{A1a},~\eqref{carot}, it holds for any~$i$ that
\begin{equation}
\abs{I_{2d}+A_i}_1 = 1+L^{\frac{1}{2}}h,\qquad\abs{A_i'}_1 \leq Lh^2.
\end{equation}
Therefore, the last term on the right-hand side of~\eqref{njd} satisfies
\begin{align}
\bigg\vert \prod_{i=0}^{K-1}(I_{2d} + A_i + A_i') - \prod_{i=0}^{K-1}(I_{2d} + A_i)\bigg\vert_1 &\leq \prod_{j = 0}^{K-1} \begin{pmatrix} K \\ j\end{pmatrix} (1+L^{\frac{1}{2}}h)^j (Lh^2)^{K-j} \nonumber\\
&\leq (1+L^{\frac{1}{2}}h + Lh^2)^K - (1+L^{\frac{1}{2}}h)^K,\label{prbk}
\end{align}
where the right-hand side may be bounded using~\eqref{carot} 
as
\begin{align}
(1+L^{\frac{1}{2}}h + Lh^2)^K - (1+L^{\frac{1}{2}}h)^K 
&\leq (1+L^{\frac{1}{2}}h)^K((1+Lh^2)^K - 1)\nonumber\\
&\leq e^{L^{\frac{1}{2}}T}(e^{LT^2}-1)\nonumber\\
&\leq 2LT^2.\label{prbk2}
\end{align}
For the first term on the right-hand side of~\eqref{njd}, denoting~$(A_i)_{2,1}$ to be the lower triangular part of~$A_i$, it holds that
\begin{align}
&\prod_{i=0}^{K-1} (I_{2d} + A_i) \begin{pmatrix} 0\\ I_d\end{pmatrix} - \begin{pmatrix} L^{\frac{1}{2}}TI_d \\ I_d\end{pmatrix} \nonumber\\
&\quad=  \prod_{i=0}^{K-1} \bigg(\!\begin{pmatrix} I_d & L^{\frac{1}{2}}hI_d \\ 0 & I_d\end{pmatrix} + (A_i)_{2,1}\bigg) \begin{pmatrix} 0\\ I_d\end{pmatrix} 
- \prod_{i=0}^{K-1}\begin{pmatrix} I_d & L^{\frac{1}{2}}hI_d \\ 0 & I_d\end{pmatrix}\begin{pmatrix} 0 \\ I_d\end{pmatrix}.\label{prbl}
\end{align}
In expanding the first product on the right-hand side of~\eqref{prbl}, there are exactly~$K$ terms in which there appears one and only one~$(A_i)_{2,1}$. By direct calculation, these terms are of the form
\begin{align}
A_i^{(1)}:=&\begin{pmatrix} I_d & L^{\frac{1}{2}}hI_d \\ 0 & I_d\end{pmatrix}^i (A_i)_{2,1} \begin{pmatrix} I_d & L^{\frac{1}{2}}hI_d \\ 0 & I_d\end{pmatrix}^{K-1-i}
\begin{pmatrix} 0 \\ I_d \end{pmatrix} \nonumber \\
&= 
\begin{pmatrix} 
-L^{\frac{1}{2}}(K-1-i)\frac{ih^3}{2}((1+r)\nabla b(y_{i+1},\theta) + (1-r)\nabla b(y_i',\theta'))\\
-(K-1-i)\frac{h^2}{2}((1+r)\nabla b(y_{i+1},\theta) + (1-r)\nabla b(y_i',\theta'))\end{pmatrix}\label{prbh}
\end{align}
for some~$\theta,\theta'\in\Theta$. In particular, 
substituting~\eqref{prbh} into~\eqref{prbl} yields
\begin{align}
&\bigg\vert\prod_{i=0}^{K-1} (I_{2d} + A_i) \begin{pmatrix} 0\\ I_d\end{pmatrix} - \begin{pmatrix} L^{\frac{1}{2}}TI_d \\ I_d\end{pmatrix}\bigg\vert_{\infty} \nonumber\\
&\quad \leq \sum_{i=0}^{K-1}\abs{A_i^{(1)}}_{\infty} + \prod_{j=0}^{K-2}\begin{pmatrix} K\\ j\end{pmatrix} (1+L^{\frac{1}{2}}h)^j(L^{\frac{1}{2}}h)^{K-j}\nonumber\\
&\quad\leq LT^2 + (1+2L^{\frac{1}{2}}h)^K - (1+L^{\frac{1}{2}}h)^K - L^{\frac{1}{2}}T(1+L^{\frac{1}{2}}h)^{K-1},\label{pko}
\end{align}
where the last terms on the right-hand side of~\eqref{pko} may be bounded as
\begin{equation}\label{kpo}
(1+2L^{\frac{1}{2}}h)^K - (1+L^{\frac{1}{2}}h)^K - L^{\frac{1}{2}}T(1+L^{\frac{1}{2}}h)^{K-1} \leq e^{2L^{\frac{1}{2}}T} - (1+L^{\frac{1}{2}}T) - L^{\frac{1}{2}}T.
\end{equation}
The right-hand side of~\eqref{kpo} is a convex function of~$T$, so that substituting back into~\eqref{pko} and using~\eqref{carot} gives
\begin{equation}\label{nfj}
\bigg\vert\prod_{i=0}^{K-1} (I_{2d} + A_i) \begin{pmatrix} 0\\ I_d\end{pmatrix} - \begin{pmatrix} L^{\frac{1}{2}}TI_d \\ I_d\end{pmatrix}\bigg\vert_{\infty} \leq \frac{L^{\frac{1}{2}}T}{16} + \bigg(e^{1/8}-\frac{9}{8}\bigg)L^{\frac{1}{2}}T \leq \frac{L^{\frac{1}{2}}T}{10}.
\end{equation}
Gathering~\eqref{njd},~\eqref{prbk},~\eqref{prbk2} and~\eqref{nfj}, it holds by~\eqref{carot} that 
\begin{equation}\label{pebl}
\bigg\vert\nabla h^{(1)}(x,L^{-\frac{1}{2}}\eta v+g)\begin{pmatrix} 0\\ I_d\end{pmatrix} - \begin{pmatrix}L^{\frac{1}{2}}TI_d\\I_d\end{pmatrix}\bigg\vert_{\infty}\leq \frac{L^{\frac{1}{2}}T}{4}.
\end{equation}
By~\eqref{peba}, inequality~\eqref{pebl} implies for any~$g,g'\in\mathbb{R}^d$ that
\begin{equation*}
\nabla\!_g\hat{h}'(g,g') = \begin{pmatrix} L^{\frac{1}{2}}TI_d + M\\ \eta I_d + M' \end{pmatrix}
\end{equation*}
for some~$M,M'\in\mathbb{R}^{d\times d}$ satisfying~$\abs{M}\leq L^{\frac{1}{2}}T/4$. In particular,~$L^{\frac{1}{2}}TI_d + M$ is invertible. On the other hand, it holds by definition that~$\nabla\!_{g'}\hat{h}'(g,g') = 
(0 \ I_d)^{\top}$. Therefore~$\nabla \hat{h}'$ is invertible everywhere, which concludes the proof of the assertion about the pushforward of~$\gamma_{0,(1-\eta^2)/L}^{2d}$. In the same way, the assertion about the pushforward of~$\gamma_{0,1/L}^d$ follows easily from~\eqref{pebl}.
The remaining assertion 
about~$\gamma_{0,(1-\eta^2)/L}^{3d}$ 
follows easily from Theorem~4.4 in~\cite{MR4687848} and the fact that~$\nabla h^{(3)}$ is invertible. 
\end{proof}

\begin{proof}[Proof of Lemma~\ref{conlem}]
In the following, we use the shorthand that an absolutely continuous measure~$\mu$ on~$\mathbb{R}^{2d}$ or~$\mathbb{R}^d$ satisfies~$\textrm{LSI}(\bar{C})$ for a constant~$\bar{C}$ if the log-Sobolev inequality
\begin{equation}\label{thecon0}
\textrm{Ent}_{\mu}(\hat{f}^2):=\int \hat{f}^2\log \hat{f}^2d\mu - \int \hat{f}^2d\mu \log \bigg(\int \hat{f}^2d\mu\bigg) \leq 2\bar{C} \int \abs{\nabla \hat{f}}^2d\mu
\end{equation}
holds for all almost everywhere differentiable~$\hat{f}$. 
It is classical~\cite[Theorem~5.2]{MR1849347} that 
the Gaussian measures~$\gamma_{0,(1-\eta^2)/L}^{2d}$,~$\gamma_{0,1/L}^d$ satisfy~$\textrm{LSI}((1-\eta^2)/L)$ and~$\textrm{LSI}(1/L)$ respectively. 
Therefore, for any almost everywhere differentiable function~$\bar{f}:\mathbb{R}^{2d}\rightarrow\mathbb{R}$ and~$x,v\in\mathbb{R}^d$, 
inequality~\eqref{thecon0} with~$\hat{f}=\bar{f}\circ \hat{h}_{x,v}$, 
inequality~\eqref{ddev} with~$z_1=\nabla (\hat{h}_{x,v}\cdot z_2)$,~$z_2=(\nabla\bar{f})\circ \hat{h}_{x,v}$ and Lemma~\ref{abcon} imply 
\begin{align}
\textrm{Ent}_{\gamma_{0,(1-\eta^2)/L}^{2d}}\!\!(\hat{f}^2) &\leq 
\frac{2(1-\eta^2)}{L} \!\int \sum_{i= 1}^d\bigg(\sum_{j=1}^d\partial_i(\hat{h}_{x,v})_j\cdot((\partial_j\bar{f})\circ \hat{h}_{x,v})\bigg)^2\!\!d\gamma_{0,(1-\eta^2)/L}^{2d}\nonumber\\
&\leq \frac{8(1-\eta^2)}{L}\int \abs{(\nabla\bar{f})\circ \hat{h}_{x,v}}^2 d\gamma_{0,(1-\eta^2)/L}^{2d}.\label{thecon2}
\end{align}
In other words, for any~$x,v\in\mathbb{R}^d$, the measure~$\nu_{x,v}$, which is absolutely continuous by Lemma~\ref{abcon}, 
satisfies~$\textrm{LSI}(4(1-\eta^2)/L)$. 
Similarly, the measure~$\bar{\nu}_{x,v}$ satisfies~$\textrm{LSI}(8(1-\eta^2)/L)$. 
In the case~$\eta=0$, using instead~\eqref{ddev2}, for any~$x\in\mathbb{R}^d$, the measure~$\nu_x$ satisfies~$\textrm{LSI}(2T^2)$. 
Let~$\mathcal{W}_{2,e'}$ denote the~$L^2$ Wasserstein distance w.r.t. the standard Euclidean distance. 
By Theorem~6.6 in~\cite{MR1849347} and~\eqref{thecon2}, for any~$x,v\in\mathbb{R}^d$, for any probability measure~$\mu$ absolutely continuous with respect to~$\nu_{x,v}$, it holds that
\begin{equation}\label{t2c}
(\mathcal{W}_{2,e'}(\nu_{x,v},\mu))^2
\leq 
\frac{4(1-\eta^2)}{L}\textrm{Ent}(\mu|\nu_{x,v}).
\end{equation}
If~$\eta=0$, then the same inequality~\eqref{t2c} holds with~$\nu_{x,v}$ and~$4(1-\eta^2)$ replaced by~$\nu_x$ and~$2LT^2$ respectively.
By the assumption~$4LT^2\leq (1-\eta)^2$, inequality~\eqref{t2c} implies
\begin{align}
\mathcal{W}_1(\nu_{x,v}^*,\mu^*) &\leq \sqrt{2}(1.09/4+1)\mathcal{W}_{2,e'}(\nu_{x,v},\mu)\nonumber\\
&\leq 2(1.09/4+1)\sqrt{2((1-\eta^2)/L)\textrm{Ent}(\mu|\nu_{x,v})}\nonumber\\
&= 2(1.09/4+1)\sqrt{2((1-\eta^2)/L)\textrm{Ent}(\mu^*|\nu_{x,v}^*)}.\label{t1c}
\end{align}
Inequality~\eqref{t1c} implies~$\nu_{x,v}^*\in T_1(4(1.09/4+1)^2(1-\eta^2)/L)$. Similarly,~$\bar{\nu}_{x,v}^*\in T_1(8(1.09/4+1)^2(1-\eta^2)/L)$.
In case~$\eta=0$, we have instead~$\mathcal{W}_1(\nu_x,\mu)\leq (1.09/4+1)\mathcal{W}_{2,e'}(\nu_x,\mu)$, which leads to~$\nu_x\in T_1((1.09/4+1)^2T^2)$. 
\end{proof}

\begin{proof}[Proof of Corollary~\ref{gaucor_old}]
In this proof, we use~$\floor{r}=\inf\{k:\mathbb{Z}:k\geq r\}$ as its usual meaning. 
By Kantorovich-Rubenstein duality and Corollary~\ref{maincore}, it holds for any~$i\in\mathbb{N}$ and~$x,v,y,w\in\mathbb{R}^d$ that
\begin{align}
&\abs{\mathbb{E}[f(X^{(i)}) - f(Y^{(i)})]}\nonumber\\
&\quad\leq \|f\|_{\textrm{Lip}(\hat{d})} \mathcal{W}_1(\pi_{i}(\delta_{(x,v)}),\pi_{i}(\delta_{(y,w)})\nonumber\\
&\quad= \bigg(\frac{g}{2\log(2)\epsilon^*}\bigg)^{\frac{1}{2}}(1-c)^{i}\|f\|_{\textrm{Lip}(\hat{d})} (\abs{x-y+\gamma^{-1}(v-w)} + 1.09\alpha\abs{x-y}),\label{fili}
\end{align}
where~\eqref{fili} holds independently of~$(\bar{\theta}_j)_{j\in\mathbb{N}}$. 
Inequality~\eqref{fili} shows that~$(\mathbb{R}^{2d},\hat{d}) \ni(x,v)\mapsto\mathbb{E}[f(X^{(i)})]$ (and~$(\mathbb{R}^d,\hat{d}) \ni x\mapsto\mathbb{E}[f(X^{(i)})]$ when~$\eta=0$) is Lipschitz with constant~$(g/(2\log(2)\epsilon^*))^{1/2}(1-c)^{i}\|f\|_{\textrm{Lip}(\hat{d})}$. 
In the following, we focus mostly on the case~$\eta\neq 0$, but the case~$\eta=0$ follows in a similar way, if not more easily. 
Let~$\lambda>0$ and let~$\mathcal{F}_i$ denote the sigma-algebra generated by~$(X^{(i)},V^{(i)})$. 
Let~$N^*=N_0+2\floor{N/2}$. 
Assume for induction that for some~$\bar{m}\in\mathbb{N}\cap[0,\floor{N/2}-2]$ it holds a.s. that
\begin{align}
\mathbb{E}\Big[e^{\lambda\sum_{k=0}^{\floor{N/2}-1}f(X^{(N^* - 2k)})}\Big] 
&\leq 
e^{M_{\bar{m}}\|f\|_{\textrm{Lip}(\hat{d})}^2} \mathbb{E}\Big[e^{\lambda\sum_{k=\bar{m}}^{\floor{N/2}-1}f(X^{(N^*-2k)})} \nonumber\\
&\quad\cdot e^{\lambda\sum_{k=0}^{\bar{m}-1}\mathbb{E}[f(X^{(N^*-2k)})|\mathcal{F}_{N^*-2\bar{m}}]}\Big],\label{hwog}
\end{align}
where
\begin{equation}\label{Mmdef}
M_{\bar{m}} = \frac{2\bar{m}\lambda^2T^2C^*g}{\log(2)(1-\eta)\epsilon^*c^2},
\qquad C^* = \begin{cases} 
(\frac{1.09}{4}+1)^2(1+\eta)\frac{(1-\eta)^2}{2LT^2} & \textrm{if }\eta\neq 0\\
\frac{1}{8}(\frac{1.09}{4}+1)^2 & \textrm{if }\eta=0.
\end{cases}
\end{equation}
The base case~$\bar{m}=0$ is clear. 
In the right-hand side of~\eqref{hwog}, it holds by the tower property a.s. that
\begin{align}
&\mathbb{E}\Big[e^{\lambda\sum_{k=\bar{m}}^{\floor{N/2}-1} f(X^{(N^*-2k)})} e^{\lambda\sum_{k=0}^{\bar{m}-1}\mathbb{E}[f(X^{(N^*-2k)})|\mathcal{F}_{N^*-2\bar{m}}]}\Big] \nonumber\\
&\quad= \mathbb{E}\Big[e^{\lambda\sum_{k=\bar{m}+1}^{\floor{N/2}-1} f(X^{(N^*-2k)})} \mathbb{E}\Big[e^{\lambda f(X^{(N^*-2\bar{m})})} \nonumber\\
&\qquad\cdot e^{\lambda\sum_{k=0}^{\bar{m}-1}\mathbb{E}[f(X^{(N^*-2k)})|\mathcal{F}_{N^*-2\bar{m}}]}\Big|\mathcal{F}_{N^*-2\bar{m}-2}\Big]\Big].\label{jdb}
\end{align}
By Doob-Dynkin Lemma~1.14 in~\cite{MR4226142} and the Lipschitz estimate~\eqref{fili}, the sum~$f(X^{(N^*-2\bar{m})}) + \sum_{k=0}^{\bar{m}-1}\mathbb{E}[f(X^{(N^*-2k)})|\mathcal{F}_{N^*-2\bar{m}}]$ is Lipschitz with constant
\begin{equation}\label{gei}
\|f\|_{\textrm{Lip}(\hat{d})}\bigg(1+\bigg(\frac{g}{2\log(2)\epsilon^*}\bigg)^{\frac{1}{2}}\sum_{k=1}^{\bar{m}} (1-c)^{2k}\bigg)  \leq \frac{\|f\|_{\textrm{Lip}(\hat{d})}}{c} \bigg(\frac{g}{2\log(2)\epsilon^*}\bigg)^{\frac{1}{2}} =:\bar{M}
\end{equation}
as a function of~$(X^{(N^*-2\bar{m})},V^{(N^*-2\bar{m})})$. Therefore
by Theorem~1.1 in~\cite{MR2078555} and Lemma~\ref{conlem}\ref{conlem1}, 
the (first) conditional expectation on the right-hand side of~\eqref{jdb} satisfies a.s. the inequality
\begin{align}
&\mathbb{E}\Big[e^{\lambda f(X^{(N^*-2\bar{m})})}  e^{\lambda\sum_{k=0}^{\bar{m}-1}\mathbb{E}[f(X^{(N^*-2k)})|\mathcal{F}_{N^*-2\bar{m}}]}\Big| \mathcal{F}_{N^*-2\bar{m}-2}\Big]\nonumber\\
&\quad\leq e^{\lambda\sum_{k=0}^{\bar{m}}\mathbb{E}[f(X^{(N^*-2k)})|\mathcal{F}_{N^*-2\bar{m}-2}] + 4\lambda^2T^2C^* \bar{M}^2}, \label{jvh} 
\end{align}
where~$\bar{M}$ is defined in~\eqref{gei} and satisfies
\begin{equation}\label{qhf}
4\lambda^2T^2 C^*\bar{M}^2 \leq \frac{2\lambda^2T^2 C^*g}{\log(2)(1-\eta)\epsilon^*c^2} \|f\|_{\textrm{Lip}(\hat{d})}^2.
\end{equation}
Substituting~\eqref{qhf} back into~\eqref{jvh} then into~\eqref{jdb} and subsequently~\eqref{hwog} completes the inductive step. In particular, inequality~\eqref{hwog} holds with~$\bar{m}=\floor{N/2}-1$, which is that 
\begin{equation}\label{gbi}
\mathbb{E}\Big[e^{\lambda\sum_{k=0}^{\floor{N/2}-1}f(X^{(N^*-2k)})}\Big] \leq 
e^{M_{\floor{N/2}-1}\|f\|_{\textrm{Lip}(\hat{d})}^2} \mathbb{E}\Big[ e^{\lambda\sum_{k=0}^{\floor{N/2}-1}\mathbb{E}[f(X^{(N^*-2k)})|\mathcal{F}_{N_0+2}]}\Big].
\end{equation}
By the same arguments, denoting~$\bar{N}=\sup\{k\in\mathbb{N}\setminus(2\mathbb{N}):k\leq N\}$, 
it holds that
\begin{align}\label{gbi2}
&\mathbb{E}\Big[e^{\lambda\sum_{k=0}^{(\bar{N}-1)/2}f(X^{(N_0+\bar{N}-2k)})}\Big] \nonumber\\
&\quad\leq 
e^{M_{(\bar{N}-1)/2}\|f\|_{\textrm{Lip}(\hat{d})}^2} \mathbb{E}\Big[ e^{\lambda\sum_{k=0}^{(\bar{N}-1)/2}\mathbb{E}[f(X^{(N_0+\bar{N}+2k)})|\mathcal{F}_{N_0+1}]}\Big].
\end{align}
Note that the exponents on the left-hand sides of~\eqref{gbi} and~\eqref{gbi2} are concerned with respectively the even and odd SGgHMC iterations after the initial~$N_0$ iterations.
For the right-hand side of~\eqref{gbi}, assume for induction that for some~$\bar{m}\in\mathbb{N}\cap[0,\floor{N_0/2})$ it holds that
\begin{align}\label{gvq}
&\mathbb{E}\Big[ e^{\lambda\sum_{k=0}^{\floor{N/2}-1}\mathbb{E}[f(X^{(N^*-2k)})|\mathcal{F}_{N_0+2}]}\Big] \nonumber\\
&\quad\leq e^{\bar{M}_{\bar{m}}\|f\|_{\textrm{Lip}(\hat{d})}^2}\mathbb{E}\Big[ e^{\lambda\sum_{k=0}^{\floor{N/2}-1}\mathbb{E}[f(X^{(N^*-2k)})|\mathcal{F}_{N_0+2-2\bar{m}}]}\Big],
\end{align}
where~$\bar{M}_{\bar{m}}$ is given together with~$\hat{M}$ by
\begin{equation}\label{M2def}
\bar{M}_{\bar{m}} = \lambda^2\hat{M}\sum_{\bar{k}=0}^{2\bar{m}-2}(1-c)^{2\bar{k}+2},\qquad \hat{M} = \frac{ 2T^2C^* g}{\log(2)(1-\eta)\epsilon^*c^2}.
\end{equation}
Again, the base case~$\bar{m}=0$ is clear. Since the sum in the exponent within the expectation on the right-hand side of~\eqref{gvq} is Lipschitz with constant
\begin{equation*}
\bigg(\frac{g}{2\log(2)\epsilon^*}\bigg)^{\frac{1}{2}}\|f\|_{\textrm{Lip}(\hat{d})}\!\!\!\sum_{k=0}^{\floor{N/2}-1} (1-c)^{2k+2\bar{m}+2} \leq (1-c)^{2\bar{m}+2}\bigg(\frac{g\|f\|_{\textrm{Lip}(\hat{d})}^2}{2\log(2)\epsilon^*c^2}\bigg)^{\frac{1}{2}}
\end{equation*}
as a function of~$(X^{(N_0-2\bar{m})},V^{(N_0-2\bar{m})})$, by Theorem~1.1 in~\cite{MR2078555} and 
Lemma~\ref{conlem}\ref{conlem1}, 
it holds a.s. that
\begin{align}\label{gbo}
&\mathbb{E}\Big[ e^{\lambda\sum_{k=0}^{\floor{N/2}-1}\mathbb{E}[f(X^{(N^*-2k)})|\mathcal{F}_{N_0+2-2\bar{m}}]} \Big| \mathcal{F}_{N_0-2\bar{m}}\Big]\nonumber\\
&\quad\leq e^{\lambda\sum_{k=0}^{\floor{N/2}-1}\mathbb{E}[f(X^{(N^*-2k)})|\mathcal{F}_{N_0-2\bar{m}}] + \lambda^2(1-c)^{4\bar{m}+4}\hat{M}\|f\|_{\textrm{Lip}(\hat{d})}^2}.
\end{align}
Substituting inequality~\eqref{gbo} back into~\eqref{gvq} completes the induction argument to obtain~\eqref{gvq} with~$\bar{m}=\floor{N_0/2}$. Therefore, by either Lemma~\ref{conlem}\ref{conlem1} or Lemma~\ref{conlem}\ref{conlem2}, repeating the inductive argument one more time yields
\begin{align}
&\mathbb{E}\Big[ e^{\lambda\sum_{k=0}^{\floor{N/2}-1}\mathbb{E}[f(X^{(N^*-2k)})|\mathcal{F}_{N_0+2}]}\Big] \nonumber\\
&\quad\leq e^{\bar{M}_{\floor{N_0/2}+1}\|f\|_{\textrm{Lip}(\hat{d})}^2}\mathbb{E}\Big[ e^{\lambda\sum_{k=0}^{\floor{N/2}-1}\mathbb{E}[f(X^{(N^*-2k)})]}\Big].\label{halfc1}
\end{align}
In the same way, starting from the right-hand side of~\eqref{gbi2} instead of that of~\eqref{gbi}, it holds that
\begin{align}
&\mathbb{E}\Big[ e^{\lambda\sum_{k=1}^{(\bar{N}-1)/2}\mathbb{E}[f(X^{(N_0+\bar{N}+2k)})|\mathcal{F}_{N_0+1}]}\Big]\nonumber\\
&\quad\leq e^{\bar{M}_{\floor{(N_0-1)/2}+1}\|f\|_{\textrm{Lip}(\hat{d})}^2}\mathbb{E}\Big[e^{\lambda\sum_{k=0}^{(\bar{N}-1)/2}\mathbb{E}[f(X^{(N_0+\bar{N}+2k)})]}\Big].\label{halfc2}
\end{align}
Inequalities~\eqref{halfc1},~\eqref{halfc2} may be substituted into the right-hand sides of~\eqref{gbi},~\eqref{gbi2} respectively, so that H\"{o}lder's inequality implies 
\begin{align}
&\Big(\mathbb{E}\Big[e^{\frac{\lambda}{2}\sum_{k=1}^N(f(X^{(N_0+k)}) - \mathbb{E}[f(X^{(N_0+k)})])}\Big]\Big)^2 \nonumber\\
&\quad\leq 
\mathbb{E}\Big[e^{\lambda\sum_{k=0}^{\floor{N/2}-1}(f(X^{(N^*-2k)}) - \mathbb{E}[f(X^{(N^*-2k)})])}\Big] \nonumber\\
&\qquad\cdot
\mathbb{E}\Big[e^{\lambda\sum_{k=0}^{(\bar{N}-1)/2}(f(X^{(N_0+\bar{N}-2k)}) - \mathbb{E}[f(X^{(N_0+\bar{N}-2k)})])}\Big] \nonumber\\
&\quad\leq e^{ (M_{N-2} + 2\bar{M}_{\floor{N_0/2}+1})\|f\|_{\textrm{Lip}(\hat{d})}^2}.\label{dig}
\end{align}
Using (again) the definitions~\eqref{Mmdef},~\eqref{M2def} of~$M_{N-2}$,~$\bar{M}_{\floor{N_0/2}+1}$,~$\hat{M}$, it holds that
\begin{equation}\label{dig2}
M_{N-2} + 2\bar{M}_{N_0/2+1} \leq \lambda^2\hat{M}(N + c^{-1}).
\end{equation}
Therefore, in particular for any~$t>0$ and~$\lambda = \frac{Nt}{2\hat{M}\|f\|_{\textrm{Lip}(\hat{d})}^2(N+c^{-1})}$, 
inequalities~\eqref{dig},~\eqref{dig2} imply
\begin{equation}\label{oag}
e^{-\frac{\lambda}{2} N t}\mathbb{E}\Big[e^{\frac{\lambda}{2}\sum_{k=1}^{N}(f(X^{(N_0+k)}) - \mathbb{E}[f(X^{(N_0+k)})])}\Big] \leq e^{-t^2 \hat{C}},
\end{equation}
where~$\hat{C} = \frac{N^2}{8\hat{M}\|f\|_{\textrm{Lip}(\hat{d})}^2(N+c^{-1})}$. 
By Markov's inequality, it holds for any~$t\in\mathbb{R}$ that
\begin{align*}
&\mathbb{P}\bigg(\frac{1}{N}\sum_{k=N_0+1}^{N_0+N}(f(X^{(k)})-\mathbb{E}[f(X^{(k)})])>t\bigg) \\
&\quad=\mathbb{P}\Big(e^{\frac{\lambda}{2} \sum_{k=N_0+1}^{N_0+N}(f(X^{(k)}) - \mathbb{E}[f(X^{(k)})])}\geq e^{\frac{\lambda}{2} Nt}\Big)\\
&\quad\leq e^{-\frac{\lambda}{2} Nt} \mathbb{E}\Big[e^{\frac{\lambda}{2} \sum_{k=N_0+1}^{N_0+N}(f(X^{(k)}) - \mathbb{E}[f(X^{(k)})])}\Big],
\end{align*}
which concludes by~\eqref{oag} and substituting the definitions~\eqref{M2def},~\eqref{Mmdef} for~$\hat{M},C^*$.
\end{proof}

\section{Wasserstein-to-TV bounds}\label{sec:WtoTV}

Theorem~\ref{prop:TV/W1}, concerned with the Verlet and randomized middle-point integrators~\eqref{hameqs}, is similar to \cite[Proposition~3]{monmarche2023entropic} where the former are replaced by the true Hamiltonian flow, and the proofs are the same. In order to highlight the similarities, we introduce now a general framework which encompasses these various cases and many other similar splitting schemes. Then we state our main result in these general settings, Theorem~\ref{thm:generalTVW1}. Theorem~\ref{prop:TV/W1} is then a simple corollary. The proofs rely on technical computations which are exactly the same as in previous studies \cite{MR4497240,camrud2023second,monmarche2023entropic}  and thus are omitted in the following.

\subsection{General settings}

In all this section, we consider a space $\Theta$ and, for $T\in(0,1]$ and $\theta \in \Theta$, a function $\Phi_T^\theta \in\mathcal C^1 (\R^{2d},\R^{2d})$, which plays the role of an Hamiltonian propagator with integration time $T$ and parameter $\theta$. To fix ideas, for instance, in \cite{monmarche2023entropic}, $\Theta=\{0\}$ and $\Phi_T^0$ is the Hamiltonian flow while, in the present work, $\Phi_T^{\bar \theta}(x,v) = (\bar q_T,\bar p_T) $ as given by \eqref{hameqs}. The precise assumptions on $\Phi_T^\theta$ are given below. The parameter $\theta $ may account for random seeds (as in \eqref{hameqs}, either due to a stochastic approximation of the force or to a random integrator) but also for time-inhomogeneous dynamics, which are thus covered in the following. For now, a sequence  $(\theta_k)_{k\in\N} \in \Theta^{\N}$  is fixed.

Given $T\in(0,1]$, $\eta\in(0,1]$, and an initial distribution $\nu_0\in\mathcal P(\R^{2d})$, we construct a Markov chain as follows. Let $z_0 = (x_0,v_0)\sim \nu_0$. Assume that $z_k=(x_k,v_k)$ has been defined for some $k\in\N$ and, independently, let $(G_k,G_k')$ be two independent standard $d$-dimensional Gaussian variables. The next transition of the chain is given by
\begin{subequations}\label{eq:DVD}
\begin{align}
\tilde v_{k+1/3} & = \eta v_k + \sqrt{1-\eta^2} G_k \label{eq:DVD1} \\
(x_{k+1},\tilde v_{k+2/3}) &=  \Phi_T^{\theta_k}(x_k,\tilde v_{k+1/3})\label{eq:DVD2}\\
 v_{k+1} & =  \eta \tilde v_{k+2/3} + \sqrt{1-\eta^2} G_k'\,.\label{eq:DVD3}
 \end{align}
\end{subequations}
We write $\hat \pi_n(\nu_0)$ the law of $z_n$ for $n\in\N$. By contrast to the SGgHMC chain considered e.g. in Theorems~\ref{intma} and \ref{prop:TV/W1} (with law $\pi_n(\nu_0)$ after $n$ steps), two velocity refreshments are performed in~\eqref{eq:DVD}, respectively before and after the Hamiltonian step~\eqref{eq:DVD2}. As discussed in~\cite{monmarche2023entropic}, this is necessary to ensure that both velocity and position are regularized even after~$n=1$ transition. When~$n>1$, applying successively \eqref{eq:DVD3} and \eqref{eq:DVD1} amounts to performing a single velocity refreshment, with parameter~$\eta^2$ instead of~$\eta$. It is then not difficult,  up to replacing $\eta$ by $\sqrt{\eta}$, to pass from results written in terms of~$\hat\pi_n(\nu_0)$ and~$\pi_n(\nu_0)$ (using  that total variation and Wasserstein $2$ distances are decreased by velocity refreshments) when $n>1$ (see Section~\ref{subsec:appliTVW1}).

The proof of \cite[Proposition~3]{monmarche2023entropic} works as soon as $\Phi_T^{\theta}$ satisfy estimates similar to the Hamiltonian flow, such as stated in \cite[Lemmas 1 and 2]{monmarche2023entropic}. They are essentially  quantitative ways to state that~$\Phi_T^\theta(x,v) = (x + T v, v-T F_{\theta}(x) ) + \underset{T\rightarrow 0}{\mathcal O}(T^2)$ for some Lipschitz continuous $F_{\theta}$, uniformly in $\theta$. We gather them in the next assumption, where we decompose the position and velocity after the Hamiltonian step as $\Phi_T^\theta=(\Phi_T^{\theta,1},\Phi_T^{\theta,2})$.

\begin{assumption}\label{assum:generalTVW1}
There exists $C,T_0>0$ such that for all $T\in(0,T_0]$ and $\theta\in\Theta$, the following holds.
\begin{enumerate}
    \item For any $z,z'\in\R^d$, with $z=(x,v)$, $z'=(x',v')$,  denote $(q,p)=\Phi_T^\theta(z)$ and $(q',p')=\Phi_T^\theta(z')$. Then 
\begin{eqnarray}
|\na_x q - I|&   \leq &    CT^2    \label{eq:lem1_na_xPhi1}     \\
|\na_v q - T I| &  \leq &     CT^3  \label{eq:lem1_na_vPhi1}\\
|\na_x p| &   \leq &   T + CT^3   \label{eq:lem1_na_xPhi2}\\
|\na_v p - I| &   \leq &     C T^2   \label{eq:lem1_na_vPhi2}
\end{eqnarray} 
and   
\begin{eqnarray}
|q - q'| & \leq &  \po 1 + CT^2 \  \pf |x-x'| + \po T + C T^3 \pf |v-v'|   \label{eq:lem1_Phi1z-z'} \\
|p-p'| 
& \leq  &  \po T + C T^3 \pf |x-x'| + \po 1 + C T^2  \pf |v-v'|\,.   \label{eq:lem1_Phi2z-z'}
\end{eqnarray}
Moreover,
\begin{eqnarray}
  \|\na_x q - \na_xq'\|_F  &\leq&   C T^2   |x-x'| + C T^3  L_H |v-v'| \label{eq:lem1_Frob_naxPhi1}\\
  \|\na_v q - \na_v q'\|_F  &\leq& CT^3  |x-x'| + CT^4 |v-v'| \label{eq:lem1_Frob_navPhi1} \\
  \|\na_x p - \na_x p'\|_F  &\leq&  CT  |x-x'| + CT^2  |v-v'|
   \label{eq:lem1_Frob_naxPhi2}  \\
  \|\na_v p- \na_vp'\|_F  &\leq& CT^2 |x-x'| + CT^3  |v-v'|  \label{eq:lem1_Frob_navPhi2} \,.
\end{eqnarray}
\item  There exists a function $K\in\mathcal C^1 (\R^{2d},\R^d)$ such that, for all $x,x'\in\R^{d}$, $v=K(x,x')$ is the unique solution of $\Phi_T^{\theta,1}(x,K(x,x'))=x'$.  For $u_0,u_1\in\R^d$, consider the function $K_{u_0,u_1}\in\mathcal C^1(\R^{2d},\R^d)$ given by $K_{u_0,u_1}(x,v) = K(x+u_0,\Phi_T^{\theta,1}(x,v)+u_1)$, i.e. such that $v'=K_{u_0,u_1}(x,v)$  is the unique solution of  $\Phi_T^{\theta,1}(x+u_0,v')=\Phi_T^{\theta,1}(x,v)+u_1$. Then, for all  $x,v,u_0,u_1\in\R^d$, 
\begin{eqnarray}
|u_1-u_0+tv -t K_{u_0,u_1}(x,v)| &\leq&   CT^2\po  |u_0| +   |u_1-u_0|\pf   \label{lem:Hxx'_eq1}\\
 \| \na_v K_{u_0,u_1}(x,v) - I_d \|_F & \leq & C T^2 \po  |u_0| + T |u_1-u_0|\pf   \label{lem:Hxx'_eq2}  \\
  | \na_v K_{u_0,u_1}(x,v) - I_d | & \leq & C T^2  \label{lem:Hxx'_eq3}  \\
 \| \na_x K_{u_0,u_1}(x,v) \|_F & \leq &  CT \po  |u_0| + T |u_1-u_0|\pf \label{lem:Hxx'_eq4}\\ 
  | \na_x K_{u_0,u_1}(x,v) | &\leq & CT   \,. \label{lem:Hxx'_eq5}
\end{eqnarray}
\end{enumerate}
\end{assumption}

\begin{remark}
The existence of $K$ in the second item, for $T$ small enough, is in fact implied by \eqref{eq:lem1_na_vPhi1}. Indeed, for a fixed $q\in\R^d$, consider the application $g:p \mapsto \Phi_T^{\theta,1}(q,p)$. Using \eqref{eq:lem1_na_vPhi1}, for $p,p'\in\R^d$,
\begin{equation}
    \label{eq:gp}
    |g(p)-g(p') - T (p-p')| \leq C T^3|p-p'|\,,
\end{equation}
and in particular $g$ is injective for $T$ small enough since
\begin{equation*}
|g(p)-g(p')| \geqslant \po T - C T^3\pf |p-p'|.
\end{equation*}
Moreover, \eqref{eq:lem1_na_vPhi1} implies that $\na g(p)$ is non-singular for all $p$ when $CT^2<1$. It follows by the global inversion theorem that $g$ is a $\mathcal C^1$-diffeomorphism from $\R^d$ to its image. To see that $g$ is surjective, take $q'\in\R^d$ and, given any $p_0\in\R^d$, define by induction~$p_{n+1} = p_n + \frac1T \po q'-g(p_n)\pf$. 
Thanks to \eqref{eq:gp}, for all $n\in\N$, we have~$|q' - g(p_{n+1})|  \leq CT^3 |p_{n+1}-p_n| =  CT^2 |q' - g(p_{n})|$. 
Assuming $CT^2<1$, we get that $g(p_n)$ converges geometrically fast to $q'$, and since $|p_{n+1}-p_n| = |q'-g(p_n)|/T$, we get that $(p_n)_{n\in\N}$ is a Cauchy sequence, hence it converges to some $p'$ with $g(p')=q'$.
\end{remark}

\begin{theorem}\label{thm:generalTVW1}
Under Assumption~\ref{assum:generalTVW1}, there exists $C',T_0'>0$, which depend only on $C,T_0$, such that the following holds. For any $z,z'\in\R^{2d}$,  $T\in(0,T_0']$, $(\theta_k)_{k\in\N}\in\Theta^{\N}$,
\[\mathrm{Ent} \po \hat\pi_1(\delta_z)|\hat \pi_1(\delta_{z'})\pf \leq   \frac{C'}{(1-\eta)T^2}|z-z'|^2 \,, \]
and, for any $n\geqslant 1$, assuming furthermore that $\eta>0$ and $s:=nT\leq 1$,
\[\mathrm{Ent} \po \hat\pi_n(\delta_z)|\hat \pi_n(\delta_{z'})\pf \leq   C' \frac{s}{\gamma^*} \po \frac{1}{\eta s^2} + \frac{\gamma^*}{\eta s} + 1\pf^2 |z-z'|^2 \,, \]
where $\gamma^*=(1-\eta)/T$.
\end{theorem}

\begin{proof}
Throughout the proof, $C'$ denotes a constant which depends only on $C$ and may change from line to line.

First, let us simply consider the case $n=1$ (i.e. a single transition). To alleviate notations, simply write $\Phi = \Phi_T^{\theta_0}$.  For $z=(x,v) \in\R^d$, $\bG=(G,G')\in\R^{2d}$, let
\begin{equation}\label{eq:PsizG}
\Psi_{z}(\bG) = \begin{pmatrix}
1 &\  0 \\ 0 & \eta
\end{pmatrix} \Phi\po x , \eta v + \sqrt{1-\eta^2} G\pf + \sqrt{1-\eta^2} \begin{pmatrix}
0\\G'
\end{pmatrix}
\end{equation}
which is the state of the chain  starting from $z$, after one transition \eqref{eq:DVD} if~$G$ and~$G'$ are the variables used in the  randomization  steps. In particular, if~$ \bG \sim \mathcal N(0,I_{2d})$ then the law of $\Psi_z(\bG)$ is $\hat \pi_1(\delta_z) $.

Let us check that $\Psi_z$ is a diffeomorphism.  The equation $(X,Y) = \Psi_{z}(G,G')$ leads, considering first the velocities in this equality and recalling the notation~$\Phi=(\Phi^1,\Phi^2)$, to 
\[G' = \frac{1}{\sqrt{1-\eta^2}} \po Y -  \eta \Phi^{2}\po x , \eta v + \sqrt{1-\eta^2} G\pf\pf \]
and, considering then the positions, to~$X =   \Phi^1(x,\eta v + \sqrt{1-\eta^2} G)$, in other words,~$G =  (1/\sqrt{1-\eta^2}) \po K(x,X)- \eta v\pf$, 
where $K$ is the function defined by the fact $\Phi^1(x,K(x,X)) =  X$, as introduced in Assumption~\ref{assum:generalTVW1}.
This concludes since we have obtained that
\[\begin{pmatrix}
G \\ G' 
\end{pmatrix} = \Psi_z^{-1}\begin{pmatrix}
X\\ Y
\end{pmatrix} = \frac{1}{\sqrt{1-\eta^2}}  \begin{pmatrix}
  K(x,X)- \eta v \\
  Y -   \eta \Phi^2\po x , K(x,X) \pf 
\end{pmatrix}\,.\]
As a consequence, writing $j_z(y) = |\det(\na \Psi_{z}^{-1}(y))|$ for any~$y$, the law $\hat \pi_1(\delta_z)$ admits a density $f_z$ given by~$f_z(y) =  (2\pi)^{-d/2} j_z(y) e^{-|\Psi_z^{-1}(y)|^2/2}$. 
Denoting by~$\rho$ the density of the standard Gaussian law~$\mathcal N(0,I_{2d})$, 
\[\mathrm{Ent}\po \hat \pi_1(\delta_{z'}) |\hat \pi_1(\delta_z)\pf = \int_{\R^{2d}} \ln \po \frac{f_{z'}}{f_{z}}\pf f_{z'} 
 =   \int_{\R^{2d}} \ln \po \frac{\rho}{g_{z,z'}}\pf \rho\]
with
\[g_{z,z'}(y) = \frac{f_{z}\po \Psi_{z'}(y)\pf \rho(y) }{f_{z'}\po \Psi_{z'}(y)\pf  } = |\det(\na  \Psi_{z'}(y))| f_{z}\po  \Psi_{z'}(y)\pf, \]
which is the density of the image of~$f_{z}$ by~$\Psi_{z'}^{-1}$, namely is the image of~$\rho$ by~$\Psi_{z'}^{-1}\circ \Psi_{z}=:\Psi_{z,z'}$. As established in the proof of~\cite[Lemma 15]{MR4497240},  provided
\begin{equation}\label{eq:conditionPsi}
|\na \Psi_{z,z'}(\bG) - I_{2d}|\leq \frac12 \,,
\end{equation}
it holds that
\begin{equation}\label{eq:entropie_kernel}
\int_{\R^{2d}} \ln \po \frac{\rho}{g_{z,z'}}\pf \rho \leq \mathbb E\po \frac12 |\Psi_{z,z'}(\bG)-\bG|^2 + \|\na \Psi_{z,z'}(\bG) - I_{2d}\|_F^2 \pf \,.
\end{equation}
 It remains to bound this expectation and to establish \eqref{eq:conditionPsi}. The computations are exactly the same as in \cite[Proposition 3]{monmarche2023entropic} using the bounds in Assumption~\ref{assum:generalTVW1}, and thus we omit them. They lead to~$|\na \Psi_{z,z'}(\bG) - I_{2d}|\leq  C' T^2$, so that \eqref{eq:conditionPsi} holds for $T \leq T_0' = \min(T_0,(2C')^{-1/2})$, and to
 \[ \mathbb E\po \frac12 |\Psi_{z,z'}(\bG)-\bG|^2 + \|\na \Psi_{z,z'}(\bG) - I_{2d}\|_F^2 \pf \leq \frac{C'}{(1-\eta^2) T^2}|z-z'|^2\,.\]
 This concludes the proof of Theorem~\ref{thm:generalTVW1} in the case $n=1$.
 
 The case $n>1$ with $\eta>0$  follows the same lines, but is computationally more involved as, instead of coupling two trajectories starting from $z$ and $z'$ in a single step, we consider a bridge with $n$ steps.

More specifically, fix $z,z'\in\R^{2d}$. For~$\bG=(\bG_1,\dots,\bG_n)$ i.i.d. standard Gaussian variables on~$\R^{2d}$ (where we decompose~$\bG_k = (G_k,G_k')\in\R^d\times\R^d$), denote by~$\Psi_z^n(\bG)$ the state of a chain starting from~$z$ after~$n$ transitions, using the variables~$(G_k,G_k')$ in the two randomization steps of the~$k^{th}$ transition for $k\in\{1,\dots,n\}$. In other words, defining by induction~$z_0=z$ and then~$z_{k+1}=\Psi_{z_k}(\bG_{k+1})$ (omitting here for simplicity the dependency in~$\theta_k$ in the definition of $\Psi_{z_k}$ for the~$k^{th}$ transition), we have~$\Psi_z^n(\bG)= z_{n}$. Our goal is to define a function~$\Psi_{z,z'}^n:\R^{2dn} \rightarrow \R^{2dn}$ in such  a way that~$\bW = \Psi_{z,z'}^n(\bG)$ satisfies~$\Psi_{z'}^n(\bW)=\Psi_z^n(\bG)$. There could be many ways to enforce this, for instance we could merge the two chains in one step, namely take~$\bW_1=(W_1,W_1')$ as in the first part of the proof and then~$\bW_k = \bG_k$ for all~$k\geqslant 2$. However, for fixed~$z,z'$, this would be a highly unlikely trajectory starting from~$z'$ for small values of~$t$, i.e. the law of~$\bW_1$ would be far from a standard Gaussian law on~$\R^{2d}$.

Let~$y_0,\dots,y_n\in\R^{2d}$ be a fixed deterministic sequence to be determined, with~$y_0  = z'-z$ and~$y_n=0$. We define the function~$\Psi_{z,z'}^n$ by the fact~$\bW = \Psi_{z,z'}^n(\bG)$ satisfies~$\Psi_{z'}^k(\bW_1,\dots,\bW_k)=\Psi_z^k(\bG_1,\dots,\bG_k) + y_k$ for all~$k\in\{1,\dots,n\}$.  In other words,~$\bW=(\bW_1,\dots,\bW_n)$ is such that if two chains start respectively at~$z$ and~$z'$ and use respectively the variables~$\bG_k=(G_k,G_k')$ and~$\bW_k=(W_k,W_k')$ in the randomization steps of the~$k^{th}$ transitions then after~$k$ transitions the difference between the states of the two chains is~$y_k$, for all~$k\in\{ 0,\dots,n\}$.

Since $y_n=0$, this construction implies that $\Psi_{z'}^n(\bW) =\Psi_z^n(\bG)$ which, following the argument of the proof in the case $n=1$, implies that
\begin{equation}\label{eq:BorneGauss_n}
\mathrm{Ent}\po \hat\pi_n(\delta_{z'}) |\hat\pi_n(\delta_{z})\pf  \leq  \mathbb E\po \frac12 |\Psi_{z,z'}^n(\bG)-\bG|^2 + \|\na \Psi_{z,z'}^n(\bG) - I_{2dn}\|_F^2 \pf\,,
\end{equation}
provided
\begin{equation}\label{eq:condition_n}
|\na \Psi_{z,z'}^n(\bG) - I_{2dn}| \leq \frac12\,.
\end{equation}
It remains to bound the right hand side of \eqref{eq:BorneGauss_n} and to establish \eqref{eq:condition_n}.

Again, the computations are similar to the proof of  \cite[Proposition 3]{monmarche2023entropic} using the bounds in Assumption~\ref{assum:generalTVW1}, hence omitted. These computations lead to the design of the following bridge   $y_k=(u_k,w_k)$ for $k\in\{1,\dots,n-1\}$ to get a suitable scaling of the final estimate in the regime $t\rightarrow0$. We set 
\begin{eqnarray*}
w_k  &=&  \po 1 - \frac{k}{n}\pf \po v'-v\pf  - \frac{3k(n-k)}{(n^3 -n)\eta T } \po 2(x'-x)+ \eta T(n+1)(v'-v)\pf \\
&= & \po 1 - \frac{k}{n} - \frac{3k(n-k)}{n^2 -n }\pf \po v'-v\pf  - \frac{6k(n-k)}{(n^3 -n)\eta T } (x'-x)\,,
\end{eqnarray*}
which is designed so that
\[w_0 = v'-v\,,\qquad w_n= 0\,,\qquad \eta T \sum_{k=0}^{n} w_k =  x-x'\,.\]
Hence, setting~$u_k = x'-x + \eta T \sum_{j=0}^{k-1} w_j$, 
we get
\begin{equation*}
u_0 = x'-x\,,\qquad u_n=0 \qquad\text{and}\qquad u_{k+1} = u_k + \eta T w_k
\end{equation*}
for all  $k\in\{ 0,\dots,n-1\}$. This can be interpreted as follows: since there is no noise directly in the position, we simply let the difference between the positions of the two chains evolve deterministically following the difference between the velocities (which are coupled along time so that both positions and velocities merge after $n$ steps).

The computations of~\cite{monmarche2023entropic} show that, as soon as $nT\leq 1$, we have~$|\na \Psi_{z,z'}^n(\bG) - I_{2dn}| \leq C' T$, 
so that~\eqref{eq:condition_n} holds if~$T \leq 1/(2C') $ and, writing $s=nT$,
\begin{multline*}
\mathbb E\po \frac12 |\Psi_{z,z'}^n(\bG)-\bG|^2 + \|\na \Psi_{z,z'}^n(\bG) - I_{2dn}\|_F^2 \pf \\ 
\leqslant C' \frac{s}{\gamma^*} \po \frac{1}{\eta s^2} + \frac{\gamma^*}{\eta s} + 1\pf^2 |z-z'|^2 \,,    
\end{multline*}
which concludes the proof of Theorem~\ref{thm:generalTVW1}.
 \end{proof}
 
 In the next statement, we consider the case where $(\theta_k)_{k\in\N}$ are possibly random (but independent from the initial condition and the Gaussian variables appearing in \eqref{eq:DVD}). For $n\in\N$ and $\nu\in\mathcal P(\R^{2d})$, write~$\Pi_n(\nu) = \mathbb E \po \hat\pi_n(\nu)\pf$, 
 where the expectation runs over the law of $(\theta_k)_{k=1,\dots,n}$. Thus, $\Pi_n(\nu)$ is the law of the chain with initial distribution $\nu$ and transitions \eqref{eq:DVD} with random $\theta_k$'s after $n$ steps.
 
 \begin{corollary}\label{cor:TVW1}
 Under the settings and with the notations of Theorem~\ref{thm:generalTVW1}, for all $\nu,\nu'\in\mathcal P(\R^{2d})$,
 \[\| \Pi_1(\nu) - \Pi_1(\nu')\|_{TV}^2  \leq   \frac{C'}{(1-\eta)T^2}\mathcal W_1^2(\nu,\nu') \,, \]
and, for any $n\geqslant 1$, assuming furthermore that $\eta>0$ and $s:=nT\leq 1$,
\[\| \Pi_n(\nu) - \Pi_n(\nu')\|_{TV}^2 \leq   C' \frac{s}{\gamma^*} \po \frac{1}{\eta s^2} + \frac{\gamma^*}{\eta s} + 1\pf^2 \mathcal W_1^2(\nu,\nu') \,. \]
 \end{corollary}
 
 \begin{proof}
Given initial conditions $(Z_0,Z_0')$ sampled according to an arbitrary coupling of $\nu$ and $\nu'$, conditioning on these initial conditions and the $\theta_k$'s and using Theorem~\ref{thm:generalTVW1} and  Pinsker's inequality, we get 
 \[\| \hat\pi_n(\delta_{Z_0}) - \hat\pi_n(\delta_{Z_0'})\|_{TV}^2 \leq   C' \frac{s}{\gamma^*} \po \frac{1}{\eta s^2} + \frac{\gamma^*}{\eta s} + 1\pf^2    |Z_0-Z_0'|^2  \,. \]
 Taking $(Z,Z')$ as an optimal TV coupling  of $\hat\pi_n(\delta_{Z_0})$ and $\hat\pi_n(\delta_{Z_0'})$ conditionally to $(Z_0,Z_0')$ and the $\theta_k$'s, we get that $(Z,Z')$ is a coupling of $\Pi_n(\nu)$ and $\Pi_n(\nu')$ with
 \begin{equation*}
 \| \Pi_n(\nu) - \Pi_n(\nu')\|_{TV}  \leq   2 \mathbb P \po Z\neq Z'\pf \leq   \sqrt{C' \frac{s}{\gamma^*}} \po \frac{1}{\eta s^2} + \frac{\gamma^*}{\eta s} + 1\pf  \mathbb E\po |Z_0-Z_0'|\pf. 
 \end{equation*}
 Taking the infimum over all couplings of $\nu$ and $\nu'$ concludes.
 \end{proof}
 
 \subsection{Application: proof of Theorem~\ref{prop:TV/W1} }\label{subsec:appliTVW1}
 
The proof  that Assumption~\ref{assum:generalTVW1} holds for the Verlet and randomized midpoint integrators under the settings of Theorem~\ref{prop:TV/W1} follows from exactly the same computations as in the continuous-time Hamiltonian flow, see \cite[Section 3]{monmarche2023entropic} (the differential equation satisfied by the flow, used in \cite[Lemma 1]{monmarche2023entropic}, being replaced the discrete-time version of \cite[Lemma 10]{camrud2023second}). In fact all the conditions of  Assumption~\ref{assum:generalTVW1} are established for the Verlet integrator in \cite{MR4497240}. Following these computations, under the settings of Theorem~\ref{prop:TV/W1}, the constants $C,T_0$ in Assumption~\ref{assum:generalTVW1} only depend on $L$ and $L_H$.

As a consequence, Corollary~\ref{cor:TVW1} applies under the assumptions made in Theorem~\ref{prop:TV/W1}, when $\Phi_T^{\theta}$ is given by \eqref{hameqs} (in the parameter $\theta$, we consider the variables $\hat\theta$ involved in the stochastic gradient to be fixed, but in the randomized midpoint case, we consider the uniform variables $\hat u$ to be random, so that their contribution is averaged in the law $\Pi_n(\nu)$).

Denote by $\mathcal D_{\eta}$ the Markov operator associated to the velocity refreshment with parameter $\eta$, namely~$\mathcal D_\eta f(x,v) = \mathbb E( f( x,\eta v + \sqrt{1-\eta^2} G) )$ with~$G\sim\mathcal N(0,I_d)$. Then~$\mathcal D_{\eta}^2 = \mathcal D_\eta$. Fixing $n>1$ and a sequence $(\theta_k)_{k=0,\dots,n-1}$, denote $\mathcal V f(x,v) =   f ( \Phi_T^{\theta_{n-1}}(x,v))$ the operator associated to the Hamiltonian integration in the $n^{th}$ transition. Denote by $\hat \pi_n'(\nu)$ the law $\hat\pi_n(\nu)$ when $\eta$ is replaced by $\sqrt{\eta}$ in the transition~\eqref{eq:DVD}. Starting from an initial distribution $\nu$, performing $n>1$ steps of SGgHMC as defined in Section~\ref{sec:intro} (producing a r.v. distributed according to $\pi_n(\nu)$) is equivalent to perform a transition $\mathcal D_{\sqrt{\eta}}$, then $n-1$ transitions~\eqref{eq:DVD}, again a  transition $\mathcal D_{\sqrt{\eta}}$ and then a transition $\mathcal V$. In other words,
\[\pi_n(\nu) = \mathbb E\po \po \hat\pi_{n-1}'\po \nu \mathcal D_{\sqrt{\eta}} \pf\pf \mathcal D_{\sqrt{\eta}}\mathcal V\pf\,, \]
where the expectation runs over the distribution of the variables $\hat u$ in the case of the randomized midpoint integrator (as in the definition of $\Pi_n$).

Applying any Markov kernel does not increase total variation, and a parallel coupling shows that $\mathcal D_{\sqrt{\eta}}$ does not increase the $\mathcal W_1$ distance. Hence, considering two initial distributions $\nu_1,\nu_2$ and using for both chains the same variables $\hat u$ in the randomized midpoint case, we get
\begin{eqnarray*}
\lefteqn{\|\pi_n(\nu_1)-\pi_n(\nu_2) \|_{TV}}\\
& \leqslant & \mathbb E\po \left\| \po \hat\pi_{n-1}'\po \nu_1 \mathcal D_{\sqrt{\eta}} \pf\pf \mathcal D_{\sqrt{\eta}}\mathcal V - \po \hat\pi_{n-1}'\po \nu_2 \mathcal D_{\sqrt{\eta}} \pf\pf \mathcal D_{\sqrt{\eta}}\mathcal V\right\|_{TV}\pf \\
& \leqslant & \mathbb E\po \left\| \po \hat\pi_{n-1}'\po \nu_1 \mathcal D_{\sqrt{\eta}} \pf\pf  - \po \hat\pi_{n-1}'\po \nu_2 \mathcal D_{\sqrt{\eta}} \pf\pf  \right\|_{TV}\pf \\
& \leqslant &   \sqrt{\frac{C'  s}{\gamma^*}} \po \frac{1}{\eta s^2} + \frac{\gamma^*}{\eta s} + 1\pf \mathcal W_1\po \nu_1 \mathcal D_{\sqrt{\eta}} ,\nu_2 \mathcal D_{\sqrt{\eta}} \pf\\
& \leqslant &  \sqrt{\frac{C'  s}{\gamma^*}} \po \frac{1}{\eta s^2} + \frac{\gamma^*}{\eta s} + 1\pf  \mathcal W_1 \po \nu_1  ,\nu_2  \pf\,,
\end{eqnarray*}
where we applied Theorem~\ref{thm:generalTVW1} with $s=(n-1)T$. However, since $n\geqslant 2$, we can bound $n/2 \leq n-1\leq n$ to get a similar bound but now with $s=nT$, up to replacing $C'$ by $16C'$. This concludes the proof of Theorem~\ref{prop:TV/W1}.

\section{Proofs for convex region}\label{conapp}
\begin{proof}[Proof of Proposition~\ref{lba}]

By the mean value theorem, it holds that
\begin{equation}\label{mvt}
\begin{pmatrix}
\bar{q} - \bar{q}' + \frac{h}{2}(\bar{p} - \bar{p}')  - \frac{h^2}{4}(b(\bar{q},\theta) - b(\bar{q}',\theta)) \\
\bar{p} - \bar{p}' - \frac{h}{2}(b(\bar{q},\theta) - b(\bar{q}',\theta))
\end{pmatrix}  = \begin{pmatrix}
I_d - \frac{h^2}{4}H & \frac{h}{2}I_d\\
-\frac{h}{2}H & I_d
\end{pmatrix}
\begin{pmatrix}
\bar{q} - \bar{q}'\\
\bar{p} - \bar{p}'
\end{pmatrix},
\end{equation}
where
\begin{equation}\label{mvth}
H = \int_0^1 \nabla\!_x b(\bar{q}' + t(\bar{q} - \bar{q}'),\theta)dt.
\end{equation}
By~\eqref{A1a} and~\eqref{A1b},~$H$ 
is positive definite with eigenvalues between~$\frac{m}{2}$ and~$L$ for~$\bar{q},\bar{q}'$ satisfying~$\abs{\bar{q} - \bar{q}'}\geq 4R(1+\frac{L}{m})$. To prove~\eqref{lbaeq}, it suffices to show that the matrix
\begin{align}
\mathcal{H}&:=(1-c)
\begin{pmatrix}
I_d & 0\\
0 & \eta_0 I_d
\end{pmatrix}M
\begin{pmatrix}
I_d & 0\\
0 & \eta_0 I_d
\end{pmatrix}\nonumber\\
&\quad-
\begin{pmatrix}
I_d - \frac{h^2}{4}H & \frac{h}{2}I_d\\
-\frac{h}{2}H& I_d
\end{pmatrix}^{\!\!\!\top\!\!\!}
\begin{pmatrix}
I_d & 0\\
0 & \eta_1I_d
\end{pmatrix}
M
\begin{pmatrix}
I_d & 0\\
0 & \eta_1I_d
\end{pmatrix}
\begin{pmatrix}
I_d - \frac{h^2}{4}H & \frac{h}{2}I_d\\
-\frac{h}{2}H & I_d
\end{pmatrix}\label{ch}
\end{align}
is positive definite for~$\abs{\bar{q}-\bar{q}'}\geq 4R(1+\frac{L}{m})$.
By Proposition~4.2 in~\cite{lc23}, the matrix~$\mathcal{H}$ is positive definite if and only if the square matrices~$A,B,C$ given by
\begin{equation}\label{habc}
\begin{pmatrix}
A & B\\
B & C
\end{pmatrix} = \mathcal{H}
\end{equation}
are such that~$A$ and~$AC-B^2$ are positive definite. Denoting\footnote{There is abuse of notation with the function~$b(\cdot,\cdot)$, which should not introduce confusion given context.} the off-diagonal entry of~$M$ as
\begin{equation}\label{bdef}
b = \frac{h}{2(\eta_0-\eta_1)},
\end{equation}
it holds by direct calculation that
\begin{subequations}\label{abcdefs}
\begin{align}
A &= -cI_d + \frac{h}{2}(h + 2b\eta_1) H - \frac{h^2}{4}\bigg(\frac{h^2}{4}+bh\eta_1 + 2b^2\eta_1^2\bigg) H^2\label{Adef}\\
B &= (1-c)b\eta_0I_d - \bigg(\frac{h}{2} + b\eta_1\bigg)\bigg(I_d-\frac{h^2}{4}H\bigg) + \frac{h}{2}\bigg(\frac{bh}{2}\eta_1 + 2b^2\eta_1^2\bigg)H\label{Bdef}\\
C &= \bigg(2b^2(1-c)\eta_0^2 - \frac{h^2}{4} - bh\eta_1 - 2b^2\eta_1^2\bigg)I_d.\label{Cdef}
\end{align}
\end{subequations}
For any eigenvalue~$\frac{m}{2}\leq\lambda\leq L$ of~$H$, the matrix~$A$ admits the same corresponding eigenspace with eigenvalue~$A_{\lambda}$ given by
\begin{equation*}
A_{\lambda} = -c + \frac{h}{2}(h+2b\eta_1)\lambda - \frac{h^2}{4}\bigg( \frac{bh}{2}(\eta_0+\eta_1)+ 2b^2\eta_1^2\bigg)\lambda^2.
\end{equation*}
Using the definition~\eqref{bdef} for~$b$ as well as that for~$c$,~$A_{\lambda}$ can be rewritten as
\begin{align}
A_{\lambda}&= -\frac{bhm\eta_0}{8} + hb\eta_0\lambda -  \frac{\lambda^2bh^3}{8}(\eta_0+\eta_1)  - \frac{\lambda^2b^2h^2}{2}\eta_1^2\nonumber\\ 
&= bh\bigg(-\frac{m\eta_0}{8} + \lambda\eta_0 - \frac{h^2}{8}\lambda^2(\eta_0+\eta_1) - \frac{bh}{2}\lambda^2\eta_1^2\bigg),\label{al}
\end{align}
which, by~$\frac{m}{2}\leq \lambda\leq L$,~$\eta_1 \leq 1$ and~$h^2\leq \frac{(\eta_0-\eta_1)^2}{L}\leq \frac{1}{4L}$ 
implies
\begin{equation}\label{alp0}
A_{\lambda}\geq bh\lambda\eta_0\bigg(1-\frac{1}{4} - \frac{1}{16}\bigg) - \frac{\lambda bh}{4}(\eta_0-\eta_1)\geq \frac{7}{16}\lambda bh\eta_0 + \frac{1}{4}\lambda b h\eta_1
\end{equation}
and therefore that~$A_{\lambda}$ is positive.
Moreover, the matrix~$AC-B^2$ admits the same eigenspace with eigenvalue~$\mu$ given by
\begin{equation}\label{ml}
\mu = A_{\lambda}C_{\lambda} - B_{\lambda}^2,
\end{equation}
where
\begin{align}
C_{\lambda}&= 2b^2(1-c)\eta_0^2 - \frac{h^2}{4} - bh\eta_1 - 2b^2\eta_1^2,\nonumber\\
B_{\lambda}&=(1-c)b\eta_0 - \bigg(\frac{h}{2} + b\eta_1\bigg)\bigg(1-\frac{h^2}{4}\lambda\bigg) + \frac{h}{4}\bigg(bh\eta_1 + 4 b^2\eta_1^2\bigg)\lambda.\label{Bldef}
\end{align}
By definition 
of~$b$ and of~$c$,~$C_{\lambda}$ satisfies
\begin{align}
C_{\lambda} &= 2b^2(\eta_0^2 - \eta_1^2) - bh\bigg(\frac{ b^2m\eta_0^3}{4} + \frac{1}{2}(\eta_0-\eta_1) + \eta_1\bigg)\nonumber\\
&=  2b^2(\eta_0^2 - \eta_1^2) - bh\bigg(\frac{b^2m\eta_0^3}{4} + \frac{1}{2}(\eta_0+\eta_1)\bigg)\label{cl}
\end{align}
and~$B_{\lambda}$ satisfies
\begin{align}
B_{\lambda} &= b(\eta_0 - \eta_1) - \frac{h}{2} - bh \bigg(\frac{mb\eta_0^2}{8} - \frac{h\lambda\eta_1}{2} - \frac{h\lambda}{4}(\eta_0-\eta_1) \bigg) + b^2h\eta_1^2\lambda\nonumber\\
&= b^2h\eta_1^2\lambda -  \frac{mb^2h\eta_0^2}{8} + \frac{bh^2\lambda}{4}(\eta_0+\eta_1).\label{bl}
\end{align}
In the following, the terms of~$\mu$ from~\eqref{ml},~\eqref{al},~\eqref{cl} and~\eqref{bl} are organized in terms of order in~$h$, with~$\eta_0 - \eta_1$ considered as~$O(h)$ and~$b$ considered as~$O(1)$ as~$h\rightarrow 0$. The order~$h^2$ terms in~$\mu$ are given by
\begin{align*}
\mu_2 &:= bh\bigg(-\frac{m\eta_0}{8} + \lambda\eta_0\bigg)\bigg(2 b^2(\eta_0^2 - \eta_1^2) - \frac{mb^3h\eta_0^3}{4} - \frac{bh}{2}(\eta_0 + \eta_1)\bigg) \\
&\quad- b^4\lambda^2h^2\eta_1^4 + \frac{1}{4}\lambda mb^4h^2\eta_0^2\eta_1^2 - \frac{m^2b^4h^2\eta_0^4}{64},
\end{align*}
which can be rewritten as
\begin{align*}
\mu_2 &= bh\bigg( - \frac{ mb^2\eta_0}{4}(\eta_0^2 - \eta_1^2) + \frac{ m^2b^3h\eta_0^4}{32} + \frac{mbh\eta_0}{16}(\eta_0 + \eta_1) \\
&\quad+ 2 b^2\lambda\eta_0(\eta_0^2-\eta_1^2) - \frac{\lambda mb^3h\eta_0^4}{4} - \frac{\lambda bh\eta_0}{2}(\eta_0 + \eta_1)\\
&\quad -2\lambda^2 b^4\eta_1^4(\eta_0 - \eta_1) + \frac{\lambda m b^3h\eta_0^2\eta_1^2}{4} - \frac{m^2b^3h\eta_0^4}{64}\bigg)
\end{align*}
and subsequently bounded 
as
\begin{align*}
\mu_2&\geq bh\bigg(\lambda b^2\eta_0(\eta_0^2-\eta_1^2) - 2\lambda^2b^4\eta_1^4(\eta_0 - \eta_1)\\
&\quad  -\frac{m}{16}\Big(4b^2\eta_0(\eta_0^2 - \eta_1^2)- bh\eta_0(\eta_0 + \eta_1) + 4\lambda b^3h\eta_0^2(\eta_0^2-\eta_1^2)\Big)\bigg).
\end{align*}
Furthermore, by also~$\frac{m}{2}\leq \lambda$,~$b^2\leq \frac{1}{4L}$ and~$1-\eta_0(\eta_0 - \eta_1) >0$, it holds that
\begin{align*}
&\mu_2 - bh\bigg(\lambda b^2\eta_0(\eta_0^2-\eta_1^2) - 2\lambda^2 b^4\eta_1^4(\eta_0 - \eta_1)\bigg)\\
&\quad\geq bh\bigg( -\frac{\lambda b^2\eta_0}{2}(\eta_0^2 - \eta_1^2) + \frac{mbh\eta_0}{16}(\eta_0 + \eta_1)(1-\eta_0(\eta_0-\eta_1)) \bigg)\\
&\quad> -\frac{\lambda b^3h\eta_0}{2}(\eta_0^2-\eta_1^2),
\end{align*}
so that
\begin{equation}\label{m2l}
\mu_2 > \frac{\lambda b^3h\eta_0}{2}(\eta_0^2-\eta_1^2) - 2\lambda^2b^5h\eta_1^4(\eta_0-\eta_1).
\end{equation}
The order~$h^3$ terms in~$\mu$ are given by
\begin{align*}
\mu_3 &:= -\frac{\lambda^2b^2h^2\eta_1^2}{2}\bigg(2b^2(\eta_0^2-\eta_1^2) - \frac{mb^3h\eta_0^3}{4} - \frac{bh}{2}(\eta_0+\eta_1)\bigg)\\
&\quad - \frac{\lambda^2b^3h^3\eta_1^2}{2}(\eta_0
+ \eta_1) + \frac{\lambda m b^3h^3\eta_0^2}{16}(\eta_0 + \eta_1),
\end{align*}
which, by considering terms without a factor of~$m$ and using the inequality~$h^2 \leq \frac{\eta_0-\eta_1}{2L}$, 
satisfies
\begin{equation}\label{m3l}
\mu_3 \geq -2\lambda^2b^5h\eta_1^2(\eta_0-\eta_1)(\eta_0^2- \eta_1^2) - \frac{\lambda b^3h\eta_1^2}{8}(\eta_0^2-\eta_1^2).
\end{equation}
Finally, the order~$h^4$ terms in~$\mu$ are given by
\begin{align*}
\mu_4 &:= -\frac{\lambda^2 bh^3(\eta_0+\eta_1)}{8}\bigg(2b^2(\eta_0^2-\eta_1^2) - \frac{mb^3h\eta_0^3}{4} - \frac{bh(\eta_0 + \eta_1)}{2}\bigg) \\
&\quad- \frac{\lambda^2b^2 h^4}{16}(\eta_0+\eta_1)^2,
\end{align*}
which, by~$\eta_0 > \eta_1$ and~$h^2 \leq \frac{1}{4L}$, satisfies
\begin{equation}\label{m4l}
\mu_4 \geq -\frac{\lambda^2 b^3h^3\eta_0}{2}(\eta_0^2-\eta_1^2) \geq -\frac{\lambda b^3h\eta_0}{8}(\eta_0^2-\eta_1^2).
\end{equation}
Gathering~\eqref{m2l},~\eqref{m3l},~\eqref{m4l} and using~$\eta_0-\eta_1\leq \frac{1}{2}$ 
gives
\begin{align*}
\mu &= \mu_2 + \mu_3 + \mu_4\\
&> \frac{\lambda b^3h\eta_0}{2}(\eta_0^2-\eta_1^2) - 2\lambda^2b^5h\eta_0^2\eta_1^2(\eta_0-\eta_1) - \frac{\lambda b^3h\eta_1^2}{8}(\eta_0^2-\eta_1^2)\\ 
&\quad- \frac{\lambda b^3h\eta_0}{8}(\eta_0^2-\eta_1^2),
\end{align*}
which, using~$\eta_1<\eta_0\leq1$ and~$b^2\leq \frac{1}{4L}$ 
implies
\begin{equation}
\mu > \frac{\lambda b^3h\eta_0}{4}(\eta_0^2-\eta_1^2) - \frac{\lambda b^3h\eta_0^2\eta_1}{4}(\eta_0^2-\eta_1^2) >0. \label{ul}
\end{equation}
\end{proof}

\begin{proof}[Proof of Proposition~\ref{lab}]
By the mean value theorem, it holds that
\begin{equation*}
\begin{pmatrix}
\bar{q} - \bar{q}' + \frac{h}{2}(\bar{p} - \bar{p}')\\
\bar{p} - \bar{p}' - \frac{h}{2}(b(\bar{q} + \frac{h}{2}\bar{p},\theta) - b(\bar{q}' + \frac{h}{2}\bar{p}'))
\end{pmatrix}
=
\begin{pmatrix}
I_d & \frac{h}{2}I_d\\
-\frac{h}{2}\bar{H} & I_d - \frac{h^2}{4}\bar{H}
\end{pmatrix}
\begin{pmatrix}
\bar{q} - \bar{q}'\\
\bar{p} - \bar{p}'
\end{pmatrix},
\end{equation*}
where
\begin{equation}\label{barH}
\bar{H} = \int_0^1 \nabla\!_xb\bigg(\bar{q}' + \frac{h}{2}\bar{p}' + t\bigg(\bar{q} - \bar{q}' + \frac{h}{2}(\bar{p} - \bar{p}')\bigg),\theta\bigg)dt.
\end{equation}
As before for~$H$ in the proof of Proposition~\ref{lba},~$\bar{H}$ is positive definite with eigenvalues between~$\frac{m}{2}$ and~$L$ for the range of~$\bar{q},\bar{q}',\bar{p},\bar{p}'$ under consideration here. Let~$\bar{\mathcal{H}}$ be defined by
\begin{align}
\bar{\mathcal{H}} &:= \begin{pmatrix}
I_d & 0\\
0 & \eta_0 I_d
\end{pmatrix}M
\begin{pmatrix}
I_d & 0\\
0 & \eta_0 I_d
\end{pmatrix}\nonumber\\
&\quad-
\begin{pmatrix}
I_d  & \frac{h}{2}I_d\\
-\frac{h}{2}\bar{H} & I_d - \frac{h^2}{4}\bar{H}
\end{pmatrix}^{\!\!\!\top\!\!\!}
\begin{pmatrix}
I_d & 0\\
0 & \eta_1I_d
\end{pmatrix}
M
\begin{pmatrix}
I_d & 0\\
0 & \eta_1I_d
\end{pmatrix}
\begin{pmatrix}
I_d & \frac{h}{2}I_d\\
-\frac{h}{2}\bar{H} & I_d - \frac{h^2}{4}\bar{H}
\end{pmatrix} \label{ch2}
\end{align}
and let~$\bar{A},\bar{B},\bar{C}$ be the square matrices given by~\eqref{habc} with~$\bar{A},\bar{B},\bar{C},\bar{\mathcal{H}}$ replacing~$A,B,C,\mathcal{H}$ respectively. Denoting again the expression~\eqref{bdef} with the constant~$b$, a direct calculation gives that
\begin{subequations}\label{ABCb}
\begin{align}
\bar{A} &= bh\eta_1\bar{H} - \frac{b^2h^2\eta_1^2}{2}\bar{H}^2,\\
\bar{B} &= b(\eta_0 - \eta_1)I_d - \frac{h}{2}I_d + \frac{bh^2\eta_1}{2}\bar{H} + b^2h\eta_1^2\bar{H}\bigg(I_d -\frac{h^2}{4}\bar{H}\bigg),\\
\bar{C} &= 2b^2(\eta_0^2 - \eta_1^2)I_d - \frac{h^2}{4}I_d - bh\eta_1I_d + \frac{bh^3\eta_1}{4}\bar{H} + b^2h^2\eta_1^2\bar{H} - \frac{b^2h^4\eta_1^2}{8}\bar{H}^2.
\end{align}
\end{subequations}
For any eigenvalue~$\frac{m}{2}\leq \bar{\lambda}\leq L$ of~$\bar{H}$, the matrix~$\bar{A}$ has the same corresponding eigenspace with eigenvalue~$\bar{A}_{\bar{\lambda}}$ satisfying
\begin{equation}\label{Als}
\bar{A}_{\bar{\lambda}} = \bar{\lambda} b h\eta_1 - \bar{\lambda}^2b^3 h\eta_1^2(\eta_0 - \eta_1) \geq \bar{\lambda} b h\bigg(\eta_1 - \frac{\eta_1}{8}\bigg) > 0,
\end{equation}
where we have used~$\eta_1\leq 1$,~$\eta_0 - \eta_1\leq \frac{1}{2}$ and~$b^2\leq \frac{1}{4L}$. Moreover, the matrix~$\bar{A}\bar{C}-\bar{B}^2$ has the same eigenspace with eigenvalue~$\bar{\mu}$ given by
\begin{align}
\bar{\mu} &= h\bar{\lambda}\bigg(b\eta_1 - \frac{b^2\bar{\lambda} h\eta_1^2}{2}\bigg)\bigg(2b^2(\eta_0^2 - \eta_1^2) - \frac{h^2}{4} - bh\eta_1 + \frac{\bar{\lambda} bh^3\eta_1}{4}\nonumber\\
&\quad + \bar{\lambda} b^2h^2\eta_1^2 - \frac{\bar{\lambda}^2 b^2h^4\eta_1^2}{8}\bigg) - \frac{h^2}{4}\bigg(\bar{\lambda} bh\eta_1 + 2\bar{\lambda}b^2\eta_1^2\bigg(1-\frac{\bar{\lambda} h^2}{4}\bigg)\bigg)^2.\label{mub}
\end{align}
In the following, the terms on the right-hand side of~\eqref{mub} are organized in terms of order in~$h$, with~$\eta_0-\eta_1$ considered as~$O(h)$ and~$b$ considered as~$O(1)$ as~$h\rightarrow 0$. The order~$h^2$ terms on the right-hand side of~\eqref{mub} are given by
\begin{align*}
\bar{\mu}_2&:= 2\bar{\lambda} b^3 h\eta_1 (\eta_0^2 - \eta_1^2) - \bar{\lambda} b^2h^2\eta_1^2 - \bar{\lambda}^2 b^4 h^2\eta_1^4\\
&\quad= \bar{\lambda} b^2 h^2\eta_1(\eta_0 + \eta_1) - \frac{\bar{\lambda} h^4 \eta_1^2}{4(\eta_0 - \eta_1)^2} - \bar{\lambda}^2 b^4 h^2\eta_1^4,
\end{align*}
and the order~$h^3$ terms on the right-hand side of~\eqref{mub} are given by
\begin{align*}
\bar{\mu}_3&:= -\frac{\bar{\lambda} bh^3\eta_1}{4}  + \bar{\lambda}^2 b^3h^3\eta_1^3 - \frac{\bar{\lambda}^2 b^2h^2\eta_1^2}{2}\bigg(2b^2(\eta_0^2 - \eta_1^2) - bh\eta_1\bigg) - \bar{\lambda}^2b^2h^3\eta_1^3.
\end{align*}
Therefore, it holds that
\begin{align*}
\bar{\mu}_2 + \bar{\mu}_3 &= \bar{\lambda}b^2h^2\eta_1(\eta_0+\eta_1) - \frac{\bar{\lambda} h^4 \eta_1^2}{4(\eta_0 - \eta_1)^2} -\frac{\bar{\lambda} bh^3\eta_1}{4} - \bar{\lambda}^2 b^4 h^2\eta_1^2\eta_0^2,
\end{align*}
which, by the assumptions~$h^2\leq \frac{(\eta_0-\eta_1)^2}{L}$,~$\eta_1 <\eta_0\leq 1$ and~$\bar{\lambda} \leq L$, gives
\begin{equation}\label{m23}
\bar{\mu}_2 + \bar{\mu}_3 \geq \bar{\lambda} b^2h^2\eta_1(\eta_0 + \eta_1)\bigg(1- \frac{1}{2}\bigg) - \frac{\bar{\lambda} b^2 h^2\eta_1}{2}(\eta_0 - \eta_1) - \frac{\bar{\lambda} b^2 h^2\eta_1^2}{4}= \frac{3\bar{\lambda} b^2 h^2\eta_1^2}{4}.
\end{equation}
The rest of the terms on the right-hand side of~\eqref{mub} with fourth order or higher in~$h$ are given by
\begin{align}
\bar{\mu}_{4+}&:= \bar{\lambda} b h \eta_1 \bigg[\frac{\bar{\lambda} b h^3 \eta_1}{4} - \frac{\bar{\lambda}^2b^2h^4\eta_1^2}{8}\bigg]\nonumber\\
&\quad  - \frac{\bar{\lambda}^2 b^2h^2\eta_1^2}{2}\bigg[-\frac{h^2}{4} + \frac{\bar{\lambda} bh^3\eta_1}{4} + \bar{\lambda} b^2h^2\eta_1^2 - \frac{\bar{\lambda}^2b^2h^4 \eta_1^2}{8}\bigg]\nonumber\\
&\quad -\frac{h^2}{4}\bigg[ \bar{\lambda}^2 b^2h^2\eta_1^2 - \bar{\lambda}^3 b^3h^3 \eta_1^3 + 4\bar{\lambda}^2 b^4 \eta_1^4\bigg(-\frac{\bar{\lambda} h^2}{2} + \frac{\bar{\lambda}^2 h^4}{16}\bigg) \bigg].\label{m4}
\end{align}
The second term in each of the square brackets on the right-hand side of~\eqref{m4} cancel each other, as do the last two terms in the second square bracket with those in the round bracket. In addition, the first term in the first square bracket cancels with the first term in the last square bracket. Consequently,~$\bar{\mu}_{4+}$ can be written simply as
\begin{equation}\label{m4s}
\bar{\mu}_{4+} = \frac{\bar{\lambda}^2b^2h^4\eta_1^2}{8}.
\end{equation}
Writing~$\bar{\mu} = \bar{\mu}_2 +\bar{\mu}_3 +\bar{\mu}_{4+}$ and gathering~\eqref{m23},~\eqref{m4s} concludes the proof by Proposition~4.2 in~\cite{lc23}.
\end{proof}

\begin{proof}[Proof of Proposition~\ref{lbaR}]
By the mean value theorem, it holds that
\begin{align}\label{mvtR}
&\begin{pmatrix}
\bar{q} - \bar{q}' + \frac{h}{2}(\bar{p} - \bar{p}')  - \frac{h^2}{4}(b(\bar{q}+\bar{u}\bar{p},\theta) - b(\bar{q}'+\bar{u}\bar{p}',\theta)) \\
\bar{p} - \bar{p}' - \frac{h}{2}(b(\bar{q}+\bar{u}\bar{p},\theta) - b(\bar{q}'+\bar{u}\bar{p}',\theta))
\end{pmatrix} \nonumber\\
&\quad= \begin{pmatrix}
I_d - \frac{h^2}{4}H & \frac{h}{2}I_d - \frac{h^2\bar{u}}{4}H\\
-\frac{h}{2}H & I_d - \frac{h\bar{u}}{2}H
\end{pmatrix}
\begin{pmatrix}
\bar{q} - \bar{q}'\\
\bar{p} - \bar{p}'
\end{pmatrix},
\end{align}
where
\begin{equation}\label{mvthR}
H = \int_0^1 \nabla\!_x b(\bar{q}'+\bar{u}\bar{p}' + t(\bar{q} +\bar{u}\bar{p} - \bar{q}'+\bar{u}\bar{p}'),\theta)dt.
\end{equation}
By~\eqref{A1a} and~\eqref{A1b},~$H$ is positive definite with eigenvalues between~$\frac{m}{2}$ and~$L$ for~$\bar{q},\bar{q}',\bar{p},\bar{p}'+\bar{u}\bar{p}$ satisfying~$\abs{\bar{q}+\bar{u}\bar{p} - \bar{q}'-\bar{u}\bar{p}'}\geq 4R(1+\frac{L}{m})$. To prove~\eqref{lbaeqR}, it suffices to show that the matrix
\begin{align}
\mathcal{H}&:=(1-c)
\begin{pmatrix}
I_d & 0\\
0 & \eta_0 I_d
\end{pmatrix}M
\begin{pmatrix}
I_d & 0\\
0 & \eta_0 I_d
\end{pmatrix} -
\begin{pmatrix}
I_d - \frac{h^2}{4}H & \frac{h}{2}I_d- \frac{h^2\bar{u}}{4}H\\
-\frac{h}{2}H& I_d - \frac{h\bar{u}}{2}H
\end{pmatrix}^{\!\!\!\top\!\!\!}
\nonumber\\
&\quad
\cdot
\begin{pmatrix}
I_d & 0\\
0 & \eta_1I_d
\end{pmatrix}
M
\begin{pmatrix}
I_d & 0\\
0 & \eta_1I_d
\end{pmatrix}
\begin{pmatrix}
I_d - \frac{h^2}{4}H & \frac{h}{2}I_d - \frac{h^2\bar{u}}{4}H\\
-\frac{h}{2}H & I_d - \frac{h\bar{u}}{2}H
\end{pmatrix}\label{chR}
\end{align}
is positive definite for~$\abs{\bar{q} + \bar{u}\bar{p}-\bar{q}'- \bar{u}\bar{p}'}\geq 4R(1+\frac{L}{m})$.
The matrix~$\mathcal{H}$ is positive definite if and only if the square matrices~$A,B,C$ given by~\eqref{habc} 
are such that~$A$ and~$AC-B^2$ are positive definite. Denoting again~\eqref{bdef},
it holds by direct calculation that~$A$ is given by~\eqref{Adef} and
\begin{subequations}\label{abcdefsR} 
\begin{align}
B &= (1-c)b\eta_0I_d + \bigg(I_d-\frac{h\bar{u}}{2}H\bigg)\bigg[-\bigg(\frac{h}{2} + b\eta_1\bigg)\bigg(I_d-\frac{h^2}{4}H\bigg)\nonumber\\
&\quad+ \frac{h}{2}\bigg(\frac{bh}{2}\eta_1 + 2b^2\eta_1^2\bigg)H\bigg],\label{BdefR}\\
C &= 2b^2(1-c)\eta_0^2I_d +\bigg(I_d-\frac{h\bar{u}}{2}{H}\bigg)^2\bigg(- \frac{h^2}{4} - bh\eta_1 - 2b^2\eta_1^2\bigg).\label{CdefR}
\end{align}
\end{subequations}
By definitions~\eqref{mdef} and~\eqref{bdef}, it holds that~$h/2+b\eta_1 = b\eta_0$,~$(bh/2)\eta_1 + 2b^2\eta_1^2=b^2(\eta_0+\eta_1)\eta_1$ and the expression~\eqref{BdefR} may be rewritten as
\begin{align}
B&=-(m\eta_0^2b^3/4)(\eta_0-\eta_1)I_d + b^3\eta_0 H(\eta_0-\eta_1)^2 + b^3H(\eta_0^2-\eta_1^2) \nonumber\\
&\quad+ (h\bar{u}/2)H (b\eta_0(I_d-b^2H(\eta_0-\eta_1)^2) - b^3H(\eta_0^2-\eta_1^2)) \label{hqo}
\end{align}
For any eigenvalue~$\frac{m}{2}\leq\lambda\leq L$ of~$H$, the matrix~$B$ admits the same corresponding eigenspace with eigenvalue~$B_{\lambda}$ given by~\eqref{hqo} with~$H$ replaced by~$\lambda$. In~$B_{\lambda}$, 
the first and third terms in~\eqref{hqo} correspond to
\begin{equation*}
- (m\eta_0^2b^3/4)(\eta_0-\eta_1) + b^3\lambda (\eta_0^2-\eta_1^2) \geq b^3\lambda (\eta_0-\eta_1)\eta_0 - (\lambda\eta_0b^3/2)(\eta_0-\eta_1) \geq 0.
\end{equation*}
Using~$b= h/(2(\eta_0-\eta_1))\leq 1/(2\sqrt{L})$ and~$\eta_1\leq \eta_0\leq 1$, the term in~$B_{\lambda}$ corresponding to the term with factor~$\bar{u}$ in~\eqref{hqo} is also positive. Therefore~$B_{\lambda}>0$ holds. Similarly, the square brackets in~\eqref{BdefR} with~$H$ replaced by~$\lambda$ is negative, which implies~$B_{\lambda}$ may be bounded above by~\eqref{hqo} with~$H$ replaced by~$\lambda$ and~$\bar{u}$ replaced by~$h$. Therefore it holds that
\begin{equation}\label{bsb}
B_{\lambda}^2 < (c_1b^3)^2 = c_1^2b^6,
\end{equation}
where
\begin{align*}
c_1 &= \lambda (-\eta_0(\eta_0-\eta_1)^2 +\eta_0^2-\eta_1^2 + 2\eta_0(\eta_0-\eta_1)) - m\eta_0^2(\eta_0-\eta_1)/4\\
&\leq \lambda(3\eta_0 + \eta_1)(\eta_0-\eta_1).
\end{align*}
On the other hand, by the same derivations,~$A_{\lambda}$ satisfies~\eqref{alp0}, where~$A_{\lambda}$ is the corresponding eigenvalue of~$A$. Moreover, the corresponding eigenvalue~$C_{\lambda}$ for~$C$ satisfies
\begin{align*}
C_{\lambda} &\geq 2b^2(1-c)\eta_0^2 - (h^2/4 + bh\eta_1 + 2b^2\eta_1^2)\\
&\geq 2b^2(1-c)\eta_0^2 - b^2((\eta_0-\eta_1)^2 + 2\eta_1(\eta_0-\eta_1) + 2\eta_1^2)\\
&= 2b^2(1-c)\eta_0^2 - b^2(\eta_0^2+\eta_1^2)\\
&= b^2(\eta_0^2-\eta_1^2) - mb^4\eta_0^3(\eta_0-\eta_1)/2\\
&\geq (7/8)b^2(\eta_0^2-\eta_1^2).
\end{align*}
Therefore it holds that
\begin{align*}
A_{\lambda}C_{\lambda} &\geq \bigg(\frac{7}{8}\eta_0 + \frac{1}{2}\eta_1\bigg) \lambda b^2(\eta_0-\eta_1)\cdot\frac{7}{8}b^2(\eta_0^2-\eta_1^2)\\
&= \frac{7}{8}\bigg(\frac{7}{8}\eta_0^2 + \frac{11}{8}\eta_0\eta_1 + \frac{1}{2}\eta_1^2\bigg) \lambda b^4 (\eta_0-\eta_1)^2,
\end{align*}
which, together with~\eqref{bsb} and~$b^2\leq 1/(4L)$, gives~$A_{\lambda}C_{\lambda}-B_{\lambda}^2> 0$ and concludes.
\end{proof}

\begin{proof}[Proof of Proposition~\ref{lba2}]
By equation~\eqref{mvt} with~\eqref{mvth}, it holds that
\begin{align*}
&\bigg\|\begin{pmatrix} 
I_d & 0\\
0 & \eta_1I_d
\end{pmatrix}
\begin{pmatrix}
\bar{q} - \bar{q}' + \frac{h}{2}(\bar{p} - \bar{p}') - \frac{h^2}{4}(b(\bar{q},\theta) - b(\bar{q}',\theta))\\
\bar{p} - \bar{p}' - \frac{h}{2}(b(\bar{q},\theta) - b(\bar{q}',\theta)) 
\end{pmatrix} \bigg\|_M\\
&\quad = (1-c)\bigg\| \begin{pmatrix}
I_d & 0\\
0 & \eta_0 I_d
\end{pmatrix}
\begin{pmatrix}
\bar{q} - \bar{q}'\\
\bar{p} - \bar{p}'
\end{pmatrix}\bigg\|_M
-
\begin{pmatrix} \bar{q} - \bar{q}' \\ \bar{p} - \bar{p}'\end{pmatrix}^{\!\!\!\top\!\!} 
\mathcal{H}
\begin{pmatrix}\bar{q}-\bar{q}'\\ \bar{p}-\bar{p}'\end{pmatrix},
\end{align*}
where~$\mathcal{H}$ is defined by~\eqref{ch}. Let the square matrices~$A$,~$B$ and~$C$ be given by~\eqref{habc} and satisfying~\eqref{abcdefs} with~\eqref{bdef}. Again, let the constant~$b$ be given by~\eqref{bdef}. For~$\bar{q},\bar{q}',\bar{p},\bar{p}'$ satisfying~$\abs{\bar{p} - \bar{p}'}^2 \geq (17/4)L\abs{\bar{q} - \bar{q}'}^2$ and~$h$ satisfying~$h^2\leq (\eta_0 - \eta_1)^2/L$, it holds that
\begin{equation}\label{pgq}
2b^2(\eta_0^2 - \eta_1^2)\abs{\bar{p} - \bar{p}'}^2 \geq (17/2)b^2L(\eta_0^2 - \eta_1^2) \abs{\bar{q} - \bar{q}'}^2 = (17/4)Lbh(\eta_0 + \eta_1)\abs{\bar{q} - \bar{q}'}^2,
\end{equation}
which, similar to~\eqref{cl}, implies
\begin{align*}
(\bar{p} - \bar{p}')^{\top}C(\bar{p} - \bar{p}') &\geq \bigg(\frac{2}{3}\cdot2b^2(\eta_0^2 - \eta_1^2) - bh\bigg(\frac{mb^2\eta_0^3}{4} + \frac{1}{2}(\eta_0 + \eta_1)\bigg)\bigg)\abs{\bar{p} - \bar{p}'}^2\\
&\quad+ (17/12)Lbh(\eta_0+\eta_1)\abs{\bar{q} - \bar{q}'}^2
\end{align*}
and therefore
\begin{equation}\label{hc}
\begin{pmatrix}
\bar{q} - \bar{q}'\\
\bar{p} - \bar{p}'
\end{pmatrix}^{\!\!\top\!\!}
\mathcal{H}
\begin{pmatrix}
\bar{q} - \bar{q}'\\
\bar{p} - \bar{p}'
\end{pmatrix} \geq
\begin{pmatrix}
\bar{q} - \bar{q}'\\
\bar{p} - \bar{p}'
\end{pmatrix}^{\!\!\top\!\!}
\mathcal{H}'
\begin{pmatrix}
\bar{q} - \bar{q}'\\
\bar{p} - \bar{p}'
\end{pmatrix},
\end{equation}
where~$\mathcal{H}'$ is given by~\eqref{habc} with~$\mathcal{H}',A',C'$ replacing~$\mathcal{H},A,C$ respectively, 
for square matrices~$A',C'$ given by
\begin{align}
A' &= (17/12)Lbh(\eta_0 + \eta_1)I_d + A,\nonumber\\
C' &= \bigg[\frac{2}{3}\cdot2b^2(\eta_0^2-\eta_1^2) - bh\bigg(\frac{mb^2\eta_0^3}{4} + \frac{1}{2}(\eta_0 + \eta_1)\bigg)\bigg]I_d\label{cl2}
\end{align}
and~$B$ given by~\eqref{abcdefs}. To prove~\eqref{lbaeq} for the~$\bar{q},\bar{q}',\bar{p},\bar{p}',h$ under consideration, it suffices to show that~$\mathcal{H}'$ is positive definite. By Proposition~4.2 in~\cite{lc23},~$\mathcal{H}'$ is positive definite if and only if~$A'$ and~$A'C'-B^2$ are positive definite. For any eigenvalue~$-L\leq \lambda \leq L$ of~$H$, the matrix~$A'$ admits the same corresponding eigenspace with eigenvalue~$A_{\lambda}'$ satisfying
\begin{align}
A_{\lambda}' &= \frac{17}{12}Lbh(\eta_0 + \eta_1) - c + \frac{h}{2}(h+2b\eta_1)\lambda - \frac{h^2}{4}\bigg( \frac{bh}{2}(\eta_0+\eta_1)+2b^2\eta_1^2 \bigg)\lambda^2\nonumber\\
&\geq \frac{17}{12}Lbh(\eta_0 + \eta_1) - c - \frac{h}{2}(h+2b\eta_1)L - \frac{h^2}{4}\bigg( bh\eta_0+\frac{\eta_1^2}{2L}\bigg)L^2\nonumber\\
&\geq \frac{17}{12}Lbh\eta_1 + bh\bigg(-\frac{m\eta_0}{8} + \frac{5}{12}\eta_0L 
- \frac{h^2}{4}L^2\eta_0 - \frac{L\eta_1^2}{4}(\eta_0 - \eta_1)\bigg),\label{alp}
\end{align}
where the definitions for~$b,c$ have been used along with the assumption on~$h$. 
By a similar estimate to~\eqref{alp0},~$A_{\lambda}'$ is positive. Moreover, the matrix~$A'C'- B^2$ admits the same eigenspace with eigenvalue~$\mu'$ given by~$\mu' = A_{\lambda}'C_{\lambda}'-B_{\lambda}^2$, where~$C_\lambda'$ is given by 
\begin{equation}\label{Clp}
C_{\lambda}' = \frac{2}{3}\cdot2b^2(\eta_0^2 - \eta_1^2) - bh\bigg(\frac{mb^2\eta_0^3}{4} + \frac{1}{2}(\eta_0 + \eta_1)\bigg)
\end{equation}
and~$B_{\lambda}$ is given by~\eqref{Bldef}. By~\eqref{bdef} and the bounds~$h^2\leq \frac{(\eta_0-\eta_1)^2}{L}$,~$m\eta_0^3\leq L \eta_0 \leq \frac{2L}{3}(\eta_0+\eta_1)$, it holds that
\begin{equation}\label{clpe}
C_{\lambda}' \geq 2b^2(\eta_0^2-\eta_1^2)\bigg(\frac{2}{3} - \bigg(\frac{1}{24} + \frac{1}{2}\bigg)\bigg)>0.
\end{equation}
In the following, the lower bound on~$\mu' = A_{\lambda}'C_{\lambda}' - B_{\lambda}^2$ resulting from~\eqref{alp},~\eqref{Clp} and~\eqref{bl} is organized in terms of order in~$h$ as~$h\rightarrow 0$. The factor~$\eta_0-\eta_1$ is considered~$O(h)$ and~$b$ is considered~$O(1)$. The order~$h^2$ terms of the lower bound are given by
\begin{align*}
\mu_2' &:= bh\bigg(-\frac{m\eta_0}{8} 
+ \frac{5}{12}L\eta_0+\frac{17}{12}L\eta_1\bigg)
\bigg( \frac{4}{3}b^2(\eta_0^2 - \eta_1^2)  \\
&\quad - bh\bigg(\frac{mb^2\eta_0^3}{4} + \frac{1}{2}(\eta_0 + \eta_1)\bigg)\bigg) - b^4h^2\eta_1^4\lambda^2 + \frac{mb^4h^2\eta_0^2\eta_1^2\lambda}{4} - \frac{m^2b^4h^2\eta_0^4}{64}.
\end{align*}
Using~\eqref{bdef},~$h^2\leq \frac{(\eta_0-\eta_1)^2}{L}$,~$m\leq L$,~$-L\leq \lambda\leq L$ and~\eqref{clpe}, we obtain
\begin{align}
\mu_2' & \geq \bigg(\frac{7}{24}\eta_0+\frac{17}{12}\eta_1\bigg)Lb^3h\bigg(\frac{4}{3} - \frac{1}{12} - 1\bigg)(\eta_0^2-\eta_1^2)\nonumber\\
&\quad - \frac{Lb^3h\eta_1^4}{2}(\eta_0-\eta_1) - \frac{Lb^3h\eta_0^2\eta_1^2(\eta_0-\eta_1)}{8} - \frac{Lb^3h\eta_0^4}{2\cdot 64}(\eta_0-\eta_1).\label{iakw}
\end{align}
By~$(1/2)\eta_0\leq\eta_1<\eta_0\leq1$, the last three terms on the right-hand side of~\eqref{iakw} satisfy
\begin{align}
&- \frac{Lb^3h\eta_1^4}{2}(\eta_0-\eta_1) - \frac{Lb^3h\eta_0^2\eta_1^2(\eta_0-\eta_1)}{8} - \frac{Lb^3h\eta_0^4}{2\cdot 64}(\eta_0-\eta_1)\nonumber\\
&\quad\geq Lb^3h\eta_1(\eta_0^2-\eta_1^2)\bigg(- \frac{1}{4} - \frac{1}{16} 
\bigg)
- \frac{1}{128}Lb^3h\eta_0(\eta_0^2-\eta_1^2).\nonumber\\
&\quad= -(5\eta_1/16 + \eta_0/128)Lb^3h(\eta_0^2-\eta_1^2).\label{m2pl}
\end{align}
Inserting~\eqref{m2pl} back into~\eqref{iakw} yields
\begin{equation}\label{m3ll}
\mu_2' \geq (25\eta_0/384 + \eta_1/16)Lb^3h(\eta_0^2-\eta_1^2).
\end{equation}
The order~$h^3$ terms are given by
\begin{align*}
\mu_3' &:= -\frac{Lbh\eta_1^2}{4}(\eta_0 - \eta_1)\bigg(\frac{2}{3}\cdot2b^2(\eta_0^2 - \eta_1^2) - \frac{mb^3h\eta_0^3}{4} - \frac{bh}{2}(\eta_0+\eta_1)\bigg)\\
&\quad -\frac{\lambda bh^3}{2}(\eta_0+\eta_1)\bigg(\lambda b^2\eta_1^2 - \frac{mb^2\eta_0^2}{8}\bigg),
\end{align*}
which, by considering only the terms in the brackets without a factor of~$m$ and using~$bh=2b^2(\eta_0-\eta_1)$, satisfies 
\begin{equation}\label{m3pl}
\mu_3' \geq -\bigg(\frac{1}{24}+\frac{1}{32}\bigg)Lb^3h\eta_1(\eta_0^2-\eta_1^2)=-\frac{7}{96}Lb^3h\eta_1(\eta_0^2-\eta_1^2).
\end{equation}
The order~$h^4$ terms are given by
\begin{equation*}
\mu_4' := -\frac{bh^3L^2\eta_0}{4}\bigg(\frac{2}{3}\cdot2b^2(\eta_0^2 - \eta_1^2) - \frac{mb^2h\eta_0^3}{4} - \frac{bh}{2}(\eta_0 + \eta_1)\bigg) - \frac{b^2h^4\lambda^2}{16}(\eta_0 + \eta_1)^2,
\end{equation*}
which, again considering only terms with a factor of~$m$, using~$\eta_1<\eta_0\leq 1$ and~$Lh^2\leq 1/16^2$, 
satisfies
\begin{equation}\label{m4pl}
\mu_4' \geq -\frac{L^2b^3h^3\eta_0}{12}(\eta_0^2-\eta_1^2) - \frac{L^2b^3h^3}{8}(\eta_0+\eta_1)(\eta_0^2-\eta_1^2) > -\frac{1}{3\cdot16^2}Lb^3h\eta_0(\eta_0^2 - \eta_1^2).
\end{equation}
Gathering~\eqref{m3ll},~\eqref{m3pl} and~\eqref{m4pl} gives~$\mu' \geq \mu_2' + \mu_3' + \mu_4' > 0$.
\end{proof}

\begin{proof}[Proof of Proposition~\ref{lab2}]
By the mean value theorem, it holds that
\begin{align*}
&\bigg\|
\begin{pmatrix}
I_d & 0\\
0 & \eta_1I_d
\end{pmatrix}
\begin{pmatrix}
\bar{q} - \bar{q}' + \frac{h}{2}(\bar{p} - \bar{p}') \\
\bar{p} - \bar{p}' - \frac{h}{2}(b(\bar{q} + \frac{h}{2}\bar{p}) - b(\bar{q}' + \frac{h}{2}\bar{p}'))
\end{pmatrix}
\bigg\|_M^2\\
&\quad= 
\bigg\|\begin{pmatrix}
I_d & 0\\
0 & \eta_0I_d
\end{pmatrix}\begin{pmatrix}
\bar{q} - \bar{q}'\\
\bar{p} - \bar{p}'
\end{pmatrix}
\bigg\|_M ^2
-
\begin{pmatrix}
\bar{q} - \bar{q}'\\
\bar{p} - \bar{p}'
\end{pmatrix}^{\!\!\top\!\!}
\bar{\mathcal{H}}
\begin{pmatrix}
\bar{q} - \bar{q}'\\
\bar{p} - \bar{p}'
\end{pmatrix},
\end{align*}
where~$\bar{\mathcal{H}}$ is defined by~\eqref{ch2} with~\eqref{barH}. Let the square matrices~$\bar{A}$,~$\bar{B}$ and~$\bar{C}$ be given by~\eqref{habc} with~$\bar{A},\bar{B},\bar{C},\bar{\mathcal{H}}$ replacing~$A,B,C,\mathcal{H}$. Moreover, let~$b$ be given by~\eqref{bdef}. As before in the proof of Proposition~\ref{lab}, the matrices~$\bar{A}$,~$\bar{B}$ and~$\bar{C}$ satisfy~\eqref{ABCb}. For~$\bar{q},\bar{q}',\bar{p},\bar{p}',h$ satisfying~$\abs{\bar{p} - \bar{p}'}^2\geq (17/4)L\abs{\bar{q} - \bar{q}}^2$ and~$Lh^2\leq (\eta_0 - \eta_1)^2$, inequality~\eqref{pgq} holds. Inequality~\eqref{pgq} implies~\eqref{hc} with~$\bar{\mathcal{H}},\bar{\mathcal{H}}'$ replacing~$\mathcal{H},\mathcal{H}'$, where~$\bar{\mathcal{H}}'$ is given by~\eqref{habc} with~$\bar{A}',\bar{B},\bar{C}',\bar{\mathcal{H}}'$ replacing~$A,B,C,\mathcal{H}$ and~$\bar{A}',\bar{C}'$ are given by
\begin{align*}
\bar{A}' &= \frac{17}{16}bhL(\eta_0 + \eta_1)I_d + bh\eta_1\bar{H} - \frac{b^2h^2\eta_1^2}{2}\bar{H}^2,\\
\bar{C}' &= \frac{3}{4}\cdot2b^2(\eta_0^2 - \eta_1^2)I_d -\frac{h^2}{4} I_d - bh\eta_1 I_d + \frac{bh^3\eta_1}{4}\bar{H} + b^2h^2\eta_1^2\bar{H} - \frac{b^2h^4\eta_1^2}{8}\bar{H}^2.
\end{align*}
We proceed with a similar argument as in the proof of Proposition~\ref{lba2} to show that~$\bar{\mathcal{H}}'$ is positive definite. For any eigenvalue~$-L\leq \bar{\lambda} \leq L$ of~$\bar{H}$, the matrices~$\bar{A}',\bar{C}'$ admit the same corresponding eigenspace with respective eigenvalues~$\bar{A}_{\bar{\lambda}}',\bar{C}_{\bar{\lambda}}'$ satisfying
\begin{align}
\bar{A}_{\bar{\lambda}}' &\geq bhL\bigg(\frac{17}{16}\eta_0 + \frac{1}{16}\eta_1 - \frac{\eta_1^2}{4}(\eta_0 - \eta_1)\bigg) \geq \frac{9}{16}bhL(\eta_0 + \eta_1),\label{ablb}\\
\bar{C}_{\bar{\lambda}}' &\geq \frac{3b^2(\eta_0^2 - \eta_1^2)}{2} - \frac{bh}{2}(\eta_0 + \eta_1) - \frac{bh^3L\eta_1}{4} - \frac{h^2\eta_1^2}{4} - \frac{h^4L\eta_1^2}{32}\nonumber\\
&\geq b^2(\eta_0^2 - \eta_1^2)\bigg(\frac{3}{2} - 1 - \frac{1}{4\cdot16^2} - \frac{1}{4} - \frac{1}{2\cdot16^3}\bigg),\label{cblb}
\end{align}
where the definition~\eqref{bdef} for~$b$ and our assumption on~$h$ have been repeatedly used, so that~$\bar{A}_{\bar{\lambda}}',\bar{C}_{\bar{\lambda}}'>0$. 
Moreover, the matrix~$\bar{A}'\bar{C}'-\bar{B}^2$ admits the same eigenspace with eigenvalue~$\bar{\mu}'$, which by~\eqref{ablb} and~\eqref{cblb} satisfies
\begin{equation}\label{mpb}
\bar{\mu}' \geq \frac{9}{16}\cdot\frac{31}{125}Lb^3h(\eta_0 + \eta_1)(\eta_0^2 - \eta_1^2) - \bigg( \bar{\lambda} b^2h\eta_1^2 + \frac{\bar{\lambda} bh^2\eta_1}{2} - \frac{\bar{\lambda}^2 b^2 h^3\eta_1^2}{4}\bigg)^2.
\end{equation}
For the last term on the right-hand side of~\eqref{mpb}, using~\eqref{bdef},~$-L\leq \bar{\lambda}\leq L$ and the assumptions~$\sqrt{L}h\leq \min(\eta_0 - \eta_1,1/16)$,~$0\leq\eta_1<\eta_0\leq 1$, 
it holds that
\begin{align}
&- \bigg( \bar{\lambda}b^2 h\eta_1^2 + \frac{\bar{\lambda} bh^2\eta_1}{2} - \frac{\bar{\lambda}^2 b^2 h^3\eta_1^2}{4}\bigg)^2\nonumber\\
&\quad= - \bar{\lambda}^2 b^4h^2 \eta_1^4  - \bar{\lambda}^2 b^3h^3\eta_1^3 + \frac{\bar{\lambda}^3b^4h^4\eta_1^4}{2} - \frac{\bar{\lambda}^2 b^2h^4\eta_1^2}{4} + \frac{\bar{\lambda}^3b^3h^5\eta_1^3}{4} - \frac{\bar{\lambda}^4b^4h^6\eta_1^4}{16}\nonumber\\
&\quad\geq -\frac{Lb^3h\eta_1^2}{2}(\eta_0-\eta_1)\bigg(1 + \frac{1}{16} + \frac{1}{4\cdot 16^2} + \frac{1}{2\cdot 16^2} + \frac{1}{4\cdot16^4} + \frac{1}{32\cdot 16^4}\bigg).\label{odh}
\end{align}
Applying the inequality~$\frac{\eta_1^2}{2} \leq \frac{1}{8}(\eta_0 + \eta_1)^2$ to the right-hand side of~\eqref{odh}, then inserting into~\eqref{mpb} concludes the proof.
\end{proof}

\begin{proof}[Proof of Proposition~\ref{lba2R}]
We may follow along the beginning of the proof of Proposition~\ref{lba2}, with initial differences already given in Proposition~\ref{lbaR} and its proof. In particular, in the proof of Proposition~\ref{lba2}, the definition of~$B$ should be replaced by~\eqref{BdefR} (replacing the definition of~$C$ with~\eqref{CdefR} only increases the value of~$C_{\lambda}'$). This takes into account the difference in the integrators considered. In addition, using instead~$\abs{\bar{q}-\bar{q}'}\geq 4\sqrt{L}\abs{\bar{q}-\bar{q}'}$, we have in place of~\eqref{pgq} that
\begin{equation*}
2b^2(\eta_0^2-\eta_1^2)\abs{\bar{p}-\bar{p}'}^2 \geq 16Lbh(\eta_0+\eta_1)\abs{\bar{q}-\bar{q}'}^2,
\end{equation*}
which implies
\begin{align*}
(\bar{p}-\bar{p}')^{\top}C(\bar{p}-\bar{p}')&\geq \bigg(\frac{3}{4}\cdot2b^2(\eta_0^2-\eta_1^2)-bh\bigg(\frac{mb^2\eta_0^3}{4}+\frac{1}{2}(\eta_0+\eta_1)\bigg)\bigg)\abs{\bar{p}-\bar{p}'}^2 \\
&\quad+4Lbh(\eta_0+\eta_1)\abs{\bar{q}-\bar{q}'}^2.
\end{align*}
Continuing with the structure of the proof of Proposition~\ref{lba2}, we have in place of~\eqref{alp} that
\begin{equation*}
A_{\lambda}' = 4b^2L(\eta_0^2-\eta_1^2) - c + \frac{h}{2}(h+2b\eta_1)\lambda - \frac{h^2}{4}\bigg(\frac{bh}{2}(\eta_0+\eta_1)+2b^2\eta_1^2\bigg)\lambda^2.
\end{equation*}
By definitions~\eqref{mdef},~\eqref{bdef},~$h+2b\eta_1 = 2b\eta_0$ and~$bh(\eta_0+\eta_1)=2b^2(\eta_0^2-\eta_1^2)$, this implies
\begin{align*}
A_{\lambda}'&\geq 4b^2L(\eta_0^2-\eta_1^2) - L\eta_0 b^2(\eta_0-\eta_1)/4 - 2Lb^2\eta_0(\eta_0-\eta_1)\\
&\quad- L^2b^4(\eta_0-\eta_1)^2(\eta_0^2+\eta_1^2),
\end{align*}
which, by~$\eta_1<\eta_0\leq1$,~$b\leq 1/(2\sqrt{L})$,~$\sqrt{L}h\leq \sqrt{2}/16$ and also therefore~$L^2b^4(\eta_0-\eta_1)= L^2b^3h/2 \leq Lb^2/32$, implies
\begin{align}
A_{\lambda}'&\geq 4b^2L(\eta_0+\eta_1)(\eta_0-\eta_1) - (1/4)L b^2\eta_0(\eta_0-\eta_1) - 2Lb^2\eta_0(\eta_0-\eta_1)\nonumber\\
&\quad- (1/32)Lb^2(\eta_0-\eta_1)(\eta_0+\eta_1)\nonumber\\
&\geq (17/10)Lb^2\eta_0(\eta_0-\eta_1) + 3.9Lb^2\eta_1(\eta_0-\eta_1).\label{asn}
\end{align}
On the other hand, it holds by~$bh\leq (1/2)(\eta_0-\eta_1)$ and~$\eta_0^3\leq \eta_0+\eta_1$ that
\begin{align}
C_{\lambda}' &= \frac{3}{4}\cdot 2b^2(\eta_0^2-\eta_1^2)-bh\bigg(\frac{mb^2\eta_0^3}{4} + \frac{1}{2}(\eta_0+\eta_1)\bigg)\nonumber\\
&\geq 2b^2(\eta_0^2-\eta_1^2)\bigg(\frac{3}{4} - \frac{1}{16} - \frac{1}{2}\bigg).\label{csn}
\end{align}
Moreover, using~$\lambda\geq-L$,~$\bar{u}\leq h$ and again~\eqref{bdef},~$\sqrt{L}h\leq \sqrt{2}/16$, equation~\eqref{hqo} implies
\begin{align}
\abs{B_{\lambda}} &\leq (1/4)Lb^3\eta_0(\eta_0-\eta_1) + (1/2)Lhb^2\eta_0(\eta_0-\eta_1) + Lb^3(\eta_0^2-\eta_1^2) \nonumber \\
&\quad+ (\sqrt{2}/16)\sqrt{L}b^2\eta_0(\eta_0-\eta_1)(1 + 2/16^2) + (h^2/2)L^2b^3(\eta_0^2-\eta_1^2),\label{nnj}
\end{align}
where the second term on the right-hand side of~\eqref{nnj} may be bounded as
\begin{equation*}
(1/2)Lhb^2\eta_0(\eta_0-\eta_1) \leq (1/16)\sqrt{L}b^2\eta_0(\eta_0-\eta_1)
\end{equation*}
and the last term on the right-hand side of~\eqref{nnj} may be bounded as
\begin{equation*}
(h^2/2)L^2b^3(\eta_0^2-\eta_1^2) \leq h^2L^2b^3\eta_0(\eta_0-\eta_1) \leq (1/16)Lb^3\eta_0(\eta_0-\eta_1).
\end{equation*}
Therefore it holds that
\begin{align*}
\abs{B_{\lambda}}&\leq (1/4+1+1/16)Lb^3\eta_0(\eta_0-\eta_1) + Lb^3\eta_1(\eta_0-\eta_1)\\
&\quad+ (1/16+\sqrt{2}/16+1/8^3)\sqrt{L}b^2\eta_0(\eta_0-\eta_1)\\
&\leq (21/16)Lb^3\eta_0(\eta_0-\eta_1) + (3/16)\sqrt{L}b^2\eta_0(\eta_0-\eta_1)+ Lb^3\eta_1(\eta_0-\eta_1).
\end{align*}
Consequently, it holds that
\begin{align*}
B_{\lambda}^2&\leq Lb^4(\eta_0-\eta_1)^2((21/16)\sqrt{L}b\eta_0 + (3/16)\eta_0 + \eta_1)^2\\
&\leq Lb^4(\eta_0-\eta_1)^2(0.72\eta_0^2 + 1.7\eta_0\eta_1 + \eta_1^2),
\end{align*}
which yields~$A_{\lambda}C_{\lambda}-B_{\lambda}^2>0$ together with~\eqref{asn} and~\eqref{csn}.
\end{proof}

\begin{proof}[Proof of Proposition~\ref{expo}]
Let~$b$ be given by~\eqref{bdef},~$H$ be given by~\eqref{mvth} and let~$\tilde{A},\tilde{B},\tilde{C}$ be given by~\eqref{abcdefs} with~$\tilde{A},\tilde{B},\tilde{C}$ replacing~$A,B,C$ and~$-\bar{c}$ replacing~$c$. It suffices to show that the matrices~$\tilde{A},\tilde{B},\tilde{C}$ are such that~$\tilde{A}$ and~$\tilde{A}\tilde{C}-\tilde{B}^2$ are positive definite. 
In the following, the assumptions on~$h,\eta_0,\eta_1$ and the definition for~$b$ are used without mention. Recall the alternate expressions~\eqref{al},~\eqref{cl} and~\eqref{bl} for the associated eigenvalues. For any eigenvalue~$-L\leq \lambda\leq L$ of~$H$, the matrix~$\tilde{A}$ admits the same corresponding eigenspace with eigenvalue~$\tilde{A}_{\lambda}$ given by
\begin{equation*}
\tilde{A}_{\lambda} = bh\bigg(6L\eta_0 + \lambda\eta_0 - \frac{h^2}{8}\lambda^2(\eta_0+\eta_1) - \lambda^2b^2\eta_1^2(\eta_0-\eta_1)\bigg),
\end{equation*}
which satisfies
\begin{equation}\label{abR}
\tilde{A}_{\lambda} \geq bhL\eta_0\bigg(6 - 1-\frac{1}{16}\bigg) - \frac{Lbh}{4}(\eta_0-\eta_1) \geq \frac{9}{2}Lbh\eta_0 + \frac{1}{4}L bh\eta_1,
\end{equation}
so that~$\tilde{A}_{\lambda}$ is positive. Moreover, the matrix~$\tilde{A}\tilde{C}-\tilde{B}^2$ admits the same eigenspace with eigenvalue~$\tilde{\mu}$ given by~\eqref{ml} with~$\tilde{\mu},\tilde{A}_{\lambda},\tilde{B}_{\lambda},\tilde{C}_{\lambda}$ replacing~$\mu,A_{\lambda},B_{\lambda},C_{\lambda}$ and~$\tilde{B}_{\lambda},\tilde{C}_{\lambda}$ satisfying
\begin{align*}
\tilde{B}_{\lambda} &= 6Lb^2h\eta_0^2 + \frac{bh^2\lambda}{4}(\eta_0+\eta_1) + b^2h\eta_1^2\lambda\\
\tilde{C}_{\lambda} &= 2b^2(\eta_0^2 - \eta_1^2) + bh\bigg(12Lb^2\eta_0^3 - \frac{1}{2}(\eta_0+\eta_1)\bigg)> b^2(\eta_0^2 - \eta_1^2)  >0.
\end{align*}
Let~$\tilde{\mu}_2$ be defined by
\begin{equation*}
\tilde{\mu}_2 := bh\bigg(\frac{9}{2}L\eta_0 + \frac{1}{4}L\eta_1\bigg)b^2(\eta_0^2 - \eta_1^2)  - 36L^2b^4h^2\eta_0^4 - 3b^2h^2\eta_0^2\eta_1^2\lambda - b^4h^2\eta_1^4\lambda^2,
\end{equation*}
which may be bounded as
\begin{align}
\tilde{\mu}_2 &\geq \frac{9Lbh\eta_0}{2}\bigg(\frac{1}{2}bh(\eta_0+\eta_1) \bigg) - Lb^2h^2\Big(8\eta_0^4 + 3\eta_0^2\eta_1^2 \Big) - \frac{1}{4}Lb^2h^2\eta_1^4\nonumber\\
&\geq \frac{9}{4}Lb^2h^2\eta_0 \bigg(\eta_0+\eta_1 \bigg).\label{cwq}
\end{align}
Let~$\tilde{\mu}_3$ be defined by
\begin{equation*}
\tilde{\mu}_3:= \frac{bh^2\lambda}{2}(\eta_0+\eta_1)(6Lb^2h\eta_0^2 + b^2h\eta_1^2\lambda),
\end{equation*}
which satisfies
\begin{equation}\label{cwq2}
\tilde{\mu}_3 \geq - 3L^2b^3h^3\eta_0^2(\eta_0+\eta_1) \geq -\frac{3}{2}Lb^2h^2\eta_0^2(\eta_0^2 - \eta_1^2)\geq -\frac{3}{2}Lb^2h^2\eta_0^4.
\end{equation}
Lastly, let~$\tilde{\mu}_4$ be defined by
\begin{equation*}
\tilde{\mu}_4:= -\frac{b^2h^4\lambda^2}{16}(\eta_0+\eta_1)^2,
\end{equation*}
which satisfies
\begin{equation}\label{cwq3}
\tilde{\mu}_4 \geq -\frac{b^2h^2L}{16}(\eta_0^2-\eta_1^2)^2 \geq -\frac{b^2h^2L}{16}\eta_0^4.
\end{equation}
Gathering~\eqref{cwq},~\eqref{cwq2} and~\eqref{cwq3} yields~$\tilde{\mu}\geq \tilde{\mu}_2+\tilde{\mu}_3+\tilde{\mu}_4 >0$.
\end{proof}

\begin{proof}[Proof of Proposition~\ref{expo2}]
Let~$b$ be given by~\eqref{bdef},~$\bar{H}$ be given by~\eqref{barH} and let~$\bar{A},\bar{B},\bar{C}$ be given by~\eqref{ABCb}. Moreover, let
\begin{equation*}
\hat{A} = \bar{c}I_d + \bar{A},\qquad\hat{B} = \bar{c}b\eta_0I_d +  \bar{B},\qquad
\hat{C} = 2\bar{c}b^2\eta_0^2I_d + \bar{C}.
\end{equation*}
It suffices to show that~$\hat{A}$ and~$\hat{A}\hat{C}-\hat{B}^2$ are positive definite. Again, in the following, the assumptions on~$h,\eta_0,\eta_1$ and the definition~$b$ are used without mention. In particular, the equation
\begin{equation*}
\bar{c}=4Lbh\eta_0
\end{equation*}
will be used. For any eigenvalue~$-L\leq \bar{\lambda} \leq L$ of~$\bar{H}$, the matrix~$\hat{A}$ admits the same corresponding eigenspace with eigenvalue~$\hat{A}_{\bar{\lambda}}$ satisfying
\begin{equation*}
\hat{A}_{\bar{\lambda}} = \bar{c} + bh\eta_1\bar{\lambda} - \frac{1}{2}b^2h^2\eta_1^2\bar{\lambda}^2 \geq 3Lbh\eta_0 - Lbh(\eta_0 - \eta_1) - \frac{1}{4}Lbh\eta_1^2(\eta_0-\eta_1) >0.
\end{equation*}
Moreover, the matrix~$\hat{A}\hat{C} - \hat{B}^2$ admits the same eigenspace with eigenvalue~$\hat{\mu}$ given by
\begin{align}
\hat{\mu} &= \bigg(\bar{c} + bh\eta_1\bar{\lambda} - \frac{1}{2}b^2h^2\eta_1^2\bar{\lambda}^2\bigg) \bigg(2\bar{c}b^2 \eta_0^2 + 2b^2(\eta_0^2 - \eta_1^2) - \frac{h^2}{4} -bh\eta_1 + \frac{bh^3\bar{\lambda}\eta_1}{4}  \nonumber\\
&\quad+ b^2h^2\bar{\lambda}\eta_1^2 - \frac{b^2h^4\bar{\lambda}^2\eta_1^2}{8}\bigg) 
-\bigg(\bar{c}b\eta_0 + \frac{bh^2\bar{\lambda}\eta_1}{2} + b^2h\bar{\lambda}\eta_1^2\bigg(1-\frac{h^2\bar{\lambda}}{4}\bigg)\bigg)^2.\label{mx}
\end{align}
Let~$\hat{\mu}_2$ be defined by
\begin{align*}
\hat{\mu}_2&:= 3Lbh\eta_0\bigg(2\bar{c}b^2\eta_0^2 + 2b^2(\eta_0^2-\eta_1^2) -bh\eta_1\bigg) - \bar{c}^2b^2\eta_0^2 - 2\bar{c}b^3\eta_0h\bar{\lambda}\eta_1^2\\
&\quad- b^4h^2\bar{\lambda}^2 \eta_1^4,
\end{align*}
which satisfies
\begin{align}
\hat{\mu}_2 &\geq 3Lbh\eta_0\bigg(8Lb^3h\eta_0^3 + 2b^2(\eta_0^2 - \eta_1^2) - bh\eta_1\bigg) - 16L^2b^4h^2\eta_0^4 - 8Lb^4h^2\eta_0^2\eta_1^2\nonumber\\
&\quad- L^2b^4h^2\eta_1^4\nonumber\\
&\geq 8L^2b^4h^2\eta_0^2(\eta_0^2-\eta_1^2) + 6Lb^3h\eta_0(\eta_0^2-\eta_1^2) - 3Lb^2h^2\eta_0\eta_1 - L^2b^4h^2\eta_1^4\nonumber\\
&= 8L^2b^4h^2\eta_0^2(\eta_0^2-\eta_1^2) 
+ 6Lb^3h(\eta_0(\eta_0^2-\eta_1^2) -  \eta_0\eta_1(\eta_0-\eta_1)) - L^2b^4h^2\eta_1^4.\nonumber\\
&\geq 8L^2b^4h^2\eta_0^2(\eta_0^2-\eta_1^2) 
+ \frac{11}{2}Lb^3h\eta_0^2(\eta_0-\eta_1).\label{mx2}
\end{align}
Let~$\hat{\mu}_3$ be defined by
\begin{align*}
\hat{\mu}_3 &:= 3Lbh\eta_0\bigg(-\frac{h^2}{4} + b^2h^2\bar{\lambda}\eta_1^2\bigg) - Lbh(\eta_0-\eta_1)\bigg(1+\frac{\eta_1^2}{4}\bigg)(8Lb^3h\eta_0^3 \\
&\quad+ 2b^2(\eta_0^2-\eta_1^2) -bh\eta_1) - bh^2\bar{\lambda}\eta_1(4Lb^2h\eta_0^2 + b^2h\bar{\lambda}\eta_1^2),
\end{align*}
which satisfies
\begin{align}
\hat{\mu}_3 &\geq -\frac{3}{4}Lbh^3\eta_0(1+\eta_1^2) - Lbh(\eta_0-\eta_1)\bigg(1+\frac{\eta_1^2}{4}\bigg) \bigg[8Lb^3h\eta_0^3 + 2b^2(\eta_0^2 - \eta_1^2)\nonumber\\
&\quad - 2b^2\eta_1(\eta_0-\eta_1)\bigg] - \bigg[8L^2b^4h^2\eta_0^2\eta_1(\eta_0-\eta_1) + \frac{bh^3L\eta_1^3}{4}\bigg]\label{ltb}
\end{align}
The first terms in the each of the square brackets on the right-hand side of~\eqref{ltb} may be bounded as
\begin{align}
&- Lbh(\eta_0-\eta_1)\bigg(1+\frac{\eta_1^2}{4}\bigg)\cdot8Lb^3h\eta_0^3 - 8L^2b^4h^2\eta_0^2\eta_1(\eta_0-\eta_1)\nonumber\\
&\quad= -8L^2b^4h^2\eta_0^2(\eta_0^2-\eta_1^2) - 2L^2b^4h^2\eta_0^3\eta_1^2(\eta_0-\eta_1)\nonumber\\
&\quad\geq -8L^2b^4h^2\eta_0^2(\eta_0^2-\eta_1^2) - \frac{1}{2}L^2b^3h\eta_1^2(\eta_0-\eta_1).\label{ltb1}
\end{align}
The first term on the right-hand side of~\eqref{ltb} may be bounded as
\begin{align}
-\frac{3}{4}Lbh^3\eta_0(1+\eta_1^2) &= -3(1+\eta_1^2)Lb^3h\eta_0(\eta_0-\eta_1)^2\nonumber\\
&\geq -3Lb^3h\eta_0^2(\eta_0-\eta_1) + \frac{3}{2}Lb^3h\eta_0\eta_1(\eta_0-\eta_1).\label{ltb2}
\end{align}
The second and third terms in the first square bracket on the right-hand side of~\eqref{ltb} may be bounded as
\begin{align}
&- Lbh(\eta_0-\eta_1)\bigg(1+\frac{\eta_1^2}{4}\bigg)\cdot 2b^2\eta_0(\eta_0-\eta_1) \nonumber\\
&\quad\geq -2Lb^3h\eta_0^2(\eta_0-\eta_1) + Lb^3h\eta_0\eta_1(\eta_0-\eta_1).\label{ltb3}
\end{align}
Lastly, the last term in the second square bracket on the right-hand side of~\eqref{ltb} may be bounded as
\begin{equation}\label{ltb4}
-\frac{bh^3L\eta_1^3}{4} \geq -Lb^3h\eta_1^3(\eta_0-\eta_1)^2 \geq -\frac{1}{2}Lb^3h\eta_1^2(\eta_0-\eta_1).
\end{equation}
Gathering~\eqref{ltb1},~\eqref{ltb2},~\eqref{ltb3},~\eqref{ltb4} and inserting into~\eqref{ltb} yields
\begin{equation}\label{mx3}
\hat{\mu}_3\geq -8L^2b^4h^2\eta_0^2(\eta_0^2-\eta_1^2) - 5Lb^3h\eta_0^2(\eta_0-\eta_1) + \frac{3}{2}Lb^3h\eta_0\eta_1(\eta_0-\eta_1).
\end{equation}
Finally, let~$\hat{\mu}_{4+}$ be given by
\begin{equation*}
\hat{\mu}_{4+} := \frac{\bar{c}bh^3\bar{\lambda}\eta_1}{4} - \frac{\bar{c}b^2h^4\bar{\lambda}^2\eta_1^2}{8} + cb\eta_0\cdot\frac{b^2h^3\bar{\lambda}^2\eta_1^2}{2} + \bar{\mu}_{4+},
\end{equation*}
where~$\bar{\mu}_{4+}$ is given by~\eqref{m4}. By the same observations as those directly following~\eqref{m4}, equation~\eqref{m4s} holds. Therefore, it holds that
\begin{align}
\hat{\mu}_{4+} &\geq  
-\frac{\bar{c}bh^3L\eta_1}{4} 
- \frac{\bar{c}b^3h^3\bar{\lambda}^2\eta_1^2}{4}(\eta_0-\eta_1) + \frac{\bar{c}b^3h^3\bar{\lambda}^2\eta_1^2}{2}\eta_0 + \frac{\bar{\lambda}^2h^4\eta_1^2}{32L}\nonumber\\
&\geq -L^2b^2h^4\eta_0\eta_1 \nonumber\\
&\geq -2L^2b^3h^3\eta_0\eta_1(\eta_0-\eta_1)\nonumber\\
&\geq -\frac{1}{2}Lb^3h\eta_0\eta_1(\eta_0-\eta_1).\label{mx4}
\end{align}
Gathering~\eqref{mx2},~\eqref{mx3},~\eqref{mx4} and inserting into equation~\eqref{mx} yields~$\hat{\mu}\geq \hat{\mu}_2+\hat{\mu}_3+\hat{\mu}_{4+} >0$.
\end{proof}

\begin{proof}[Proof of Proposition~\ref{RvRth}]
Let~$\mathcal{H}$ be given by~\eqref{chR} with~$c$ replaced by~$-\bar{c}$. It suffices to show that~$\mathcal{H}$ is positive definite for any~$\bar{q},\bar{q}',\bar{p},\bar{p}'$. The matrix~$\mathcal{H}$ is positive definite if and only if the square matrices~$A,B,C$ given by~\eqref{habc} 
are such that~$A$ and~$AC-B^2$ are positive definite. Denoting again~\eqref{bdef},
it holds by direct calculation that~$A,B,C$ are given by~\eqref{Adef} and~\eqref{abcdefsR} all with~$c$ replaced by~$-\bar{c}$. For any eigenvalue~$-L\leq \lambda\leq L$ of~$H$, the matrices~$A,B,C$ admit the same corresponding eigenspace with eigenvalues~$A_{\lambda},B_{\lambda},C_{\lambda}$. By the same derivations,~$A_{\lambda}$ satisfies~\eqref{abR} with~$\bar{A}_{\lambda}$ replaced by~$A_{\lambda}$, which implies
\begin{equation}\label{aiR}
A_{\lambda} \geq ((9/2)\eta_0 +(1/4)\eta_1))Lbh \geq 9Lb^2\eta_0 (\eta_0-\eta_1).
\end{equation}
From the definition~\eqref{abcdefsR} of~$B$ (with~$c$ replaced by~$-\bar{c}$ and similar to~\eqref{hqo}) and definition~\eqref{cbdef} of~$\bar{c}$, it holds by~$\bar{u}\leq h$ that
\begin{align*}
B_{\lambda} &\leq 12Lb^3\eta_0^2(\eta_0-\eta_1) + Lb^3 \eta_0 (\eta_0-\eta_1)^2 + Lb^3(\eta_0^2-\eta_1^2) \\
&\quad+ 2Lb^3\eta_0(\eta_0-\eta_1)^2(1+Lb^2(\eta_0-\eta_1)^2) + 2L^2b^5(\eta_0-\eta_1)^2(\eta_0^2-\eta_1^2).
\end{align*}
Using~$b=h/(2(\eta_0-\eta_1))\leq 1/(2\sqrt{L})$ and~$Lh^2\leq 1/16^2$, this implies
\begin{align*}
B_{\lambda} &\leq 12Lb^3\eta_0^2(\eta_0-\eta_1) + (1/32)\sqrt{L}b^2 \eta_0 (\eta_0-\eta_1) + (1/2)\sqrt{L}b^2(\eta_0^2-\eta_1^2) \\
&\quad+ (1/16)\sqrt{L}b^2\eta_0(\eta_0-\eta_1)(1+1/64) + (1/(4\cdot16^2))\sqrt{L}b^2(\eta_0^2-\eta_1^2).
\end{align*}
Consequently, it holds that
\begin{equation*}
B_{\lambda} \leq 12Lb^3\eta_0^2(\eta_0-\eta_1) + \sqrt{L}b^2(\eta_0^2-\eta_1^2),
\end{equation*}
which implies by Young's inequality that
\begin{align}\label{blR}
B_{\lambda}^2 &\leq (3/2)12^2L^2b^6\eta_0^4(\eta_0-\eta_1)^2 + 3Lb^4(\eta_0^2-\eta_1^2)^2\nonumber\\
&\leq (3/2)12^2L^2b^6\eta_0^4(\eta_0-\eta_1)^2 + 6 Lb^4\eta_0(\eta_0+\eta_1)(\eta_0-\eta_1)^2
\end{align}
On the other hand, from the definition~\eqref{abcdefsR} of~$C$ (with~$c$ replaced by~$-\bar{c}$), it holds that
\begin{equation*}
C_{\lambda} \geq 2b^2\eta_0^2 + 6Lb^2h^2\eta_0^3/(\eta_0-\eta_1) - b^2(\eta_0^2+\eta_1^2) = b^2(\eta_0^2 - \eta_1^2) + 24Lb^4\eta_0^3(\eta_0-\eta_1).
\end{equation*}
Therefore, together with~\eqref{aiR}, implies
\begin{equation*}
A_{\lambda}C_{\lambda} \geq 9Lb^4\eta_0(\eta_0+\eta_1)(\eta_0-\eta_1)^2 + 216L^2b^6\eta_0^4(\eta_0-\eta_1)^2,
\end{equation*}
which shows~$A_{\lambda}C_{\lambda}-B_{\lambda}^2$ together with~\eqref{blR}.
\end{proof}

\section{Proofs for biases}\label{spw}
Firstly, similar to Lemma~30 in~\cite{MR4309974} and its proof, an estimate on the second moment of~$e^{-U(x)}dx$ is presented without assuming convexity.
\begin{lemma}\label{eal}
Under Assumption~\ref{A1} with~$b(\cdot,\theta)=\nabla U$ for some~$U:\mathbb{R}^d\rightarrow\mathbb{R}$ and~$\theta\in\Theta$, it holds for~$\beta\in\{2,4\}$ that
\begin{equation*}
\int_{\mathbb{R}^d} \abs{x}^{\beta}\mu_1(dx) \leq \bigg(\frac{3L(R')^2+2d + 2\beta - 4}{m}\bigg)^{\frac{\beta}{2}},
\end{equation*}
where~$\mu_1(dx)$ is given by~$e^{-U(x)}dx/\int_{\mathbb{R}^d}e^{-U(x)}dx$ and~$R' = 4R(1+\frac{L}{m})$.
\end{lemma}
\begin{proof}
Let~$\mathcal{L}$ denote the differential operator~$-\nabla U\cdot\nabla\!_x + \Delta_x$. For any~$\bar{x}\in\mathbb{R}^d$, inequality~\eqref{A1b} with~$x=0$,~$y=\bar{x}$ and~$u=\bar{x}$ implies for any~$\zeta\geq 1$ that
\begin{align*}
\frac{1}{2}\mathcal{L}\abs{\bar{x}}^{2\zeta} &= -2\zeta\abs{\bar{x}}^{2(\zeta-1)}\nabla U(\bar{x})\cdot \bar{x} + \abs{\bar{x}}^{2(\zeta-1)}(2\zeta d + 4\zeta(\zeta-1))\\
&= - 2\zeta\abs{\bar{x}}^{2(\zeta-1)}\int_0^1 \bar{x}^{\top} \Delta U(\lambda \bar{x}) \bar{x} d\lambda + \abs{\bar{x}}^{2(\zeta-1)}(2\zeta d + 4\zeta(\zeta-1))\\
&\leq 2\zeta\abs{\bar{x}}^{2(\zeta-1)}\bigg(-\frac{m}{2}\abs{\bar{x}}^2 + \frac{3}{2}L(R')^2\bigg) + \abs{\bar{x}}^{2(\zeta-1)}(2\zeta d + 4\zeta(\zeta-1)).
\end{align*}
Integrating against~$\mu_1$ and using Jensen's inequality  
yields
\begin{equation*}
\int_{\mathbb{R}^d}\abs{\bar{x}}^{2\zeta}\mu_1(d\bar{x}) 
\leq \frac{2}{m}\bigg(\frac{3}{2}L(R')^2 + d + 2(\zeta-1)\bigg)\bigg(\int_{\mathbb{R}^d} \abs{\bar{x}}^{2\zeta} \mu_1(d\bar{x})\bigg)^{\frac{\zeta-1}{\zeta}},
\end{equation*}
which implies the assertion.
\end{proof}

\begin{proof}[Proof of Corollary~\ref{invd}]
The proof follows along the same lines as the proof of Proposition~16 in~\cite{gouraud2023hmc}, except we deal here with the twisted Euclidean metric~$\hat{d}$ given by~\eqref{tmet} and we consider our velocity Verlet integrator to possibly be replaced by a randomized midpoint version. We first consider the case where the velocity Verlet integrator~\eqref{vvi} is used, then at the end of the proof, we point out where the argument should be altered for the randomized midpoint version. 
We assume~$L=1$ w.l.o.g., then parameters are rescaled at the end so that~$L$ reappears in the assertion. 
For~$t\in[0,T]$, let~$P_t,P_R$ be the transition operators given by~$P_t f(x,v) = f(\bar{q}_t(x,v,\bar{\theta}),\bar{p}_t(x,v,\bar{\theta}))$ and~$P_R f(x,v) = \mathbb{E}f(x,\eta v + \sqrt{1-\eta^2}G)$
for all~$x,v\in\mathbb{R}^d$ and bounded measurable~$f$, but with~$b(\cdot,\theta)$ replaced by~$\nabla U$ for any~$\theta\in\Theta$ in the definitions of~$\bar{q}_{\cdot},\bar{p}_{\cdot}$.
By Lemma~\ref{eal} and Jensen's inequality, it holds that
\begin{align*}
\int_{\mathbb{R}^{2d}} (\abs{x}^2 + \abs{v}^2) d\mu(x,v) &\leq \frac{3(R')^2+2d}{m} + d,\\
\int_{\mathbb{R}^{2d}} (\abs{x}^2 + \abs{v}^2)^2 d\mu(x,v) &\leq \bigg(\frac{3(R')^2+2d + 4}{m} \bigg)^2 + 2d\bigg(\frac{3(R')^2+2d + 4}{m}\bigg)\\
&\quad+ 2d + d^2,
\end{align*}
which imply for~$u=0$ (velocity Verlet), in the same way as for~(39) 
and~(40) both 
in~\cite{gouraud2023hmc}, that
\begin{align}
\mathcal{W}_{2,e}(\mu P_h,\mu) &\leq 2h^2\bigg(\int_{\mathbb{R}^{2d}}(\abs{x}^2 + \abs{v}^2) d\mu(x,v)\bigg)^{\frac{1}{2}}\nonumber\\
&\leq 2h^2 \bigg(\frac{3(R')^2+2d}{m}  + d\bigg)^{\frac{1}{2}}\nonumber\\
&\leq 2\sqrt{6}h^2\max\bigg(d,\frac{d}{m},\frac{(R')^2}{m}\bigg)^{\frac{1}{2}}\label{wtu}
\end{align}
and if additionally~$\nabla^2 U$ is~$L_2$-Lipschitz, that
\begin{align}
&\mathcal{W}_{2,e}(\mu P_h,\mu)\nonumber\\
&\quad\leq  2h^3(1+L_2) \bigg(\int_{\mathbb{R}^{2d}} (1 + 2\abs{x}^2 + 2\abs{v}^2 + (\abs{x}^2 + \abs{v}^2)^2) d\mu(x,v)\bigg)^{\frac{1}{2}}\nonumber\\
&\quad\leq 2h^3(1+L_2)\bigg(1 + 4d + \frac{6(R')^2+4d}{m} + \bigg(d+ \frac{3(R')^2+2d + 4}{m}\bigg)^2 \bigg)^{\frac{1}{2}}\nonumber\\
&\quad\leq 22h^3(1+L_2)\max\bigg(d,\frac{d}{m},\frac{(R')^2}{m}\bigg).\label{wtu2}
\end{align}
Following the same bounds as in the corresponding inequality in the proof of Proposition~16 in~\cite{gouraud2023hmc}, it holds that
\begin{equation}\label{qws}
\mathcal{W}_{2,e}(\pi_n(\mu),\mu) \leq \frac{(1+2h)^{nK}}{2h}\mathcal{W}_{2,e}(\mu P_h ,\mu).
\end{equation}
If instead~$u=1$ (randomized midpoint), then instead of~\eqref{wtu} and~\eqref{wtu2}, we may apply Lemma~7 in~\cite{bourabee2022unadjusted} for the same arguments as above. Note that the referenced lemma statement concerns only the position variables, but the proofs bound both the position and velocity variables at the same time, as required here. Consequently, without assuming that~$\nabla^2 U$ is Lipschitz and considering instead the full Hamiltonian trajectory, it holds for~$u=1$ that
\begin{equation}\label{wtu3}
\mathcal{W}_{2,e}(\mu P_T,\mu)\leq 71\sqrt{6}L^{1/4}h^{3/2}\max\bigg(d,\frac{d}{m},\frac{(R')^2}{m}\bigg).
\end{equation}
Moreover, using the estimates in the proof of Lemma~\ref{basic} to control distances after steps of the randomized midpoint integrator, it holds for~$u=1$  that
\begin{align*}
\mathcal{W}_{2,e}(\pi_n(\mu),\mu) &= \mathcal{W}_{2,e}(\pi_{n-1}(\mu)P_RP_T,\mu)\\
&\leq \mathcal{W}_{2,e}(\pi_{n-1}(\mu)P_RP_T,\mu P_T) + \mathcal{W}_{2,e}(\mu P_T,\mu)\\
&\leq (1+2T)\mathcal{W}_{2,e}(\pi_{n-1}(\mu)P_R,\mu ) + \mathcal{W}_{2,e}(\mu P_T,\mu)\\
&= (1+2T)\mathcal{W}_{2,e}(\pi_{n-1}(\mu),\mu ) + \mathcal{W}_{2,e}(\mu P_T,\mu).
\end{align*}
Repeating~$n-1$ more times and applying~\eqref{wtu3} yields
\begin{align}
\mathcal{W}_{2,e}(\pi_n(\mu),\mu) &\leq \sum_{k=0}^{n-1}(1+2T)^k \mathcal{W}_{2,e}(\mu P_T,\mu)\nonumber\\
&\leq \frac{(1+2T)^{n-1}}{T}36\sqrt{6}L^{1/4} h^{3/2}\max\bigg(\frac{d}{m},\frac{(R')^2}{m}\bigg).\label{fhw}
\end{align}
In both cases~\eqref{qws} and~\eqref{fhw}, by Corollary~\ref{maincore} and setting
\begin{equation}\label{itu}
n=\left\lceil\frac{\frac{1}{2}\log (2\log(2)\epsilon^*) - \frac{1}{2}\log g - \log(2)}{\log(1-c)}\right\rceil,
\end{equation}
where~$\ceil{k}=\min\{l\in\mathbb{N}:k\leq l\}$ here, it holds that
\begin{align*}
\mathcal{W}_1(\mu,\tilde{\mu}) &\leq \mathcal{W}_1(\pi_n(\mu),\tilde{\mu}) + \mathcal{W}_1(\pi_n(\mu),\mu) \\
&\leq \frac{1}{2}\mathcal{W}_1(\mu,\tilde{\mu}) + \sqrt{2}\max(1+1.09\alpha,\gamma^{-1})\mathcal{W}_{2,e}(\pi_n(\mu),\mu),
\end{align*}
which, by~$4LT^2\leq(1-\eta)^2$, implies
\begin{equation}\label{we1}
\mathcal{W}_1(\mu,\tilde{\mu}) \leq 
4\mathcal{W}_{2,e}(\pi_n(\mu),\mu).
\end{equation}
In case~$u=0$, by~\eqref{qws},~$\log(1-c)\leq-c$ and~\eqref{itu}, inequality~\eqref{we1} implies
\begin{align*}
\mathcal{W}_1(\mu,\tilde{\mu})&\leq 2e^{2nT}h^{-1}\mathcal{W}_{2,e}(\mu P_h,\mu)\\
&\leq 2\exp\bigg(2T\bigg(1-\frac{1}{2c}\log\bigg(\frac{\log(2)\epsilon^*}{2g}\bigg)\bigg)\bigg)h^{-1}\mathcal{W}_{2,e}(\mu P_h,\mu)\\
&= 2e^{2T}(2g/(\log(2)\epsilon^*))^{T/c}h^{-1}\mathcal{W}_{2,e}(\mu P_h,\mu)
\end{align*}
which shows the claim by~\eqref{carot} and either~\eqref{wtu} or~\eqref{wtu2}. 
In case~$u=1$, inserting~\eqref{itu} into~\eqref{fhw} then~\eqref{we1} and using again~$\log(1-c)\leq -c$ completes the proof.
\end{proof}

\begin{proof}[Proof of Proposition~\ref{invd2}]
Throughout the proof, assume~$L=1$ w.l.o.g. We only deal with the~$u=0$ case. The~$u=1$ case is similar and therefore omitted. 
Let~$x,v\sim\tilde{\mu}$ be 
independent of~$(G_i)_i,(\hat{G}_i)_i,(\bar{\theta}_i)_i,((u_{kh}^{(i)})_k)_i$ and 
for any~$i\in\mathbb{N}$,~$j\in\mathbb{N}\cap[0,T/h)$, 
let~$\theta_{i,jh},\theta_{i,jh}'\in\Theta$ 
satisfy~$\hat{\theta}_i = (\theta_{i,jh})_j$ and~$\hat{\theta}_i' = (\theta_{i,jh}')_j$. 
Moreover, let~$v'$ be given by~\eqref{ous},~$x_0=y_0=x$,~$v_0=w_0=v'$ and for any~$j\in\mathbb{N}\cap[0,T/h)$, let
\begin{align*}
v_{j+\frac{1}{2}} &= v_j - (h/2)b(x_j,\theta_{0,jh}), & w_{j+\frac{1}{2}} &= w_j - (h/2)\nabla U(y_j),\\
x_{j+1} &= x_j + hv_{j+\frac{1}{2}}, & y_{j+1} &= y_j + hw_{j+\frac{1}{2}},\\
v_{j+1} &= v_{j+\frac{1}{2}} - (h/2)b(x_{j+1},\theta_{0,jh}'), & 
w_{j+1} &= w_{j+\frac{1}{2}} - (h/2)\nabla U(y_{j+1}).
\end{align*}
By the tower property, for any~$j\in\mathbb{N}\cap[0,T/h)$ it holds that
\begin{align}
\mathbb{E}\abs{v_{j+\frac{1}{2}}-w_{j+\frac{1}{2}}}^2 &= \mathbb{E}[\abs{v_j-w_j - (h/2)(\nabla U(x_j)-\nabla U(y_j))}^2] \nonumber\\
&\quad+ (h/2)^2\mathbb{E}[\abs{\nabla U(x_j)-b(x_j,\theta_{0,jh})}^2].\label{kwb0}
\end{align}
The first term on the right-hand side of~\eqref{kwb0} can be bounded as
\begin{align*}
&\mathbb{E}[\abs{v_j-w_j - (h/2)(\nabla U(x_j)-\nabla U(y_j))}^2]\\
&\quad\leq (1+h/2)\mathbb{E}[\abs{v_j-w_j}^2] + (h/2)(1+h/2)\mathbb{E}[\abs{\nabla U(x_j)-\nabla U(y_j)}^2] \\
&\quad\leq (1+h/2)\mathbb{E}[\abs{v_j-w_j}^2] + (h/2)(1+h/2)\mathbb{E}[\abs{x_j-y_j}^2] 
\end{align*}
and the second term on the right-hand side of~\eqref{kwb0} can be bounded, by the tower property, the unbiasedness assumption and independence between~$\hat{\theta}_0$ and~$x,v,G_0$, as
\begin{align*}
&(h/2)^2\mathbb{E}[\abs{\nabla U(x_j)-b(x_j,\theta_{0,jh})}^2]\\
&\quad= \bigg(\frac{h}{2}\bigg)^2 \mathbb{E}\bigg[\bigg|\frac{1}{p}\sum_{i=1}^p (\nabla U(x_j)-\nabla U_{\vartheta_i(\theta_{0,jh})}(x_j))\bigg|^2\bigg]\\
&\quad= \frac{h^2}{4p^2}\sum_{i=1}^p\mathbb{E}[| \nabla U(x_j)-\nabla U_{\vartheta_i(\theta_{0,jh})}(x_j))|^2]\\
&\quad\leq C_U h^2/4p.
\end{align*}
Therefore,~\eqref{kwb0} implies
\begin{equation}\label{dkw}
\mathbb{E}[\abs{v_{j+\frac{1}{2}}-w_{j+\frac{1}{2}}}^2] \leq (1+h/2)(\mathbb{E}[\abs{v_j-w_j}^2] + (h/2)\mathbb{E}[\abs{x_j-y_j}^2]) + C_Uh^2/(4p).
\end{equation}
Consequently, the step in position has the bound
\begin{align}
\mathbb{E}[\abs{x_{j+1} - y_{j+1}}^2] &\leq (1+h)\mathbb{E}[\abs{x_j - y_j}]^2 + h(1+h)\mathbb{E}[\abs{v_{j+\frac{1}{2}} - w_{j+\frac{1}{2}}}^2]\nonumber\\
&\leq (1+3h/2)\mathbb{E}[\abs{x_j - y_j}^2] + (3h/2)\mathbb{E}[\abs{v_j-w_j}^2] \nonumber\\
&\quad + C_U h^3(1+h)/(4p).\label{dkw2}
\end{align}
Moreover, by the same steps as for~\eqref{dkw}, then applying~\eqref{dkw},~\eqref{dkw2} and using~\eqref{carot}, it holds for any~$j\in\mathbb{N}\cap[0,T/h)$ that
\begin{align}
&\mathbb{E}[\abs{v_{j+1}-w_{j+1}}^2] \nonumber\\
&\quad\leq (1+h/2)\mathbb{E}[\abs{v_{j+\frac{1}{2}}-w_{j+\frac{1}{2}}}^2] + (h/2)\mathbb{E}[\abs{x_{j+1}-y_{j+1}}^2]) + C_Uh^2/(4p)\nonumber\\
&\quad\leq (1+3h/2) \mathbb{E}[\abs{v_j - w_j}^2] + (3h/2)\mathbb{E}[\abs{x_j - y_j}^2] + 3C_Uh^2/(2p).\label{dkw3}
\end{align}
Inequalities~\eqref{dkw2} and~\eqref{dkw3} together with~\eqref{carot} 
imply
\begin{align}
&\mathbb{E}[\abs{x_{j+1} - y_{j+1}}^2 + \abs{v_{j+1}-w_{j+1}}^2] \nonumber\\
&\quad\leq (1+3h)\mathbb{E}[\abs{x_j - y_j}^2 + \abs{v_j-w_j}^2]
+ 2C_Uh^2/p.\label{dkws}
\end{align}
In addition, observe that by a synchronous coupling, the transition associated with the OU step~$v\mapsto v'$, denoted~$P_R$, reduces~$\mathcal{W}_{2,e}$ distance, in the sense that~$\mathcal{W}_{2,e}(\bar{\nu} P_R,\bar{\nu}' P_R)\leq \mathcal{W}_{2,e}(\bar{\nu},\bar{\nu}')$ for all probability measures~$\bar{\nu},\bar{\nu}'$. 
Therefore, using that~$\tilde{\mu}$ is invariant for the gHMC chain with nonstochastic gradient 
and repeating the computations leading to~\eqref{dkws}~$n$ times yields 
\begin{equation*}
\mathcal{W}_{2,e}(\pi_n(\tilde{\mu}),\tilde{\mu})^2
\leq \frac{2C_Uh^2}{p}\sum_{j=0}^{nK-1} (1+3h)^j  
\leq \frac{2C_Uh(1+3h)^{nK}}{3p}.
\end{equation*}
\end{proof}

\begin{proof}[Proof of Corollary~\ref{bicor_old}]
Let~$(y,w)\sim\tilde{\mu}$ be independent of~$(G_i)_i,(\hat{G}_i)_i,(\bar{\theta}_i)_i,((u_{kh}^{(i)})_k)_i$. 
By Kantorovich-Rubenstein duality, 
it holds that
\begin{align}
&\bigg|\mathbb{E}\bigg[\frac{1}{N}\sum_{i=0}^{N-1}\bigg(f(X^{(N_0+i)}) - \int fd\mu\bigg)\bigg]\bigg|\nonumber\\
&\quad\leq 
\frac{1}{N}\sum_{i=0}^{N-1} \|f\|_{\textrm{Lip}(\hat{d})} \mathcal{W}_1(\pi_{N_0+i}(\delta_{(x,v)}),\pi_{N_0+i}(\tilde{\mu})) \nonumber\\
&\qquad+ \frac{1}{N}\sum_{i=N_0}^{N_0+N-1}\|f\|_{\textrm{Lip}(\hat{d})}\mathcal{W}_1(\pi_i(\tilde{\mu}),\tilde{\mu}) + \|f\|_{\textrm{Lip}(\hat{d})}\mathcal{W}_1(\tilde{\mu},\mu).\label{skw}
\end{align}
By taking the infinimum of all couplings between~$X^{(i)}$ and~$Y^{(i)}$ for each~$i$, the first term on the right-hand side of~\eqref{skw} may be bounded, by Corollary~\ref{maincore}, as
\begin{align}
&\frac{\|f\|_{\textrm{Lip}(\hat{d})}}{N}\sum_{i=0}^{N-1} \mathcal{W}_1(\pi_{N_0+i}(\delta_{(x,v)}),\pi_{N_0+i}(\tilde{\mu}))\nonumber\\
&\quad\leq \frac{\|f\|_{\textrm{Lip}(\hat{d})}}{N}\sum_{i=0}^{N-1} \bigg(\frac{g}{2\log(2)\epsilon^*}\bigg)^{\frac{1}{2}}(1-c)^{N_0+i}\mathcal{W}_1(\delta_{(x,v)},\tilde{\mu})\nonumber\\
&\quad\leq \frac{g^{\frac{1}{2}}\|f\|_{\textrm{Lip}(\hat{d})}(1-c)^{N_0}}{(2\log(2))^{\frac{1}{2}}Nc(\epsilon^*)^{\frac{1}{2}}}[
\mathcal{W}_1(\delta_{(x,v)},\delta_{(0,0)}) + \mathcal{W}_1(\delta_{(0,0)},\mu) 
+ \mathcal{W}_1(\mu,\tilde{\mu})],\label{skw2}
\end{align}
where~$\delta_{(x,v)}$ denotes the Dirac delta at~$(x,v)$. 
For the second term in the square bracket on the right-hand side~\eqref{skw2}, it holds by Lemma~\ref{eal} and~$\alpha=4LT^2/(1-\eta)\leq 1$ that
\begin{equation}\label{skw3}
\mathcal{W}_1(\delta_{(0,0)},\mu) \leq \sqrt{2}(1+1.09\alpha)\mathcal{W}_{2,e}(\delta_{(0,0)},\mu) \leq 2((3L(R')^2 + 2d)/m + d/L)^{\frac{1}{2}}.
\end{equation}
Moreover, for the last sum on the right-hand side of~\eqref{skw}, it holds by Proposition~\ref{invd2} that
\begin{align}
\frac{1}{N}\sum_{i=N_0}^{N_0+N-1} \mathcal{W}_1(\pi_i(\tilde{\mu}),\tilde{\mu}) &\leq 2\bigg(\frac{C_U\sqrt{L} h}{p}\bigg)^{\frac{1}{2}}\cdot\frac{1}{N}\sum_{i=N_0}^{N_0+N-1} e^{\frac{3}{2}\sqrt{L} hK i}\nonumber\\
&\leq 2\bigg(\frac{C_U\sqrt{L} h}{p}\bigg)^{\frac{1}{2}}\cdot\frac{\exp(\frac{3}{2}\sqrt{L} T(N_0+N))}{\frac{3}{2}N\sqrt{L}T}.\label{skw4}
\end{align}
By gathering~\eqref{skw},~\eqref{skw2},~\eqref{skw3} and~\eqref{skw4}, then using Corollary~\ref{invd}, the proof concludes.
\end{proof}

\end{appendix}

\bibliography{document}

\end{document}